\documentclass[aos,preprint]{imsart}

\RequirePackage[OT1]{fontenc}
\RequirePackage{amsthm,amsmath, amssymb, mathrsfs, amsfonts}
\usepackage[numbers]{natbib}
\RequirePackage[colorlinks,citecolor=blue,urlcolor=blue]{hyperref}

\usepackage{bbm}

\arxiv{arXiv:1706.00850}

\startlocaldefs
\numberwithin{equation}{section}
\theoremstyle{plain}
\newtheorem{theorem}{Theorem}[section]

\newtheorem{corollary}[theorem]{Corollary}

\newtheorem{lemma}[theorem]{Lemma}
\newtheorem{proposition}[theorem]{Proposition}

\endlocaldefs

\def\E{\mathbb{E}}

\def\bt{{\bf{t}}}
\def\bx{{\bf{x}}}
\def\bnu{{\boldsymbol{\nu}}}
\def\bmu{{\boldsymbol{\mu}}}

\def\by{{\bf{y}}}

\def\FF{\mathcal F}
\def\WW{\mathcal W}
\def\HH{\mathcal H}
\def\EE{\mathcal E}

\def\XX{\mathcal X}

\def\GG{\mathcal G}

\def\RR{\mathcal R}
\def\KK{\mathcal K}

\def\NN{\mathcal N}

\def\R{\mathbb R}
\def\E{\mathbb E}

\def\Z{\mathbb Z}
\def\N{\mathbb N}
\def\P{\mathbb P}

\usepackage{xr}
\begin{document}

\begin{frontmatter}
\title{Minimax Optimal Rates of Estimation in Functional ANOVA Models With Derivatives}
\runtitle{Functional ANOVA Estimation With Derivatives}

\begin{aug}
\author{\fnms{Xiaowu} \snm{Dai}\thanksref{t1,t2}\ead[label=e1]{xdai26@wisc.edu}}
\and
\author{\fnms{Peter} \snm{Chien}\thanksref{t2}
\ead[label=e3]{peter.qian@wisc.edu}}

\thankstext{t1}{Supported in part by NSF Grant DMS-1308877.}
\thankstext{t2}{Supported in part by NSF Grant DMS-1564376.}
\runauthor{X. Dai and P. Chien}

\affiliation{University of Wisconsin-Madison}

\address{Department of Statistics\\
University of Wisconsin-Madison\\
1300 University Avenue\\
Madison, Wisconsin 53706\\
USA\\
\printead{e1}\\
\phantom{E-mail:\ }\printead*{e3}}
\end{aug}

\begin{abstract}
We establish minimax optimal rates of convergence for nonparametric estimation in functional ANOVA models when data from first-order partial derivatives are available. 
Our results reveal that
partial derivatives can improve convergence rates for function estimation with deterministic or random designs. 
In particular, for full $d$-interaction models, the optimal rates with first-order partial derivatives on 
$p$ covariates
are identical to those for $(d-p)$-interaction models without partial derivatives.  
For additive models, the rates by using all first-order partial derivatives are root-$n$ to achieve the ``parametric rate''.
We also investigate the minimax optimal rates for first-order partial derivative estimations when derivative data are available. Those rates coincide with the optimal rate for estimating the first-order derivative of a univariate function.
\end{abstract}



\end{frontmatter}


\section{Introduction}
\label{sec:Intro}

Derivative observations for complex systems are available in many applications. In dynamic systems and traffic engineerings, real-time motion sensors can record velocity,  acceleration in addition to positions \cite{murray2002nonlinear, rasmussen2006gaussian, solak2003derivative}. In economics, it has a long tradition to study costs and demands where the factor demand function is the partial derivative of the cost function by the Shephard's Lemma \cite{shepherd2015theory, jorgenson1984econometric, hall2007nonparametric, hall2010nonparametric}. In actuarial science,  mortality force data can be obtained from demography, which together with samples for the survival distribution can yield derivatives for the survival distribution function \cite{frees1998understanding}. In computer experiments, partial derivatives are available by using differentiation mechanisms at little additional cost \cite{hansen2003global, griewank2008evaluating, golub2014scientific}.  Derivative data are commonly collected in geodetics engineering \cite{schwarz1979geodetic, plessix2006review}. 
In meteorology, the wind speed and direction as functions of the gradient of barometric pressure are measured over broad geographic regions while the pressure will also be recorded \cite{breckling2012analysis}.
Moreover, an evolving system is often modeled as a constrained optimization problem or a set of partial differential equations, which give data on the first order condition or partial derivatives as well as the objective function itself \cite{forrester2008engineering,ramsay2007parameter}.

Let $\partial f(\bt)/\partial t_j$ denote the $j$th first-order partial derivative of a scalar function $f(\bt)$ of $d$ variables  $\bt=(t_1,\ldots,t_d)$. Consider the following multivariate regression model 
 \begin{equation}
\label{modelequation}
\begin{cases}
Y^{e_0}  = f_0(\bt^{e_0}) + \epsilon^{e_0},\\
Y^{e_j}  =\partial f_0/\partial t_j(\bt^{e_j})+\epsilon^{e_j},   \quad 1\leq j\leq p.
\end{cases}
\end{equation}
Here, $e_j$ is a $d$-dimensional vector with $j$th entry one and others zero and $e_0$ is a zero vector.
The response $Y^{e_0}$ is the function observation and $Y^{e_j}$ is the observation of the first-order partial derivative on the $j$th covariate.
Assume that the design points $\bt^{e_0}$ and $\bt^{e_j}$s are in a compact product space $\XX_1^d$, where $\XX_1=[0,1]$. 
The random errors $\epsilon^{e_0}$ and $\epsilon^{e_j}$s are assumed to be independent centered noises with variances $\sigma_0^2$ and $\sigma_j^2$s, respectively. Let $p\in\{1,\ldots,d\}$ denote the number of the different types of first-order partial derivatives being observed. 
Without loss of generality, we focus on
the first $p$ components for notational convenience. Let $\{(\bt_i^{e_j},y_i^{e_j}): i=1,\ldots, n\}$ be independent copies of $(\bt^{e_j},Y^{e_j})$ for $j=1,\ldots,p$, and $\{(\bt_i^{e_0},y_i^{e_0}): i=1,\ldots, n\}$ be independent copies of $(\bt^{e_0},Y^{e_0})$. 

We now discuss two popular approaches for modeling the  $d$-dimensional nonparametric unknown function $f_0(\cdot)$. The first uses a multivariate function with smoothness assumption on all $d$ dimensions. The second uses a function with tensor product structure and smoothness properties on lower dimensions. The latter approach is represented by the smoothing spline analysis of variance (SS-ANOVA). See, for example,  \cite{wahba1990, wahba1995smoothing, lin2000tensor, gu2013smoothing} and references therein. 
As a general framework for nonparametric multivariate estimation, SS-ANOVA  can adaptively control the complexity of the model with interpretable estimates. The SS-ANOVA model for a function $f(\bt)$ is
\begin{equation}
\label{eqn:anovadecompfti}
f(\bt) = \mbox{constant} + \sum_{k=1}^d f_{k}(t_k) + \sum_{k<j}f_{kj}(t_k,t_j) + \cdots, 
\end{equation}
where the $f_{k}$s are the main effects, the $f_{kj}$s are the two-way interactions, and so on. Components on the right hand side satisfy side conditions to assure identifiability. 
The series is truncated to some order $r$ of interactions  to enhance interpretability, where $1\leq r\leq d$. This model generalizes the popular additive model where $r=1$ and fitted with smoothing splines (see, e.g., \cite{buja1989linear, hastie1990generalized}).

We assume that the true function $f_0(\cdot)$ is a SS-ANOVA model and reside in a certain reproducing kernel Hilbert space (RKHS) $\HH$ on $\XX_1^d$. Let $\HH^{(k)}$ be an RKHS of functions of $t_k$ on $\XX_1$ with $\int_{\XX_1}f_k(t_k)dt_k =0$ for $f_k(t_k)\in\HH^{(k)}$ and $[1^{(k)}]$ be the one-dimensional space of constant functions on $\RR_1$. Construct $\HH$ as
\begin{equation}
\label{eqn:anovadechi}
\begin{aligned}
\HH & = \prod_{k=1}^d \left(\left\{[1^{(k)}]\right\}\oplus \left\{\HH^{(k)}\right\}\right) \\
& = [1] \oplus \sum_{k=1}^d \HH^{(k)} \oplus\sum_{k<j} [\HH^{(k)}\otimes \HH^{(j)}]\oplus \cdots,
\end{aligned}
\end{equation}
where $[1]$ denotes the constant functions on $\XX_1^d$. 
The components of the SS-ANOVA decomposition (\ref{eqn:anovadecompfti}) are now in mutually orthogonal subspaces of $\HH$ in  (\ref{eqn:anovadechi}). 
We further assume that all component functions come from a common RKHS $(\HH_1,\|\cdot\|_{\HH_1})$, that is $\HH^{(k)} \equiv\HH_1$ for  $k=1,\ldots,d$. 
 Let $K:\XX_1\times\XX_1\mapsto\R$ be a Mercer kernel generating the RKHS $\HH_1$ and write 
$K_d\left((t_1,\ldots,t_d)^\top,(t_1',\ldots,t_d')^\top\right) = K(t_1,t_1') \cdots  K(t_d,t_d').$
Then $K_d$ is the reproducing kernel of RKHS $(\HH,\|\cdot\|_{\HH})$ (see, e.g., \cite{aronszajn1950theory}).

\subsection{Deterministic designs}
We are interested in the minimax optimal convergence rates for estimating $f_0(\cdot)$ and its partial derivatives $\partial f_0/\partial t_j(\cdot)$. We begin by considering regular lattices, also known as tensor product designs \cite{bates1996experimental, riccomagno1997lattice}. 
Suppose that the eigenvalues of the $K$ decay polynomially with the $\nu$th largest eigenvalue of the order $\nu^{-2m}$. We show that the minimax rate for estimating $f_0\in\HH$ for  full $d$-interaction SS-ANOVA model is
\begin{equation}
\label{eqn:minimaxraternnrd}
\begin{aligned}
& \inf_{\tilde{f}}\sup_{f_0\in\HH}\E\int_{\XX_1^d}\left[\tilde{f}(\bt)-f_0(\bt)\right]^2d\bt \\
& \quad\quad\quad\quad =
\begin{cases}
\left[n(\log n)^{1+p-d}\right]^{-2m/(2m+1)}  & \mbox{ if } 0\leq p< d,\\
n^{-1}(\log n)^{d-1} + n^{-2md/[(2m+1)d-2]}& \mbox{ if } p=d,
\end{cases}
\end{aligned}
\end{equation}
up to a constant scaling factor. If $0\leq p<d$, the above rate is the minimax optimal rate for estimating a $(d-p)$ dimensional full interaction SS-ANOVA model with only function observations; see, for example, \cite{gu2013smoothing,lin2000tensor}. If $p=d$ and $d\geq 3$, the minimax optimal rate in (\ref{eqn:minimaxraternnrd}) becomes
\begin{equation}
\label{eqn:rrnpequald}
\inf_{\tilde{f}}\sup_{f_0\in\HH}\E\int_{\XX_1^d}\left[\tilde{f}(\bt)-f_0(\bt)\right]^2d\bt\asymp n^{-2md/[(2m+1)d-2]}.
\end{equation}
For two positive sequences $a_{n}$ and $b_{n}$, we write  $a_{n}\asymp b_{n}$ if $a_{n}/b_{n}$ is bounded away from zero and infinity. The rate given by (\ref{eqn:rrnpequald}) converges \textit{faster} than the well known optimal rate $n^{-2m/(2m+1)}$ for additive models given in  \cite{hastie1990generalized,stone1985additive}. If $p=d$ and $d=2$, the minimax optimal rate in (\ref{eqn:minimaxraternnrd}) is $n^{-1}\log n$. If $p=d$ and $d=1$, the root-$n$ consistency is achieved in (\ref{eqn:minimaxraternnrd}) and this specific phenomenon has been observed earlier (see, e.g., \cite{cox1988approximation,hall2007nonparametric}). 

We are the first to systematically investigate the estimation of general $d$-dimensional SS-ANOVA models with derivatives.
 Other convergence rate results for truncated SS-ANOVA models ($r<d$) will be given in Section \ref{sec:minmaxriskregularlat}. In particular, for the additive model $r=1$ and $p=d$, the minimax optimal rate is
$n^{-1}$,
which coincides with the \textit{parametric} convergence rate.

\subsection{Random designs}
We are interested in obtaining sharp results for random designs. Suppose that design points $\bt^{e_0}$ and $\bt^{e_j}$ are independently drawn from distributions $\Pi^{e_0}$ and $\Pi^{e_j}$s, where they are supported on $\XX_1^d$. We show that
the minimax optimal rate for estimating the full $d$-interaction SS-ANOVA model is 
\begin{equation}
\label{eqn:lowboundrandomdsgn}
\begin{aligned}
\inf_{\tilde{f}}\sup_{f_0\in\HH}\P  & \left\{  \int_{\XX_1^d}\left[\tilde{f}(\bt)-f_0(\bt)\right]^2d\bt \geq C_1\left(\left[n(\log n)^{1+p-d}\right]^{-2m/(2m+1)} \mathbbm{1}_{0\leq p<d} \right.\right.\\
& \left.\left. \quad\quad\quad\quad+ \left[n^{-1}(\log n)^{d-1} + n^{-2md/[(2m+1)d-2]}\right] \mathbbm{1}_{p=d}\vphantom{\left[n(\log n)^{1+p-d}\right]^{-2m/(2m+1)}}\right)\vphantom{\int_{\XX_1^d}\left\{\tilde{f}(\bt)-f_0(\bt)\right\}^2d\bt}\right\}=0,
\end{aligned}
\end{equation}
where $C_1$ is a constant scalar not depending on $n$. The minimax optimal rates are also obtained for estimating $\partial f_0/\partial t_j(\cdot)$ for any
$j\in\{1,\ldots,p\}$ and both full and truncated SS-ANOVA models with $r\leq d$, which are
\begin{equation}
\label{eqn:rrej2m1}
\inf_{\tilde{f}}\sup_{f_0\in\HH}\P \left\{\int_{\XX_1^d}\left[\tilde{f}(\bt)-\partial f_0/\partial t_j(\bt)\right]^2d\bt \geq C_2n^{-2(m-1)/(2m-1)} \right\}>0,
\end{equation}
where $C_2$ does not depend on $n$.
This result holds regardless of the value of $d$, $r$ and $p$. In particular, the rate is the same as the optimal rate for estimating $\partial f_0/\partial t_j(\cdot)$ if $f_0$ actually comes from a univariate function space $\HH_1$ instead of the $d$-variate function space $\HH$. See, for example, \cite{stone1980optimal, stone1982optimal}.

We achieve the minimax rates under deterministic designs (\ref{eqn:minimaxraternnrd}) and  random designs   (\ref{eqn:lowboundrandomdsgn}) by using the method of regularization in the framework of RKHS. Unlike the regularization method,  alternative methods  for modeling derivative data typically assume that data have no random noises.  See, for example, \cite{carr2001reconstruction, mitchell1994asymptotically, morris1993,  solak2003derivative} among others.
Despite these existing works, theoretical understandings of observed first-order partial derivatives is limited. Our work fills some gap in this direction. It is worth pointing out the differences between this work and \cite{hall2007nonparametric}. The estimator provided in \cite{hall2007nonparametric} relies on observing the complete set of $2^s$ types of mixed  derivatives on $s$ variables of a $d$-dimensional function with $s\leq d$. Their requirement could be infeasible for some problems in practice while our setting fits for any observed first-order partial derivatives. Moreover, \cite{hall2007nonparametric} does not provide the minimax risk analysis and considers the estimation of $d$-dimensional functions without the tensor product structure. Thus, \cite{hall2007nonparametric} concludes that adding more than one type first-order partial derivative data \textit{does not} further improve the convergence rate of their estimator. These results are different from our work in, for exmaple, (\ref{eqn:minimaxraternnrd}),  (\ref{eqn:lowboundrandomdsgn}) and (\ref{eqn:rrej2m1}) for functional ANOVA models.  

The rest of the article is organized as follows. We give the main results on estimating functions with deterministic designs in Section \ref{sec:minmaxriskregularlat}, where (\ref{eqn:minimaxraternnrd}) and (\ref{eqn:rrnpequald}) are included.  We present the main results with random designs in Section \ref{sec:randomdesign} including (\ref{eqn:lowboundrandomdsgn}). 
We  consider the optimal rates of estimating first-order partial derivatives in Section \ref{sec:minmaxriskpartder}, where (\ref{eqn:rrej2m1}) is elaborated.
Proofs of the results with random designs are given in Section \ref{sec:proofsofall}. Proofs of other results and auxiliary technical lemmas are relegated to the supplementary material.

\section{Minimax risks with regular lattices}
\label{sec:minmaxriskregularlat}

This section provides the minimax optimal rates of estimating $f_0(\cdot)$ with model (\ref{modelequation}) and regular lattices. A regular lattice of size $n=l_1\times \cdots \times l_d$ on $\XX_1^d$ is a collection of design points
\begin{equation}
\label{eqn:reglatticeofti}
\{\bt_1,\ldots,\bt_n\}=\{(t_{i_1,1}, t_{i_2,2}, \ldots, t_{i_d,d}) | i_k = 1,\ldots, l_k, k=1,\ldots, d\},
\end{equation}
where $t_{j,k} = j/l_k$, $j=1,\ldots, l_k, k=1,\ldots, d$. This design is often used in the statistical literature when the true function $f_0$ is a functional ANOVA model. This design is $D-$optimal in the sense of Kiefer and Wolfowitz \cite{kiefer1959optimum}. Readers are referred to \cite{bates1996experimental, riccomagno1997lattice} for further details.  Under the regular lattice design, it is reasonable to assume $f_0: \XX_1^d \mapsto \R$ to have a periodic boundary condition. This is because any finite-length sequence $\{f(\bt_1),\ldots, f(\bt_n)\}$ can be associated with a periodic sequence
\begin{equation*}
\begin{aligned}
&  f^{\text{per}} \left(\frac{i_1}{l_1}, \cdots, \frac{i_d}{l_d}\right)  \\
& \quad = \sum_{q_1=-\infty}^{\infty}\cdots\sum_{q_d=-\infty}^{\infty}f\left(\frac{i_1}{l_1}-q_1, \cdots, \frac{i_d}{l_d}-q_d\right),\quad \forall (i_1,\ldots,i_d)\in \Z^d
\end{aligned}
\end{equation*}
by letting $f(\cdot)\equiv 0$ outside $\XX_1^d$ and at the unobserved boundaries of $\XX_1^d$. On the other hand, any finite-length sequence $\{f(\bt_1),\ldots, f(\bt_n)\}$ can be recovered from  the periodic sequence $f^{\text{per}}(\cdot)$.

Recall that $K$ is the reproducing kernel for component RKHS $\HH_1$, which is a symmetric positive semi-definite, square integrable function on $\XX_1\times\XX_1$. 
In our setting, we require an additional differentiability condition on kernel $K$, which is given by
\begin{equation}
\label{eqn:partial2tt'kcx1x1}
\frac{\partial^2 }{\partial t\partial t'} K(t,t')\in C(\XX_1\times \XX_1).
\end{equation}
An straightforward explanation on this condition is as follows.
Denote by $\langle\cdot,\cdot\rangle_\HH$ the inner product of  RKHS $\HH$ in (\ref{eqn:anovadechi}).
Then, for any $g\in\HH$,  we have
\begin{equation}
\label{eqn:partialgttpartialtjhh}
\frac{\partial g(\bt)}{\partial t_j} = \frac{\partial \langle g, K_d(\bt,\cdot)\rangle_\HH}{\partial t_j} =  \left\langle g, \frac{ \partial K_d(\bt,\cdot)}{\partial t_j}\right\rangle_\HH,
\end{equation}
where the last step is by the continuity of $\langle\cdot,\cdot\rangle_\HH$.
This implies that the composite functional of evaluation and partial differentiation, $\partial g/\partial t_j(\bt)$,  is a bounded linear functional in $\HH$ and has a representer $\partial K_d(\bt,\cdot)/\partial t_j$ in $\HH$. 

From Mercer's theorem  \cite{riesz1955}, $K$ admits a spectral decomposition
\begin{equation}
\label{eqn:mercerdecaykttgamma}
K(t,t') = \sum_{\nu=1}^\infty \lambda_\nu\psi_\nu(t)\psi_\nu(t'),
\end{equation}
where $\lambda_{1}\geq \lambda_2\geq \cdots\geq 0$ are its eigenvalues and $\{\psi_{\nu}:\nu\geq 1\}$ are the corresponding eigenfunctions.  A canonical example of $\HH_1$ is the $m$th order Sobolev space $\WW_2^m(\XX_1)$ whose eigenvalues satisfy $\lambda_\nu\asymp \nu^{-2m}$. See, for example, Wahba \cite{wahba1990} for further examples. Here, (\ref{eqn:partialgttpartialtjhh}) implies that $\partial g/\partial t_j(\bt)$ is a continuous function. Thus, if $\HH_1=\WW_2^m(\XX_1)$, we shall require $m>3/2$ by Sobolev embedding theorem.

We are now in the position to present our main results. We first state a minimax lower bound under regular lattices.

\begin{theorem} 
\label{theorem:lowerbdfNlambdareg}
Assume that $\lambda_\nu \asymp \nu^{-2m}$ for some $m>3/2$, and design points $\bt^{e_0}$ and $\bt^{e_j}, j=1,\ldots,d$, are from the regular lattice (\ref{eqn:reglatticeofti}). Suppose that $f_0\in\HH$ has periodic boundaries on $\XX_1^d$ and is truncated up to $r$ interactions in (\ref{eqn:anovadecompfti}). Then, as $n\rightarrow \infty$,
\begin{equation*}
\begin{aligned}
& \inf_{\tilde{f}}\sup_{f_0\in\HH}\E\int_{\XX_1^d}\left[\tilde{f}(\bt)-f_0(\bt)\right]^2d\bt \\
& \quad\quad\quad\quad = 
\begin{cases}
\left[n(\log n)^{1-(d-p)\wedge r}\right]^{-2m/(2m+1)},  & \mbox{ if } 0\leq p< d\\
n^{-1}(\log n)^{r-1} + n^{-2mr/[(2m+1)r-2]}, & \mbox{ if } p=d
\end{cases}
\end{aligned}
\end{equation*}
up to a constant factor which only depends on bounded values $\sigma_0^2$, $\sigma_j^2$s, $m$, $r$, $p$,  and $d$.
\end{theorem}

We relegate the proof to Section \ref{sec:prooflowbdreg} in the supplementary material. Next, we show the lower bounds of convergence rates in Theorem \ref{theorem:lowerbdfNlambdareg} are obtainable. In particular,
we consider the method of regularization by  simultaneously minimize the empirical losses of function observations and partial derivative observations with a single penalty: 
\begin{equation}
\label{scheme1}
\begin{aligned}
\widehat{f}_{n\lambda} = \underset{f\in \HH}{\arg\min}& \left\{\frac{1}{n(p+1)}\left[\frac{1}{\sigma_0^2}\sum_{i=1}^{n}\left\{y_i^{e_0}- f(\bt_i^{e_0})\right\}^2\right.\right.\\
& \quad\left.\left.+\sum_{j=1}^p\frac{1}{\sigma_j^2}\sum_{i=1}^{n}\left\{y_i^{e_j}-\partial f/\partial t_j(\bt_i^{e_j})\right\}^2\right]+\lambda J(f)\right\},
\end{aligned}
\end{equation}
where the weighted squared error loss may be replaced by other convex losses, and $J(\cdot)$ is a quadratic penalty associated with RKHS $\HH$, and $\lambda\geq 0$ is a tuning parameter. 
The following theorem shows $\widehat{f}_{n\lambda}$ in (\ref{scheme1}) is indeed minimax rate optimal.

\begin{theorem}
\label{thm;deterdesgnupperbdgeneralpdreg}
Under the  conditions of Theorem \ref{theorem:lowerbdfNlambdareg}, $\widehat{f}_{n\lambda}$ given by (\ref{scheme1}) satisfies
\begin{equation*}
\begin{aligned}
& \E\int_{\XX_1^d}\left[\widehat{f}_{n\lambda}(\bt)-f_0(\bt)\right]^2d\bt\\
& \quad\quad\quad\quad = 
\begin{cases}
\left[n(\log n)^{1-(d-p)\wedge r}\right]^{-2m/(2m+1)}  & \mbox{ if } 0\leq p< d,\\
n^{-1}(\log n)^{r-1} + n^{-2mr/[(2m+1)r-2]}& \mbox{ if } p=d,
\end{cases}
\end{aligned}
\end{equation*}
up to a constant factor which only depends on bounded values $\sigma_0^2$, $\sigma_j^2$s, $m$, $r$, $p$, and $d$, if tuning parameter $\lambda$ is chosen by $\lambda\asymp \left[n(\log n)^{1-(d-p)\wedge r}\right]^{-2m/(2m+1)} $ when $0\leq p<d$, and $\lambda\asymp n^{-(2mr-2)/[(2m+1)r-2]}$ when $p=d, r\geq 3$, and $\lambda\asymp (n\log n)^{-(2m-1)/2m}$ when $p=d, r= 2$, and $\lambda\lesssim n^{-(m-1)/m}$ when $p=d$, $r=1$.
\end{theorem}

The proof of this theorem is presented in Section \ref{sec:proofupperbdreg} in the supplementary material. Theorems \ref{theorem:lowerbdfNlambdareg} and \ref{thm;deterdesgnupperbdgeneralpdreg} together immediately imply that
with model  (\ref{modelequation}) and regular lattices, the minimax optimal rate for estimating $f_0\in\HH$ is
\begin{equation}
\label{eqn:minaxhatftt}
\begin{aligned}
& \E\int_{\XX_1^d}\left[\widehat{f}(\bt)-f_0(\bt)\right]^2d\bt\\
& \quad\quad\quad\quad = 
\begin{cases}
\left[n(\log n)^{1-(d-p)\wedge r}\right]^{-2m/(2m+1)},  & \mbox{ if } 0\leq p< d,\\
n^{-1}(\log n)^{r-1} + n^{-2mr/[(2m+1)r-2]}, & \mbox{ if } p=d,
\end{cases}
\end{aligned}
\end{equation}
and the method of regularization achieves (\ref{eqn:minaxhatftt}).
We make several remarks on this result. First, suppose there is no derivative data, for example, $p=0$ and  $r=d$. Then, (\ref{eqn:minaxhatftt}) recovers $[n(\log n)^{1-d}]^{-2m/(2m+1)}$ and this rate is known in literature (see, e.g., \cite{gu2013smoothing}). For a large $n$, the exponential term $(\log n)^{d-1}$ makes the full $d$-interaction SS-ANOVA model impractical for large $d$. On the contrary, suppose partial derivatives data are available, for example, $p=d-1$ and $r=d$. Then, (\ref{eqn:minaxhatftt})  gives $n^{-2m/(2m+1)}$ for any $d\geq 1$, which coincides with the classical optimal rate for additive models \cite{hastie1990generalized, stone1985additive} and is not affected by the dimension $d$.

Second, if partial derivative observations are available on all covariates with  $p=d$, then the optimal rate can be much improved. Besides (\ref{eqn:rrnpequald}) for $r=d$ and $d\geq 3$, we point out 
some other interesting cases. For the additive model with $r=1$ and $d\geq 1$, (\ref{eqn:minaxhatftt}) provides the minimax rate $n^{-1}$. For the pairwise interaction model with $r=2$ and $d\geq 1$, (\ref{eqn:minaxhatftt}) provides the minimax rate $n^{-1}\log n$, which is different from $n^{-1}$ only by a $\log n$ multiplier.

Third, we remark on an ``interaction reduction" phenomenon. That is to say,
the optimal rate for estimating an unknown SS-ANOVA model by incorporating partial derivative data is the same as the optimal rate for estimating a \textit{reduced interaction} SS-ANOVA  without derivative data. For example, with $r=d$ and $p= 1$,  (\ref{eqn:minaxhatftt}) gives $[n(\log n)^{1-(d-1)}]^{-2m/(2m+1)}$, which is the same rate as  $r=d-1$ and $p=0$ involving no derivative observations but a lower degree of interactions.  And, with $r=d$ and $p=2$, (\ref{eqn:minaxhatftt}) gives $[n(\log n)^{1-(d-2)}]^{-2m/(2m+1)}$, which is the same rate as $r=d-2$ and $p=0$ involving no derivative observations but two lower degrees of interactions. Similarly, we can extend the same discussion to $p=3,\ldots, d-1$.

Fourth, by reviewing the proof for Theorem \ref{theorem:lowerbdfNlambdareg} and \ref{thm;deterdesgnupperbdgeneralpdreg},  we  find  that when $p=d$, both the squared bias and variance  are smaller in magnitude than $p<d$, and when $d-r<p<d$, only the variance is smaller in magnitude than $0\leq p\leq d-r$.

Finally, let $n_0$ denote the sample size on $(\bt^{e_0},Y^{e_0})$ and $n_j$ denote the sample sizes on $(\bt^{e_j},Y^{e_j})$, where $1\leq j\leq p$. If $n_0$ and $n_j$s are not all identical to $n$, we can show that $n$ in (\ref{eqn:minaxhatftt}) can be replaced by $\min_{1\leq j\leq p}n_j$.

\section{Minimax risks with random designs}
\label{sec:randomdesign}

We now turn to random designs for the minimax optimal rates of estimating $f_0(\cdot)$ with the regression model (\ref{modelequation}). Parallel to Theorem \ref{theorem:lowerbdfNlambdareg}, we have the following minimax lower bound of estimation under random designs.

\begin{theorem} 
\label{theorem:lowerbdfNlambdaregrandom}
Assume that $\lambda_\nu \asymp \nu^{-2m}$ for some $m>3/2$, and design points $\bt^{e_0}$ and $\bt^{e_j}, j=1,\ldots,d$, are independently drawn from $\Pi^{e_0}$ and $\Pi^{e_j}$s, respectively. Suppose that $\Pi^{e_0}$ and $\Pi^{e_j}$s have densities bounded away from zero and infinity, and $f_0\in\HH$ is truncated up to $r$ interactions in (\ref{eqn:anovadecompfti}). Then, as $n\rightarrow \infty$,
\begin{align*}
\inf_{\tilde{f}}\sup_{f_0\in\HH}\P& \left\{\int_{\XX_1^d}\left[\tilde{f}(\bt)-f_0(\bt)\right]^2d\bt \geq C_1\left(\left[n(\log n)^{1-(d-p)\wedge r}\right]^{-2m/(2m+1)} \mathbbm{1}_{0\leq p<d} \right.\right.\\
& \left.\left. \quad\quad\quad\quad\quad + \left[n^{-1}(\log n)^{r-1} + n^{-2mr/[(2m+1)r-2]}\right] \mathbbm{1}_{p=d}\vphantom{\int_{\XX_1^d}\left\{\tilde{f}(\bt)-f_0(\bt)\right\}^2d\bt}\right)\right\}>0
\end{align*}
where the constant $C_1$ only depends on bounded values $\sigma_0^2$, $\sigma_j^2$s, $m$,  $r$, $p$, and $d$.
\end{theorem}
The lower bound is  established via Fano's lemma; see, for example, \cite{tsybakovintroduction, cover2012elements}. The proof is deferred to Section \ref{subsubsec:prooflowerrandom}. Next, we show the lower bounds of convergence rates in Theorem \ref{theorem:lowerbdfNlambdaregrandom} can be achieved by using the regularized estimator in (\ref{scheme1}). 
\begin{theorem}
\label{thm:mainupperrateestf0}
Under the conditions of Theorem \ref{theorem:lowerbdfNlambdaregrandom}, we assume that $\Pi^{e_0}$ and $\Pi^{e_j}$s are known, and  $m>2$. Then, $\widehat{f}_{n\lambda}$ in (\ref{scheme1}) satisfies 
\begin{align*}
\lim_{D_1\rightarrow\infty}\underset{n\rightarrow\infty}{\lim\sup}& \sup_{f_0\in\HH} \P \left\{\int_{\XX_1^d}\left[\widehat{f}_{n\lambda}(\bt)-f_0(\bt)\right]^2d\bt > D_1\left(\left[n(\log n)^{1-(d-p)\wedge r}\right]^{-2m/(2m+1)}  \right.\right.\\
& \left.\left. \cdot \mathbbm{1}_{0\leq p<d} + \left[n^{-1}(\log n)^{r-1} + n^{-2mr/[(2m+1)r-2]}\right] \mathbbm{1}_{p=d}\vphantom{\left[n(\log n)^{1-(d-p)\wedge r}\right]^{-2m/(2m+1)}}\right)\vphantom{\int_{\XX_1^d}\left\{\tilde{f}(\bt)-f_0(\bt)\right\}^2d\bt}\right\}=0
\end{align*}
if the tuning parameter $\lambda$ is chosen by $\lambda\asymp \left[n(\log n)^{1-(d-p)\wedge r}\right]^{-2m/(2m+1)} $ when $0\leq p<d$, and $\lambda\asymp n^{-(2mr-2)/[(2m+1)r-2]}$ when $p=d, r\geq 3$, and $\lambda\asymp (n\log n)^{-(2m-1)/2m}$ when $p=d, r= 2$, and $\lambda\lesssim n^{-(m-1)/m}$ when $p=d$, $r=1$. In other words, $\widehat{f}_{n\lambda}$ is rate optimal.
\end{theorem}

We use the linearization method in \cite{cox1990asymptotic} to prove Theorem \ref{thm:mainupperrateestf0}.
The key ingredient of this method is to chose a suitable basis such that  the expected loss of the regularization  and the quadratic penalty $J(\cdot)$  can be simultaneously diagonalized. 
For applications where these two functionals are positive semi-definite,  the existence of such a basis is guaranteed by the classical operator theory (see, e.g., \cite{weinberger1974variational}). These are done in \cite{lin2000tensor, yuan2010reproducing, gu2013smoothing}. 
Our situation is different in the sense that the loss function in (\ref{scheme1}) is the sum of squared error losses for both the function and partial derivatives but we are only interested in estimating the function itself in Theorem \ref{thm:mainupperrateestf0}. This induces a third positive semi-definite functional, which is the squared error loss of function estimation. But  three functionals  are not guaranteed to be simultaneously diagonized, making the direct application of the linearization method infeasible.  We present a detailed proof in Section \ref{subsubsec:prooflowerrandom}.

Theorems \ref{theorem:lowerbdfNlambdaregrandom} and \ref{thm:mainupperrateestf0} together demonstrate  the fundamental limit rate of the squared error loss  for estimating $f_0\in\HH$ with model (\ref{modelequation}) and random designs is
\begin{equation}
\label{eqn:minaxhatfttprob}
\begin{aligned}
& \left[n(\log n)^{1-(d-p)\wedge r}\right]^{-2m/(2m+1)} \mathbbm{1}_{0\leq p<d}\\
& \quad\quad\quad\quad+ \left[n^{-1}(\log n)^{r-1} + n^{-2mr/[(2m+1)r-2]}\right] \mathbbm{1}_{p=d}
\end{aligned}
\end{equation}
in a probabilistic sense, and the regularized estimator achieves (\ref{eqn:minaxhatfttprob}). The minimax rate is the same as that with the regular lattice. We make several remarks on (\ref{eqn:minaxhatfttprob}). First, all five remarks following (\ref{eqn:minaxhatftt}) for the mean squared situation hold for (\ref{eqn:minaxhatfttprob}) in a  probabilistic sense.  

Second, for the special case when $p=0$, (\ref{eqn:minaxhatfttprob}) recovers the minimax optimal rate of convergence $O_\P\left\{[n(\log n)^{1-r}]^{-2m/(2m+1)}\right\}$ for SS-ANOVA models, which is known in \cite{lin2000tensor}. 

Third, the squared error loss in Theorems \ref{theorem:lowerbdfNlambdaregrandom} and \ref{thm:mainupperrateestf0} can be replaced by squared prediction error $\int\{\widehat{f}_{n\lambda}(\bt) - f_0(\bt)\}^2d\Pi^{e_0}(\bt)$ and it achieves the same minimax optimal rate as (\ref{eqn:minaxhatfttprob}). 

Fourth, although (\ref{eqn:minaxhatfttprob}) is established by assuming design points are drawn independently, it also holds for designs of function and derivatives can be grouped to some sets, where within the sets the design points are drawn identically and across the sets the design points are drawn independently. For example, when $p=2$, (\ref{eqn:minaxhatfttprob}) still holds if the designs can be  grouped to $\{\bt^{e_0}\mbox{ are drawn from } \Pi^{e_0}\}$ and $\{\bt^{e_1}\equiv \bt^{e_2} \mbox{ are drawn from } \Pi^{e_1}\}$ and these two sets are drawn independently.

As a byproduct of Theorem \ref{thm:mainupperrateestf0}, we show the following result of estimating the mixed partial derivatives $\frac{\partial^d f_0}{\partial t_1\cdots \partial t_d}(\bt)$ by its natural estimator $\frac{\partial^d \widehat{f}_{n\lambda}}{\partial t_1\cdots \partial t_d}(\bt)$. 

\begin{corollary}
\label{theorem:estboundderiva}
Under the conditions of Theorem \ref{thm:mainupperrateestf0} and $m>3$, we have
\begin{align*}
 \lim_{D_1'\rightarrow\infty}\underset{n\rightarrow\infty}{\lim\sup}\sup_{f_0\in\HH}\P & \left\{\int_{\XX_1^d}\left[\frac{\partial^{d}\widehat{f}_{n\lambda}(\bt)}{\partial t_1\cdots\partial t_d}-\frac{\partial^d f_0(\bt)}{\partial t_1\cdots \partial t_d}\right]^2d\bt \right.\\
&   \quad\quad> D_1'\left(\left[n(\log n)^{1-(d-p)\wedge r}\right]^{-2(m-1)/(2m+1)} \mathbbm{1}_{0\leq p<d} \right.\nonumber \\
&\quad\quad\quad\quad \left.\left. + \left[n^{-2(m-1)r/[(2m+1)r-2]}\right] \mathbbm{1}_{p=d}\vphantom{\left[n(\log n)^{1-(d-p)\wedge r}\right]^{-2m/(2m+1)}}\right)\vphantom{\int_{\XX_1^d}\left\{\tilde{f}(\bt)-f_0(\bt)\right\}^2d\bt}\right\}=0,
\end{align*}
if the tuning parameter $\lambda$ is chosen by $\lambda\asymp \left[n(\log n)^{1-(d-p)\wedge r}\right]^{-2m/(2m+1)} $ when $0\leq p<d$, and $\lambda\asymp n^{-(2mr-2)/[(2m+1)r-2]}$ when $p=d$.
\end{corollary}

\section{Minimax risk for estimating partial derivatives}
\label{sec:minmaxriskpartder}

If one observes noisy data on the function and some partial derivatives in (\ref{modelequation}), it is natural to  ask what is the optimal rate for estimating first-order partial derivatives by using all observed data.  
For brevity, we only consider random designs although similar results can be derived for regular lattices by using techniques in Section \ref{sec:minmaxriskregularlat}. The following theorem  gives the minimax lower bound for estimating $\partial f_0/\partial t_j$, $1\leq j\leq p$.

\begin{theorem} 
\label{theorem:lowerbndlimDderi}
Assume that $\lambda_\nu \asymp \nu^{-2m}$ for some $m>2$ and design points $\bt^{e_0}$ and $\bt^{e_j}, j=1,\ldots,d$, are independently drawn from $\Pi^{e_0}$ and $\Pi^{e_j}$s, respectively. Suppose that $\Pi^{e_0}$ and $\Pi^{e_j}$s have densities bounded away from zero and infinity, and $f_0\in\HH$ is truncated up to $r$ interactions in (\ref{eqn:anovadecompfti}). Then, for any $j\in\{1,\ldots,p\}$ and $1\leq r\leq d$, as $n\rightarrow\infty$,
\begin{align*}
\inf_{\tilde{f}}\sup_{f_0\in\HH}\P& \left\{\int_{\XX_1^d}\left[\tilde{f}(\bt)-\frac{\partial f_0(\bt)}{\partial t_j}\right]^2d\bt \geq C_2n^{-2(m-1)/(2m-1)} \right\}>0,
\end{align*}
where $C_2$ only depends on bounded values $\sigma_0^2$, $\sigma_j^2$s, $m$, $r$, $p$, and $d$.
\end{theorem}

We will prove this theorem in Section \ref{subsubsec:lowerbdestpartder} in the supplementary material. As a natural estimator for $\partial f_0/\partial t_j$, $\partial \widehat{f}_{n\lambda}/\partial t_j$ achieves the lower bound of convergence rates in Theorem \ref{theorem:lowerbndlimDderi}.

\begin{theorem}
\label{theorem:upperbndlimDderi}
Under the conditions of Theorem \ref{theorem:lowerbndlimDderi}, $\widehat{f}_{n\lambda}$ given by (\ref{scheme1}) satisfies that for any $j\in\{1,\ldots,p\}$ and $1\leq r\leq d$,
\begin{equation*}
\lim_{D_2\rightarrow\infty}\underset{n\rightarrow\infty}{\lim\sup}\sup_{f_0\in \HH} \P\left\{\int_{\XX_1^d}\left[ \frac{\partial \widehat{f}_{n\lambda}(\bt)}{\partial t_j} -\frac{\partial f_0(\bt)}{\partial t_j}\right]^2d\bt > D_2n^{-2(m-1)/(2m-1)}\right\} = 0,
\end{equation*}
if the tuning parameter $\lambda$ is chosen by $\lambda \asymp n^{-2(m-1)/(2m-1)}$.
\end{theorem}
The proof of this theorem is given in Section \ref{subsubsec:upperbdestpartder} in the supplementary material.
When $r=1$, this result coincides with Corollary \ref{theorem:estboundderiva}.
Different from Theorem \ref{thm:mainupperrateestf0} and Corollary \ref{theorem:estboundderiva},  the distributions $\Pi^{e_0}$ and $\Pi^{e_j}$s are not assumed to be known.

Theorems \ref{theorem:lowerbndlimDderi} and \ref{theorem:upperbndlimDderi} together give the minimax optimal rate for estimating $\partial f_0/\partial t_j$, which is given in (\ref{eqn:rrej2m1}). To the best of our knowledge, there are few existing results in literature about estimating first-order partial derivatives.
Since the optimal rate in (\ref{eqn:rrej2m1}) holds regardless of the value of $p\geq 1$, first-order partial derivative data on different covariates do not improve the optimal rates for estimating each other. For example, given noisy data on $f_0(\cdot)$ and $\partial f_0/\partial t_j(\cdot)$, the data on $\partial f_0/\partial t_k(\cdot)$ does not improve the minimax optimal rate for estimating $\partial f_0/\partial t_j(\cdot)$ if $1\leq k\neq j\leq p$.


\section{Proofs for Section \ref{sec:randomdesign}: random designs}
\label{sec:proofsofall}

Before proving the main results, we give some preliminary background on the RKHS $\HH$.  Since the SS-ANOVA model (\ref{eqn:anovadecompfti}) truncates a sequence up to $r$ interactions, 
without loss of generality, we still denote the corresponding function space in (\ref{eqn:anovadechi}) by $\HH$, which is the direct sum of some set of the orthogonal subspaces in the decomposition $\otimes_{j=1}^d\HH_1$. Define $\|\cdot\|_{\otimes_{j=1}^d\HH_1}$ as the norm on $\otimes_{j=1}^d\HH_1$ induced by component norms $\|\cdot\|_{\HH_1}$, and define $\|\cdot\|_\HH$ as the norm on $\HH$ by restricting $\|\cdot\|_{\otimes_{j=1}^d\HH_1}$ to $\HH$. Then $\HH$ is a RKHS equipped with $\|\cdot\|_\HH$.
The quadratic penalty $J(\cdot)$ in (\ref{scheme1}) is defined as a squared semi-norm on $\HH$ induced by a univariate penalty in $\HH_1$. For example, $\HH_1 = \WW_2^m(\XX_1)$, it is common to chose $J(\cdot)$ for penalizing only the smooth components of a function and an explicit form is given in Wahba \cite{wahba1990}.

Now we introduce some notations used in the proof. We define a family of the multi-index $\vec{\bnu}$ by
\begin{equation}
\label{def:Vbnu}
\begin{aligned}
V&  = \{{\vec{\bnu} = (\nu_1,\ldots,\nu_d)^\top \in\N^d},\\
& \quad\quad\quad\quad\quad \mbox{ where at most $r\geq 1$ of $\nu_k$s are not equal to } 1\}.
\end{aligned}
\end{equation}
which will be referred later since $f_0$ in the model (\ref{eqn:anovadecompfti}) is truncated up to $r$ interactions.
We write for two nonnegative sequences $\{a_n\}$ and $\{b_n\}$ as  $a_n\lesssim b_n$ (or $a_n\gtrsim b_n$) if there exists constant $c>0$ (or $c'>0$) which are independent of the problem parameters, such that $a_n\leq cb_n$ (or $a_n\geq c'b_n$) for all $n$. Let the maximizer of two scalars $\{a,b\}$ is denoted by $a\vee b$ and their the minimizer is denoted by $a\wedge b$.

\subsection{Proof of the minimax lower bound: Theorem \ref{theorem:lowerbdfNlambdaregrandom}}
\label{subsubsec:prooflowerrandom}

We establish the lower bound for the random design via Fano's lemma. 
It suffices to consider a special case where noises $\epsilon^{e_0}$ and $\epsilon^{e_j}$s are Gaussian with $\sigma_0=1$ and $\sigma_j=1$, and $\Pi^{e_0}$ and $\Pi^{e_j}$s are uniform, and $\HH_1$ is generated by periodic kernels. 

Let $N$ be a natural number whose value will be clear later.  We first derive the eigenvalue decay rate for kernel $K_d$ which generates the RKHS $\HH$.
For a given $\tau>0$, the number of multi-indices $\vec{\bnu}=(\nu_1,\ldots,\nu_r)\in\N^r$ satisfying $\nu_1^{-2m}\cdots\nu_r^{-2m}\geq \tau$ is the same as the number of multi-indices such that $\nu_1\cdots\nu_r\leq \tau^{-1/(2m)}$, which amounts to
\begin{equation}
\label{eqn:tau12mlogr1}
\begin{aligned}
\sum_{\nu_2\cdots\nu_r\leq \tau^{-1/(2m)}}\tau^{-1/(2m)}/(\nu_2\cdots\nu_r) & = \tau^{-1/(2m)}\left(\sum_{\nu\leq \tau^{-1/(2m)}}1/\nu\right)^{r-1} \\
& \asymp \tau^{-1/(2m)}(\log 1/\tau)^{r-1}.
\end{aligned}
\end{equation} 
Denote by $\lambda_N(K_d)$ the $N$th eigenvalues of $K_d$. 
By inverting (\ref{eqn:tau12mlogr1}), we obtain 
\begin{equation*}
\lambda_N(K_d)\asymp\left[N(\log N)^{1-r}\right]^{-2m}.
\end{equation*}
Hence, the multi-indices $\vec{\bnu}=(\nu_1,\ldots,\nu_r)\in\N^r$ satisfying $\nu_1\cdots \nu_r\leq N$ 
correspond to the first 
\begin{equation*}
c_0N(\log N)^{r-1}
\end{equation*} 
eigenvalues of $K_d$ for some constant $c_0$.
Let $b=\{b_{\vec{\bnu}}: \nu_1\cdots\nu_r\leq N\}\in\{0,1\}^{c_0N(\log N)^{r-1}}$ be a length-$\{c_0N(\log N)^{r-1}\}$ binary sequence, and $\{\tilde{\lambda}_{\vec{\bnu}}:\nu_1\cdots\nu_r\leq  N\}$ be the first $c_0N(\log N)^{r-1}$ eigenvalues of $K_d$. Denote by $\{\tilde{\lambda}_{\vec{\bnu}+c_0N(\log N)^{r-1}}:\nu_1\cdots\nu_r\leq  N\}$ the $\{c_0N(\log N)^{r-1}+1\}$th, $\{c_0N(\log N)^{r-1}+2\}$th,\ldots, $\{2c_0N(\log N)^{r-1}\}$th eigenvalues of $K_d$. 

For brevity, we only prove for the case $p=d$ and $r\geq 3$ while the other cases $p=d$, $r\leq 2$ and $0\leq p<d$ follow similar arguments. We deal with the differences among these cases for deterministic designs in Section \ref{sec:prooflowbdreg} of the supplementary material.
Write
\begin{align}
 & f_b(t_1,\ldots,t_r)= N^{-1/2+1/r}\sum_{\nu_1\cdots\nu_r\leq  N}b_{\vec{\bnu}}\left(1+\nu_1^2+\cdots+\nu_r^2\right)^{-1/2} \nonumber\\
 & \quad\quad\quad\quad\quad\quad\quad\quad\quad\quad\quad\quad\times\tilde{\lambda}_{\vec{\bnu}+c_0N(\log N)^{r-1}}^{1/2}\psi_{\vec{\bnu}+c_0N(\log N)^{r-1}}(t_1,\ldots,t_r),\nonumber
\end{align}
where $\psi_{\vec{\bnu}+c_0N(\log N)^{r-1}}(t_1,\ldots,t_r)$ are the corresponding eigenfunctions of $\tilde{\lambda}_{\vec{\bnu}+c_0N(\log N)^{r-1}}$ of $K_d$. Note that
\begin{align}
\|f_b\|_\HH^2 & = N^{-1+2/r}\sum_{\nu_1\cdots\nu_r\leq  N}b_{\vec{\bnu}}^2(1+\nu_1^2+\cdots+\nu_r^2)^{-1}\nonumber\\
& \leq N^{-1+2/r}\sum_{\nu_1\cdots\nu_r\leq  N}(1+\nu_1^2+\cdots+\nu_r^2)^{-1}\asymp 1,\nonumber
\end{align}
where the last step by Lemma \ref{lemma:intx1xrzxk1z2} in the supplementary material, and this implies $f_b(\cdot)\in\HH$.

By the Varshamov-Gilbert bound (see, e.g., \cite{tsybakovintroduction}), there exists a collection of binary sequences $\{b^{(1)},\ldots,b^{(M)}\}\subset\{0,1\}^{c_0N(\log N)^{r-1}}$ such that $M\geq 2^{c_0N(\log N)^{r-1}/8}$ and 
\begin{equation*}
H(b^{(l)},b^{(q)})\geq c_0N(\log N)^{r-1}/8,\quad \forall 1\leq l<q\leq M,
\end{equation*}
where $H(\cdot,\cdot)$ is the Hamming distance.
Then, for $b^{(l)},b^{(q)}\in\{0,1\}^{c_0N(\log N)^{r-1}}$, we have
\begin{align}
& \|f_{b^{(l)}}-f_{b^{(q)}}\|_{L_2}^2\nonumber\\
& \quad \geq N^{-1+2/r}(2N)^{-2m}\sum_{\nu_1\cdots\nu_r\leq  N}(1+\nu_1^2+\cdots+\nu_r^2)^{-1}\left[b^{(l)}_{\vec{\bnu}} - b^{(q)}_{\vec{\bnu}}\right]^2\nonumber\\
&\quad \geq N^{-1+2/r}(2N)^{-2m}\sum_{c_17N/8\leq \nu_1\cdots\nu_r\leq N}(1+\nu_1^2+\cdots+\nu_r^2)^{-1}\nonumber\\
& \quad = c_2 N^{-2m}\nonumber
\end{align}
for some constants $c_1$ and $c_2$, where the last step is by Lemma \ref{lemma:intx1xrzxk1z2} in the supplementary material.

On the other hand, for any $b^{(l)}\in\{b^{(1)},\ldots,b^{(M)}\}$ and by Lemma \ref{lemma:intx1xrzxk1z2},
\begin{align}
& \|f_{b^{(l)}}\|_{L_2}^2+ \sum_{j=1}^p\| \partial f_{b^{(l)}}/\partial t_j\|_{L_2}^2\nonumber\\
& \quad \leq N^{-1+2/r}\sum_{\nu_1\cdots\nu_r\leq  N}\nu_1^{-2m}\cdots \nu_r^{-2m}\left[b^{(l)}_{\vec{\bnu}} \right]^2\nonumber\\
& \quad \leq N^{-1+2/r}\sum_{\nu_1\cdots\nu_r\leq  N}\nu_1^{-2m}\cdots \nu_r^{-2m}\nonumber\\
& \quad = c_3 N^{-2m+2/r}(\log N)^{r-1}\nonumber
\end{align}
for some constant $c_3$.

A standard argument gives that the lower bound can be reduced to the error probability in a multi-way hypothesis test \cite{tsybakovintroduction}. Specifically, let $\Theta$ be a random variable uniformly distributed on $\{1,\ldots,M\}$. Note that
\begin{equation}
\label{eqn:inftildeffoinhh1}
\begin{aligned}
& \inf_{\tilde{f}}\sup_{f_0\in\HH}\P\left\{\|\tilde{f} - f_0\|_{L_2}^2\geq \frac{1}{4}\min_{b^{(l)}\neq b^{(q)}}\|f_{b^{(l)}} - f_{b^{(q)}}\|^2_{L_2}\right\}\\
& \quad\quad\quad\quad\quad\quad\quad\quad\quad\quad\quad\quad\quad\quad\quad\quad \geq \inf_{\widehat{\Theta}}\P\{\widehat{\Theta}\neq \Theta\},
\end{aligned}
\end{equation}
where the infimum on RHS is taken over all decision rules that are measurable functions of the data. By Fano's lemma, 
\begin{equation}
\label{eqn:fanolemhatthetap}
\begin{split}
& \P\left\{\widehat{\Theta}\neq \Theta|\bt_1^{e_0},\ldots,\bt_n^{e_0};\ldots;\bt_1^{e_p},\ldots,\bt_n^{e_p} \right\}\geq 1-\frac{1}{\log M}\\
& \quad \times\left[\mathbbm{1}_{\bt_1^{e_0},\ldots,\bt_n^{e_0};\ldots;\bt_1^{e_p},\ldots,\bt_n^{e_p}}(y_1^{e_0},\ldots,y_n^{e_0},\ldots,y_1^{e_p},\ldots,y_n^{e_p};\Theta)+\log 2\right],
\end{split}
\end{equation}
where $\mathbbm{1}_{\bt_1^{e_0},\ldots,\bt_n^{e_0};\ldots;\bt_1^{e_p},\ldots,\bt_n^{e_p}}(y_1^{e_0},\ldots,y_n^{e_0},\ldots,y_1^{e_p},\ldots,y_n^{e_p})$ is the mutual information between $\Theta$ and $\{y_1^{e_0},\ldots,y_n^{e_0},\ldots,y_1^{e_p},\ldots,y_n^{e_p}\}$ with the design points $\{\bt_1^{e_0},\ldots,\bt_n^{e_0};\ldots;\bt_1^{e_p},\ldots,\bt_n^{e_p}\}$ being fixed. We can derive that
\begin{equation}
\label{eqn:inftildeffoinhh3}
\begin{aligned}
& \E_{\bt_1^{e_0},\ldots,\bt_n^{e_0};\ldots;\bt_1^{e_p},\ldots,\bt_n^{e_p}}\\
& \quad\quad\quad\quad\quad \cdot\left[\mathbbm{1}_{\bt_1^{e_0},\ldots,\bt_n^{e_0};\ldots;\bt_1^{e_p},\ldots,\bt_n^{e_p}}\left(y_1^{e_0},\ldots,y_n^{e_0},\ldots,y_1^{e_p},\ldots,y_n^{e_p};\Theta\right)\right]\\
& \quad \leq \binom M2^{-1} \sum_{b^{(l)}\neq b^{(q)}} \E_{\bt_1^{e_0},\ldots,\bt_n^{e_0};\ldots;\bt_1^{e_p},\ldots,\bt_n^{e_p}} \mathcal{K}\left(\mathbf{P}_{f_{b^{(l)}}}| \mathbf{P}_{f_{b^{(q)}}}\right)\\
& \quad \leq \frac{n(p+1)}{2}\binom M2^{-1} \sum_{b^{(l)}\neq b^{(q)}} \E_{\bt_1^{e_0},\ldots,\bt_n^{e_0};\ldots;\bt_1^{e_p},\ldots,\bt_n^{e_p}}\|f_{b^{(l)}} - f_{b^{(q)}}\|_{*n}^2, 
\end{aligned}
\end{equation}
where $\mathcal{K}(\cdot|\cdot)$ is the Kullback-Leibler distance, $\mathbf{P}_{f}$ is conditional distribution of $y_i^{e_0}$ and $y_i^{e_j}$s given $\{\bt_1^{e_0},\ldots,\bt_n^{e_0};\ldots;\bt_1^{e_p},\ldots,\bt_n^{e_p}\}$, and the norm $\|\cdot\|_*$ is defined as
\begin{equation*}
\|f\|_{*n}^2 = \frac{1}{n(p+1)}\sum_{i=1}^n\left\{[f(\bt_i^{e_0})]^2+\sum_{j=1}^p[\partial f(\bt_i^{e_j})/\partial t_j]^2\right\},\quad\forall f:\XX_1^r\mapsto\R.
\end{equation*}
Thus,
\begin{equation}
\label{eqn:inftildeffoinhh4}
\begin{aligned}
& \E_{\bt_1^{e_0},\ldots,\bt_n^{e_0};\ldots;\bt_1^{e_p},\ldots,\bt_n^{e_p}}\\
& \quad\quad\quad \cdot \left[\mathbbm{1}_{\bt_1^{e_0},\ldots,\bt_n^{e_0};\ldots;\bt_1^{e_p},\ldots,\bt_n^{e_p}}(y_1^{e_0},\ldots,y_n^{e_0},\ldots,y_1^{e_p},\ldots,y_n^{e_p};\Theta)\right]\\
& \quad\leq \frac{n(p+1)}{2}\binom M2^{-1} \left.\sum_{b^{(l)}\neq b^{(q)}} \right\{\|f_{b^{(l)}} - f_{b^{(q)}}\|_{L_2}^2 \\
& \quad\quad\quad \quad\quad\quad \quad\quad\quad \quad\quad\quad \left. + \sum_{j=1}^p\|\partial f_{b^{(l)}}/\partial t_j - \partial f_{b^{(q)}}/\partial t_j\|_{L_2}^2\right\}\\
& \quad \leq \left.\frac{n(p+1)}{2} \max_{b^{(l)}\neq b^{(q)}} \vphantom{\sum_{j=1}^p}\right\{\|f_{b^{(l)}} - f_{b^{(q)}}\|_{L_2}^2 \\
& \quad\quad\quad \quad\quad\quad \quad\quad\quad \quad\quad\quad \left.+\sum_{j=1}^p\|\partial f_{b^{(l)}}/\partial t_j - \partial f_{b^{(q)}}/\partial t_j\|_{L_2}^2\right\}\\
& \quad \leq 2n(p+1)\max_{b^{(l)}\in\{b^{(1)},\ldots,b^{(M)}\}} \left\{\|f_{b^{(l)}}\|_{L_2}^2 + \sum_{j=1}^p\| \partial f_{b^{(l)}}/\partial t_j\|_{L_2}^2\right\}\\
& \quad \leq 2c_3n(p+1)N^{-2m+2/r}(\log N)^{r-1}.
\end{aligned}
\end{equation}
Now, (\ref{eqn:fanolemhatthetap}) yields
\begin{align}
& \quad\inf_{\tilde{f}}\sup_{f_0\in\HH}\P\left\{\|\tilde{f} - f_0\|_{L_2}^2\geq \frac{1}{4}c_2 N^{-2m}\right\}\nonumber\\
& \quad \geq\inf_{\widehat{\Theta}}\P\{\widehat{\Theta}\neq \Theta\}\nonumber\\
& \quad \geq 1-\frac{1}{\log M}\left[\E\mathbbm{1}_{\bt_1^{e_0},\ldots,\bt_n^{e_0};\ldots;\bt_1^{e_p},\ldots,\bt_n^{e_p}}(y_1^{e_0},\ldots,y_n^{e_0},\ldots,y_1^{e_p},\ldots,y_n^{e_p};\Theta)+\log 2\right]\nonumber\\
& \quad \geq 1-\frac{2c_3n(p+1)N^{-2m+2/r}(\log N)^{r-1} + \log 2}{c_0(\log 2)N(\log N)^{r-1}/8}.\nonumber
\end{align}
Taking $N=c_4n^{r/(2mr+r-2)}$ with an appropriate choice of $c_4$, we have
\begin{equation*}
\underset{n\rightarrow \infty}{\lim\sup}\inf_{\tilde{f}} \sup_{f_0\in\HH}\P\left\{\|\tilde{f} - f_0\|_{L_2}^2\geq C_1n^{-2mr/(2mr+r-2)}\right\}>0,
\end{equation*}
where $C_1$ only depends on $\sigma_0^2$, $\sigma_j^2$s, $m$, $r, p$, and $d$.
This completes the proof.


\subsection{Proof of the minimax upper bound: Theorem \ref{thm:mainupperrateestf0}}
\label{subsubsec:proofupperrrandom}

\paragraph{\textbf{Preliminaries for the proof}}

Denote by $\pi^{e_j}$ the density of the distribution $\Pi^{e_j}$, which by assumption is bounded away from 0 and infinity, $j=0,1,\ldots,p$. First we introduce a norm on $\HH$ for any $f\in\HH$,
\begin{equation}
\label{def:normrp+1}
\begin{aligned}
\|f\|_R^2 & = \frac{1}{p+1}\left[\frac{1}{\sigma_0^2}\int f^2(\bt)\pi^{e_0}(\bt) \right.\\
& \quad\quad\quad \quad\quad\quad \left.+ \sum_{j=1}^p \frac{1}{\sigma_j^2}\int\left\{\frac{\partial f(\bt)}{\partial t_j}\right\}^2\pi^{e_j}(\bt) \right] + J(f).
\end{aligned}
\end{equation}
Note that $\|\cdot\|_R$ is a norm since it is a quadratic form and is equal to zero if and only if $f=0$. Let $\langle\cdot,\cdot\rangle_R$ be the inner product associated with $\|\cdot\|_R$. The following lemma shows that $\|\cdot\|_R$ is well defined in $\HH$ and is equivalent to the RKHS norm $\|\cdot\|_\HH$. In particular, $\|g\|_R<\infty$ if and only if $\|g\|_\HH<\infty$. The proof of this lemma is given in Section \ref{subsec:proofoflemmanormequiv} in the supplementary material.
\begin{lemma}
\label{lemmanormrequiv}
The norm $\|\cdot\|_R$ is equivalent to $\|\cdot\|_{\HH}$ in $\HH$.
\end{lemma}
We introduce another norm $\|\cdot\|_0$ as follows: 
\begin{equation}
\label{eqn:|f|02def}
\|f\|_0^2  = \frac{1}{p+1}\left[\frac{1}{\sigma_0^2}\int f^2(\bt)\pi^{e_0}(\bt) + \sum_{j=1}^p \frac{1}{\sigma_j^2}\int\left\{\frac{\partial f(\bt)}{\partial t_j}\right\}^2\pi^{e_j}(\bt) \right].
\end{equation}
Based on (\ref{eqn:|f|02def}), we define a function space $F_0$ to be the direct sum of some set of the orthogonal subspaces in the decomposition of $\otimes_{j=1}^dL_2(\XX_1)$ as in (\ref{eqn:anovadechi}) and  equipped with the norm $\|\cdot\|_0$. Let $\langle\cdot,\cdot\rangle_0$ be the inner product associated with $\|\cdot\|_0$ in $F_0$.

With the above two norms, we introduce one additional notation. Denote the loss function in (\ref{scheme1}) by
\begin{equation*}
l_n(f) = \frac{1}{n(p+1)}\left[\frac{1}{\sigma_0^2}\sum_{i=1}^n\{f(\bt_i^{e_0})-y_i^{e_0}\}^2+\sum_{j=1}^p\frac{1}{\sigma_j^2}\sum_{i=1}^n\left\{\frac{\partial f(\bt_i^{e_j})}{\partial t_j}-y_i^{e_j}\right\}^2\right],
\end{equation*}
and write $l_{n\lambda}(f) = l_n(f)+\lambda J(f)$. Then the regularized estimator $\widehat{f}_{n\lambda} = \arg\min_{f\in\HH}{l_{n\lambda}(f)}$.
Denote the expected loss by $l_\infty(f)  = \E l_n(f) = \|f-f_0\|_0^2+1$,
and write $l_{\infty\lambda}(f) = l_\infty(f)+\lambda J(f)$. Note that $l_{\infty\lambda}(f)$ a positive quadratic form in $f\in\HH$ and hence it has a unique minimizer in $\HH$, 
\begin{equation*}
\bar{f}_{\infty\lambda} = \underset{f\in\HH}{\arg\min } l_{\infty\lambda}(f).
\end{equation*}
Thus, we decompose 
\begin{equation*}
\widehat{f}_{n\lambda} - f_0 = (\widehat{f}_{n\lambda} - \bar{f}_{\infty\lambda}) + (\bar{f}_{\infty\lambda} - f_0),
\end{equation*}
where $(\widehat{f}_{n\lambda} - \bar{f}_{\infty\lambda})$ is referred to the stochastic error and $(\bar{f}_{\infty\lambda} - f_0)$ is referred to the deterministic error. 
If data $Y^{e_0}$ and $Y^{e_j}$s  in (\ref{modelequation}) are observed without random noises as in deterministic computer experiments, then the total error is only the deterministic error with $\widehat{f}_{n\lambda} - f_0 = \bar{f}_{\infty\lambda} - f_0$.  For brevity, we omit the subscripts of $\bar{f}_{\infty\lambda}$ and $\widehat{f}_{n\lambda}$ hereafter if no confusion occurs.

\paragraph{\textbf{Outline of the proof}}

Before proceeding to the proof, we make two remarks on the setup of Theorem \ref{thm:mainupperrateestf0}. First, since the distributions $\Pi^{e_0}$ and $\Pi^{e_j}$s are known, by the inverse transform sampling, it suffices to consider uniform distributions. A detailed discussion on this inverse transform is given in Lemma \ref{lem:lineartransformuniform} in the supplementary material. Second, it suffices to consider $f_0$ having a periodic boundary on $\XX_1^d$ in the proof of this theorem. This is because $f_0$ is a tensor product function and each component function space is supported in a compact domain, thus we can smoothly extend $f_0$ to a larger compact support domain and achieve periodicity on the new boundary, for example, uniformly zero on the new boundary.  These two simplifications can make the proof easier to understand.

Recall the trigonometrical basis on $L_2(\XX_1)$ is $\psi_1(t)=1$, $\psi_{2\nu}(t) = \sqrt{2}\cos 2\pi\nu t$ and $\psi_{2\nu+1}(t) = \sqrt{2}\sin 2\pi\nu t$ for $\nu\geq 1$. Write 
\begin{equation}
\label{eqn:phivecbnudef}
\phi_{\vec{\bnu}}(t_1,\ldots,t_d) =\frac{\psi_{\nu_1}(t_1)\cdots\psi_{\nu_d}(t_d)}{\|\psi_{\nu_1}(t_1)\cdots\psi_{\nu_d}(t_d)\|_0}.
\end{equation} 
Since $f_0$ has a periodic boundary on $\XX_1^d$ and $\pi^{e_j}\equiv1$, we know $\{\phi_{\vec{\bnu}}(\bt):{\vec{\bnu}}\in V\}$, where $V$ in (\ref{def:Vbnu}) forms an orthogonal basis for $\HH$ in $\langle\cdot,\cdot\rangle_R$; an orthogonal system for $L_2(\XX_1^d)$; and an orthonormal basis for $F_0$ in $\langle\cdot,\cdot\rangle_0$, that is $\langle \phi_{\vec{\bnu}}(\bt),\phi_{\vec{\bmu}}(\bt)\rangle_0 = \delta_{\vec{\bnu}\vec{\bmu}}$, where $\delta_{\vec{\bnu}\vec{\bmu}}$ is Kronecker's delta.  Hence, any $f\in\HH$ has the decomposition 
\begin{equation}
\label{eqn:fsumnuthetanuphinvb}
f(t_1,\ldots,t_d) = \sum_{\vec{\bnu}\in V}f_{\vec{\bnu}}\phi_{\vec{\bnu}}(t_1,\ldots,t_d),\quad  \mbox{ where } f_{\vec{\bnu}} = \langle f(\bt),\phi_{\vec{\bnu}}(\bt)\rangle_0.
\end{equation} 
We denote a positive scalar series $\{\rho_{\vec{\bnu}}\}_{\bnu\in V}$ such that $\langle \phi_{\vec{\bnu}},\phi_{\vec{\bmu}}\rangle_R = (1+\rho_{\vec{\bnu}})\delta_{\vec{\bnu}\vec{\bmu}}$. Then,
\begin{equation}
\label{eqn:JgggR0}
J(f) = \langle f,f\rangle_R - \langle f,f\rangle_0 = \sum_{\vec{\bnu}\in V} \rho_{\vec{\bnu}}f_{\vec{\bnu}}^2.
\end{equation}

First, we analyze the deterministic error $(\bar{f}-f_0)$. By (\ref{eqn:fsumnuthetanuphinvb}), we write $f_0(\bt) = \sum_{\vec{\bnu}\in V}f_{\vec{\bnu}}^0\phi_{\vec{\bnu}}(\bt)$ and $\bar{f}(\bt) = \sum_{\vec{\bnu}\in V}\bar{f}_{\vec{\bnu}}\phi_{\vec{\bnu}}(\bt)$. Then, $l_\infty(f) = \sum_{\vec{\bnu}\in V} (f_{\vec{\bnu}} - f_{\vec{\bnu}}^0)^2 + 1$, and 
\begin{equation}
\label{eqn:barthetavebnubias}
\bar{f}_{\vec{\bnu}} = \frac{f_{\vec{\bnu}}^0}{1+\lambda\rho_{\vec{\bnu}}},\quad \vec{\bnu}\in V.
\end{equation} 
An upper bound of the deterministic error will be given in Lemma \ref{lem:barff0l2ar}.

Second, we analyze the stochastic error $(\widehat{f} - \bar{f})$. The existence the following Fr\'echet derivatives, for any $g,h\in\HH$, is guaranteed by Lemma \ref{lem:DlNfgDl2} in the supplementary material:
\begin{equation}
\label{eqn:frechetdlNfg}
\begin{aligned}
Dl_n(f)g & =  \frac{2}{n(p+1)}\left[\frac{1}{\sigma_0^2}\sum_{i=1}^n\{f(\bt_i^{e_0}) - y_i^{e_0}\}g(\bt_i^{e_0}) \right.\\
&\quad\quad\quad\quad\quad\quad \left.+ \sum_{j=1}^p\frac{1}{\sigma_j^2}\sum_{i=1}^n\left\{\frac{\partial f(\bt_i^{e_j})}{\partial t_j} - y_i^{e_j}\right\}\frac{\partial g(\bt_i^{e_j})}{\partial t_j}\right], 
\end{aligned}
\end{equation}
\begin{equation}
\label{eqn:dlinftyfg2p1}
\begin{aligned}
Dl_\infty(f)g & = \frac{2}{p+1}\left[\frac{1}{\sigma_0^2}\int\left\{f(\bt) - f_0(\bt)\right\}\frac{\partial g(\bt)}{\partial t_j}\pi^{e_j}(\bt)\right.\\
& \quad\quad\quad\quad\left. +\sum_{j=1}^p\frac{1}{\sigma_j^2}\int\left\{\frac{\partial f(\bt)}{\partial t_j} - \frac{\partial f_0(\bt)}{\partial t_j}\right\}\frac{\partial g(\bt)}{\partial t_j}\pi^{e_j}(\bt)\right],
\end{aligned}
\end{equation}
\begin{equation}
\label{eqn:d2lnfgh}
\begin{aligned}
D^2l_n(f)gh & =  \frac{2}{n(p+1)}\left[\frac{1}{\sigma_0^2}\sum_{i=1}^ng(\bt_i^{e_0})h(\bt_i^{e_0})\right.\\
& \quad\quad\quad\quad\quad\quad\quad\quad\quad\quad\left.+\sum_{j=1}^p\frac{1}{\sigma_j^2}\sum_{i=1}^n\frac{\partial g(\bt_i^{e_j})}{\partial t_j}\frac{\partial h(\bt^{e_j}_i)}{\partial t_j}\right],
\end{aligned}
\end{equation}
\begin{equation}
\label{eqn:d2linftyfgh}
\begin{aligned}
D^2l_\infty(f)gh & = \frac{2}{p+1}\left[\frac{1}{\sigma_0^2}\int g(\bt) h(\bt)\pi^{e_0}(\bt)\right.\\
& \quad\quad\quad\quad\quad\left. +\sum_{j=1}^p\frac{1}{\sigma_j^2}\int \frac{\partial g(\bt)}{\partial t_j}\frac{\partial h(\bt)}{\partial t_j}\pi^{e_j}(\bt)\right] =2\langle g,h\rangle_0,
\end{aligned}
\end{equation}
where $Dl_n(f)$, $Dl_\infty(f)$, $D^2l_n(f)g$, and $D^2l_\infty(f)g$ are bounded linear operators on $\HH$. By Riesz representation theorem, with slight abuse of notation, write 
\begin{align}
Dl_n(f)g & = \langle Dl_n(f),g\rangle_R, \quad Dl_\infty(f)g = \langle Dl_\infty(f),g\rangle_R,\nonumber\\
D^2l_n(f)gh & = \langle D^2l_n(f)g, h\rangle_R, \quad D^2l_\infty(f)gh = \langle D^2l_\infty(f)g, h\rangle_R.\nonumber
\end{align}
From \cite{oden2012introduction, weinberger1974variational}, there exists a bounded linear operator $U: F_0\mapsto \HH$ such that $U\phi_{\vec{\bnu}} = (1+\rho_{\vec{\bnu}})^{-1}\phi_{\vec{\bnu}}$ and $\langle f, Ug\rangle_R = \langle f,g\rangle_0$ for any $f\in \HH$ and $g\in F_0$, and the restriction of $U$ to $\HH$ is self-adjoint and positive definite. 
By (\ref{eqn:d2linftyfgh}), we further derive
\begin{align}
D^2l_{\infty\lambda}(f)\phi_{\vec{\bnu}}(\bt)  = 2(U+\lambda(I-U))\phi_{\vec{\bnu}}(\bt) = 2(1+\rho_{\vec{\bnu}})^{-1}(1+\lambda\rho_{\vec{\bnu}})\phi_{\vec{\bnu}}(\bt).\nonumber
\end{align}
Define that $G_\lambda\phi_{\vec{\bnu}}=\frac{1}{2}D^2l_{\infty\lambda}(\bar{f})\phi_{\vec{\bnu}}$. By the Lax-Milgram theorem, $G_\lambda: \HH\mapsto\HH$ has a bounded inverse $G_\lambda^{-1}$ on $\HH$, and 
\begin{equation}
\label{def:ginv}
G_\lambda^{-1}\phi_{\vec{\bnu}} = (1+\rho_{\vec{\bnu}})(1+\lambda\rho_{\vec{\bnu}})^{-1}\phi_{\vec{\bnu}}.
\end{equation}
Define 
\begin{equation*}
\tilde{f}^* = \bar{f} - \frac{1}{2}G_\lambda^{-1}Dl_{n\lambda}(\bar{f}).
\end{equation*} 
Then the stochastic error can be decomposed as
\begin{equation*}
\widehat{f} - \bar{f}= (\tilde{f}^* - \bar{f}) + (\widehat{f}- \tilde{f}^*).
\end{equation*}
The two terms on RHS will be studied separately and their upper bounds will be given in Lemma \ref{lem:tilfbarfl2a} and Lemma \ref{lem:boundhattildef}, respectively.

Finally, we define the following norm which is important in our later analysis, for $f\in\HH$
\begin{equation}
\label{eqn:fa2v1rho1}
\|f\|_{L_2(a)}^2=\sum_{{\vec{\bnu}}\in V} \left(1+\frac{\rho_{\vec{\bnu}}}{\|\phi_{\vec{\bnu}}\|_{L_2}^2}\right)^a f_{\vec{\bnu}}^2\|\phi_{\vec{\bnu}}\|_{L_2}^2, \quad \mbox{ for } 0\leq a\leq 1,
\end{equation}
where $f_{\vec{\bnu}} = \langle f, \phi_{\vec{\bnu}}\rangle_0$. By direct calculations, note that when $a=0$ this norm coincides with $\|\cdot\|_{L_2}$ on $F_0$, and when $a=1$ this norm is equivalent to $\|\cdot\|_R$ on $\HH$.

\paragraph{\textbf{Details of the proof}} Now we give the details following the outline above.
First, we present an upper bound of the deterministic error $(\bar{f} - f_0)$.
\begin{lemma} 
\label{lem:barff0l2ar}
For any $0\leq a\leq 1$, the deterministic error  satisfies
\begin{equation*}
\|\bar{f} - f_0\|_{L_2(a)}^2 = 
\begin{cases}
O\left\{\lambda^{1-a}J(f_0)\right\} \quad & \mbox{ when } 0\leq p< d,\\
O\left\{\lambda^{\frac{(1-a)mr}{mr-1}}J(f_0)\right\} \quad & \mbox{ when } p=d.
\end{cases}
\end{equation*}
\end{lemma}

\begin{proof}
For any $0\leq a\leq1$, by (\ref{eqn:JgggR0}) and (\ref{eqn:barthetavebnubias}), we have
\begin{equation}
\label{eqn:lambda2jf0supnu}
\begin{aligned}
\|\bar{f} - f_0\|_{L_2(a)}^2 &  = \sum_{{\vec{\bnu}}\in V} \left(1+\frac{\rho_{\vec{\bnu}}}{\|\phi_{\vec{\bnu}}\|_{L_2}^2}\right)^a\left(\frac{\lambda\rho_{\vec{\bnu}}}{1+\lambda\rho_{\vec{\bnu}}}\right)^2(f_{\vec{\bnu}}^0)^2\|\phi_{\vec{\bnu}}\|_{L_2}^2\\
& \leq \lambda^2\sup_{{\vec{\bnu}}\in V}\frac{(1+\rho_{\vec{\bnu}}/\|\phi_{\vec{\bnu}}\|_{L_2}^2)^a\rho_{\vec{\bnu}}\|\phi_{\vec{\bnu}}\|_{L_2}^2}{(1+\lambda\rho_{\vec{\bnu}})^2}\sum_{{\vec{\bnu}}\in V}\rho_{\vec{\bnu}}(f_{\vec{\bnu}}^0)^2\\
&\lesssim \lambda^2J(f_0)\sup_{{\vec{\bnu}}\in V}\frac{(\prod_{k=1}^d\nu_k^{2m})^{1+a}}{(1+\sum_{j=1}^p\nu_j^2+\lambda\prod_{k=1}^d\nu_k^{2m})^2}. 
\end{aligned}
\end{equation}
Write
\begin{equation*}
B_{\lambda}({\vec{\bnu}}) = \frac{(\prod_{k=1}^d\nu_k^{2m})^{1+a}}{(1+\sum_{j=1}^p\nu_j^2+\lambda\prod_{k=1}^d\nu_k^{2m})^2},\quad {\vec{\bnu}}\in V.
\end{equation*}
We discuss $B_{\lambda}({\vec{\bnu}})$ for $0\leq p\leq d-1$ and $p=d$ separately.

For $0\leq p\leq d-1$, since ${\vec{\bnu}}\in V$, there are at most $r$ of $\nu_1,\ldots,\nu_d$ not equal to 1. Suppose for any $x=\prod_{k=1}^d\nu_k^{-2m}>0$ fixed. Then $B_\lambda({\vec{\bnu}})$ is maximized by letting $\sum_{j=1}^p\nu_j^2$ be as small as possible, which implying $\nu_1=\nu_2=\cdots=\nu_p=1$. Then
\begin{equation}
\label{eqn:supnublambdanu-a}
\begin{aligned}
\underset{{\vec{\bnu}}\in V}{\sup}B_{\lambda}({\vec{\bnu}}) & \asymp \sup_{(\nu_{p+1},\ldots,\nu_{(p+r)\wedge d})^\top\in\N^{r\wedge (d-p)}}\frac{\prod_{k=p+1}^{(p+r)\wedge d}\nu_k^{2m(1+a)}}{(1+\lambda\prod_{k=p+1}^{(p+r)\wedge d}\nu_k^{2m})^2}\\
& \asymp \sup_{x>0}\frac{x^{-(1+a)}}{(1+\lambda x^{-1})^2} \asymp \lambda^{-(a+1)}, 
\end{aligned}
\end{equation}
where the last step is achieved when $x \asymp \lambda$.

For $p=d$, since ${\vec{\bnu}}\in V$ and by the symmetry of coordinates $v_1,\ldots,v_d$, assume 
that all indices except $v_1,\ldots,v_r$ being 1. Letting $z=\prod_{j=1}^r\nu_j^{-2m}>0$, then
\begin{align}
\underset{{\vec{\bnu}}\in V}{\sup}B_{\lambda}({\vec{\bnu}})  \asymp \sup_{z>0}\frac{z^{-(1+a)}}{(z^{-1/mr}+\lambda z^{-1})^2} \asymp \lambda^{\frac{2-(1+a)mr}{mr-1}},\label{eqn:lamdba2a1amr}
\end{align}
where the last step is achieved when $z\asymp \lambda^{mr/(mr-1)}$. Combining (\ref{eqn:lambda2jf0supnu}), (\ref{eqn:supnublambdanu-a}) and (\ref{eqn:lamdba2a1amr}), we complete the proof.
\end{proof}


Second, we show an upper bound of $(\tilde{f}^*-\bar{f})$, which is a part of the stochastic error.
\begin{lemma}
\label{lem:tilfbarfl2a}
When $0\leq p<d$, we have for any $0\leq a< 1 - 1/2m$,
\begin{equation*}
\|\tilde{f}^* - \bar{f}\|_{L_2(a)}^2 =
O_\P\left\{n^{-1}\lambda^{-(a+1/2m)}[\log(1/\lambda)]^{(d-p)\wedge r-1}\right\}.
\end{equation*}
When $p=d$, we have for any $0\leq a\leq 1$,
\begin{align}
& \|\tilde{f}^* - \bar{f}\|_{L_2(a)}^2 \nonumber\\
= &
\begin{cases}
O_\P\left\{n^{-1}\lambda^{\frac{mr}{1-mr}\left(a+\frac{r-2}{2mr}\right)}\right\},   \mbox{ if } r\geq 3;\\
O_\P\left\{n^{-1}\log(1/\lambda)\right\},  \mbox{ if } r=2, a=0;\quad O_\P\left\{n^{-1}\right\}, \mbox{ if } r=2, 0<a\leq 1;\\
O_\P\left\{n^{-1}\right\},  \mbox{ if } r=1,a<\frac{1}{2m};  \quad O_\P\left\{n^{-1}\log(1/\lambda)\right\},  \mbox{ if } r=1,a=\frac{1}{2m}; \\
O_\P\left\{n^{-1}\lambda^{\frac{1-2ma}{2m-2}}\right\},  \mbox{ if } r=1, a>\frac{1}{2m}.
\end{cases}\nonumber
\end{align}
\end{lemma}
\begin{proof}
Notice that $Dl_{n,\lambda}(\bar{f}) = Dl_{n,\lambda}(\bar{f}) - Dl_{\infty,\lambda}(\bar{f}) = Dl_n(\bar{f}) - Dl_\infty(\bar{f})$. Hence, for any $g\in\HH$, 
\begin{equation}
\label{eqn:dlnlambdang0}
\begin{aligned}
& \E\left[\frac{1}{2}Dl_{n,\lambda}(\bar{f})g\right]^2 = \E\left[\frac{1}{2}Dl_n(\bar{f})g - \frac{1}{2}Dl_\infty(\bar{f})g\right]^2\\
& \quad = \frac{1}{n(p+1)^2}\text{Var}\left[\frac{1}{\sigma_0^2}\left\{\bar{f}(\bt^{e_0}) - Y^{e_0}\right\}g(\bt^{e_0})+\sum_{j=1}^p\frac{1}{\sigma_{j}^2}\left\{\frac{\partial\bar{f}(\bt^{e_j})}{\partial t_j} - Y^{e_j}\right\}\frac{\partial g(\bt^{e_j})}{\partial t_j}\right]\\
& \quad\leq \frac{1}{n(p+1)}\left[\frac{1}{\sigma_0^4}\E\left\{\bar{f}(\bt^{e_0})-f_0(\bt^{e_0})\right\}^2\{g(\bt^{e_0})\}^2+\frac{1}{\sigma_{0}^2}\E\{g(\bt^{e_0})\}^2\right.\\
& \quad\quad\quad \left.+\sum_{j=1}^p\frac{1}{\sigma_j^4}\E\left\{\frac{\partial \bar{f}(\bt^{e_j})}{\partial t_j}-\frac{\partial f_0(\bt^{e_j})}{\partial t_j}\right\}^2\left\{\frac{\partial g(\bt^{e_j})}{\partial t_j}\right\}^2+\sum_{j=1}^p\frac{1}{\sigma_{j}^2}\E\left\{\frac{\partial g(\bt^{e_j})}{\partial t_j}\right\}^2\right]\\
& \quad \leq \frac{1}{n(p+1)}\left[ \frac{1}{\sigma_0^4}c_{K}^{2d}\|\bar{f}-f_0\|_R^2\E\left\{g(\bt^{e_0})\right\}^2+\frac{1}{\sigma_{0}^2}\E\left\{g(\bt^{e_0})\right\}^2\right.\\
& \quad\quad\quad \left.+\sum_{j=1}^p\frac{1}{\sigma_j^4}c_{K}^{2d}\|\bar{f}-f_0\|_R^2\E\left\{\frac{\partial g(\bt^{e_j})}{\partial t_j}\right\}^2+\sum_{j=0}^p\frac{1}{\sigma_{j}^2}\E\left\{\frac{\partial g(\bt^{e_j})}{\partial t_j}\right\}^2\right]\\
& \quad  \lesssim n^{-1}\|g\|_0^2,
\end{aligned}
\end{equation}
where the third step is by Lemma \ref{lemmanormrequiv} and Lemma \ref{lem:bounddejft} in the supplementary material, and the last step is by Lemma \ref{lem:barff0l2ar} and the definition of the norm $\|\cdot\|_0$.
From the definition of $G_\lambda^{-1}$ in (\ref{def:ginv}), we have that $\forall g\in\HH$, 
\begin{equation*}
\left\|G_\lambda^{-1}g\right\|_{L_2(a)}^2 = \sum_{{\vec{\bnu}}\in V}\left(1+\frac{\rho_{\vec{\bnu}}}{\|\phi_{\vec{\bnu}}\|_{L_2}^2}\right)^a\left(1+\lambda\rho_{\vec{\bnu}}\right)^{-2}\|\phi_{\vec{\bnu}}\|^2_{L_2}\langle g,\phi_{\vec{\bnu}}\rangle_R^2.
\end{equation*} 
Then by the definition of $\tilde{f}^*$, we have
\begin{align}
& \E\|\tilde{f}^* - \bar{f}\|_{L_2(a)}^2  = \E \left\|\frac{1}{2}G_\lambda^{-1}Dl_{n\lambda}(\bar{f})\right\|_{L_2(a)}^2\nonumber\\
& = \frac{1}{4}\E\left[\sum_{{\vec{\bnu}}\in V}\left(1+\frac{\rho_{\vec{\bnu}}}{\|\phi_{\vec{\bnu}}\|_{L_2}^2}\right)^a(1+\lambda \rho_{\vec{\bnu}})^{-2}\|\phi_{\vec{\bnu}}\|^2_{L_2}\langle Dl_{n\lambda}(\bar{f}), \phi_{\vec{\bnu}}\rangle_R^2\right]\nonumber\\
& \leq \sum_{{\vec{\bnu}}\in V}\left(1+\frac{\rho_{\vec{\bnu}}}{\|\phi_{\vec{\bnu}}\|_{L_2}^2}\right)^a(1+\lambda\rho_{\vec{\bnu}})^{-2}\|\phi_{\vec{\bnu}}\|^2_{L_2}\E\left[\frac{1}{2}Dl_{n\lambda}(\bar{f})\phi_{\vec{\bnu}}\right]^2\nonumber\\
& \lesssim n^{-1}\sum_{{\vec{\bnu}}\in V}\left(1+\frac{\rho_{\vec{\bnu}}}{\|\phi_{\vec{\bnu}}\|_{L_2}^2}\right)^a\left(1+\lambda\rho_{\vec{\bnu}}\right)^{-2}\|\phi_{\vec{\bnu}}\|^2_{L_2}\|\phi_{\vec{\bnu}}\|_{0}^2\nonumber\\
& \asymp n^{-1}N_a(\lambda), \nonumber
\end{align}
where the fourth step is by (\ref{eqn:dlnlambdang0}) and the last step is because of $\|\phi_{\vec{\bnu}}\|_0=1$, $\|\phi_{\vec{\bnu}}\|^2_{L_2} \asymp (1+\sum_{j=1}^p\nu_j^2)^{-1}$, $\rho_{\vec{\bnu}}\asymp (1+\sum_{j=1}^p\nu_j^2)^{-1}\prod_{k=1}^d\nu_k^{2m}$, and $N_a(\lambda)$ is defined in Lemma \ref{lemma:varthetamultivariatetensorrank} in the supplementary material. Hence, by Lemma \ref{lemma:varthetamultivariatetensorrank}, we complete the proof.
\end{proof}

Then, we give an upper bound of $(\widehat{f} - \tilde{f}^*)$, which is another part of the stochastic error.
Since $l_{n\lambda}(f)$ is a quadratic form of $f$, the Taylor expansion of $Dl_{n\lambda}(\widehat{f}) = 0$ at $\bar{f}$ gives
\begin{equation*}
Dl_{n\lambda}(\bar{f}) + D^2l_{n\lambda}(\bar{f})(\widehat{f}-\bar{f}) = 0,
\end{equation*}
and by the definition of $\tilde{f}^*$ and $G_\lambda$, we have
\begin{equation*}
Dl_{n\lambda}(\bar{f}) + D^2l_{\infty\lambda}(\bar{f})(\tilde{f}^* - \bar{f}) = 0.
\end{equation*}
Thus, $G_\lambda(\widehat{f} - \tilde{f}^*) = \frac{1}{2}D^2l_\infty(\bar{f})(\widehat{f} - \bar{f}) - \frac{1}{2}D^2l_{n}(\bar{f})(\widehat{f}-\bar{f})$, and
\begin{equation}
\label{eqn:hatfnlambdatildef}
\widehat{f} - \tilde{f}^* = G_\lambda^{-1}\left[\frac{1}{2}D^2l_\infty(\bar{f})(\widehat{f}- \bar{f}) - \frac{1}{2}D^2l_{n}(\bar{f})(\widehat{f} - \bar{f})\right].
\end{equation}

\begin{lemma} 
\label{lem:boundhattildef}
If $n^{-1}\lambda^{-(2a+3/2m)}[\log(1/\lambda)]^{r-1}\rightarrow 0$ and $1/2m< a<(2m-3)/4m$, we have for any $0\leq c\leq a+1/m$,
\begin{equation*} 
\|\widehat{f} - \tilde{f}^*\|_{L_2(c)}^2 = o_\P\left\{\|\tilde{f}^* - \bar{f}\|_{L_2(c)}^2\right\}.
\end{equation*} 
\end{lemma}
\begin{proof}
A  sufficient condition for this lemma is that for any $1/(2m)< a<(2m-3)/(4m)$ and $0\leq c\leq a+1/m$,
\begin{equation}
 \label{eqn:anoformhatftildefc}
\begin{aligned}
& \|\widehat{f} - \tilde{f}^*\|_{L_2(c)}^2\\
&  = 
\begin{cases}
O_\P\left\{n^{-1}\lambda^{-(c+a+1/2m)}[\log(1/\lambda)]^{r\wedge (d-p)-1}\right\}\|\widehat{f} - \bar{f}\|_{L_2(a+1/m)}^2, &\mbox{if } 0\leq p<d,\\
O_\P\left\{n^{-1}\lambda^{\frac{mr}{1-mr}\left(a+c+\frac{r-2}{2mr}\right)}\right\}\|\widehat{f} - \bar{f}\|_{L_2(a+1/m)}^2,  & \mbox{if } p=d, r\geq 3,\\
O_\P\left\{n^{-1}\right\}\|\widehat{f} - \bar{f}\|_{L_2(a+1/m)}, &\mbox{if } p=d, r=2,\\
O_\P\left\{n^{-1}\lambda^{\frac{1-2m(a+c)}{2m-2}}\right\}\|\widehat{f} - \bar{f}\|_{L_2(a+1/m)},   &\mbox{if } p=d, r=1.
 \end{cases}
\end{aligned}
\end{equation}
This is because once (\ref{eqn:anoformhatftildefc}) established, by letting $c=a+1/m$ and under the assumption $n^{-1}\lambda^{-(2a+3/2m)}[\log(1/\lambda)]^{r-1}\rightarrow 0$, we have
\begin{equation*}
\|\widehat{f} - \tilde{f}^*\|_{L_2(a+1/m)}^2 = o_\P(1)\|\widehat{f} - \bar{f}\|_{L_2(a+1/m)}^2.
\end{equation*}
By the triangle inequality, we have $\|\tilde{f}^* - \bar{f}\|_{L_2(a+1/m)}  \geq \|\widehat{f} - \bar{f}\|_{L_2(a+1/m)} - \|\widehat{f} - \tilde{f}^*\|_{L_2(a+1/m)} = [1-o_\P(1)]\|\widehat{f} - \bar{f}\|_{L_2(a+1/m)}$, which implies $\|\widehat{f}  - \bar{f}\|_{L_2(a+1/m)}^2 = O_\P\{\|\tilde{f}^* - \bar{f}\|_{L_2(a+1/m)}^2\}$. Thus by  (\ref{eqn:anoformhatftildefc}) and Lemma \ref{lem:tilfbarfl2a}, we complete the proof.

We now are in the position to prove (\ref{eqn:anoformhatftildefc}). 
For any $0\leq c\leq a+1/m$, by (\ref{eqn:hatfnlambdatildef}), we have
\allowdisplaybreaks
\begin{align}
& \|\widehat{f} - \tilde{f}^*\|_{L_2(c)}^2  \nonumber\\
& = \sum_{{\vec{\bnu}}\in V}\left(1+\frac{\rho_{\vec{\bnu}}}{\|\phi_{\vec{\bnu}}\|_{L_2}^2}\right)^c(1+\lambda\rho_{\vec{\bnu}})^{-2}\|\phi_{\vec{\bnu}}\|_{L_2}^2\nonumber\\
&   \quad\quad\quad\quad\quad\quad\times\left[\frac{1}{2}D^2l_\infty(\bar{f})(\widehat{f} - \bar{f})\phi_{\vec{\bnu}} - \frac{1}{2}D^2l_{n}(\bar{f})(\widehat{f} - \bar{f})\phi_{\vec{\bnu}}\right]^2\nonumber\\
& \quad\leq \sum_{{\vec{\bnu}}\in V}\left(1+\frac{\rho_{\vec{\bnu}}}{\|\phi_{\vec{\bnu}}\|_{L_2}^2}\right)^c(1+\lambda\rho_{\vec{\bnu}})^{-2}\|\phi_{\vec{\bnu}}\|_{L_2}^2\nonumber\\
& \quad \times\frac{1}{p+1}\left\{\left[\frac{\sum_{i=1}^n(\widehat{f}-\bar{f})(\bt_i^{e_0}) \phi_{\vec{\bnu}}(\bt_i^{e_0})}{n\sigma_0^2} - \frac{\int(\widehat{f}-\bar{f})(\bt)\phi_{\vec{\bnu}}(\bt)}{\sigma_0^2}\right]^2 \right.\label{eqn:1p1ndejf}\\
&  \quad\quad\quad\left.+\sum_{j=1}^p\left[\frac{\sum_{i=1}^n\frac{\partial (\widehat{f}-\bar{f})}{\partial t_j}(\bt_i^{e_j})\frac{\partial \phi_{\vec{\bnu}}}{\partial t_j}(\bt_i^{e_j})}{n\sigma_j^2}- \frac{\int \frac{\partial (\widehat{f} - \bar{f})(\bt)}{\partial t_j}\frac{\partial \phi_{\vec{\bnu}}(\bt)}{\partial t_j}}{\sigma_j^2}\right]^2\right\}.\nonumber
\end{align}
Denote $g_{j}(\bt) = \frac{1}{\sigma_j^2}\frac{\partial (\widehat{f} - \bar{f})}{\partial t_j} \frac{\partial \phi_{\vec{\bnu}}}{\partial t_j}$ and $g_0(\bt) = \frac{1}{\sigma_0^2}(\widehat{f} - \bar{f})\phi_{\vec{\bnu}}$.
Hence, we can do the expansion on the basis $\{\phi_{\vec{\bmu}}\}_{{\vec{\bmu}}\in \N^d}$, 
\begin{equation}
\label{eqn:gjbtsumbmund}
g_{j}(\bt)=\sum_{{\vec{\bmu}}\in \N^d} Q^{j}_{{\vec{\bmu}}}\phi_{\vec{\bmu}}(\bt),\quad  \mbox{ where } Q^{j}_{\vec{\bmu}} = \langle g_j(\bt),\phi_{\vec{\bmu}}(\bt)\rangle_0.
\end{equation}
Unlike (\ref{eqn:fsumnuthetanuphinvb}) with the multi-index $\vec{\bnu}\in V$, we require ${\vec{\bmu}}\in \N^d$ in (\ref{eqn:gjbtsumbmund}) since now $g_{j}(\bt)$ is a product function.
By Cauchy-Schwarz inequality,
\begin{align}
& \left[\frac{1}{n\sigma_j^2}\sum_{i=1}^n\frac{\partial (\widehat{f} - \bar{f})}{\partial t_j}(\bt_i^{e_j})\frac{\partial \phi_{\vec{\bnu}}}{\partial t_j}(\bt_i^{e_j})  - \frac{1}{\sigma_j^2}\int \frac{\partial (\widehat{f} - \bar{f})(\bt)}{\partial t_j}\frac{\partial \phi_{\vec{\bnu}}(\bt)}{\partial t_j}\right]^2\nonumber\\
& \quad =  \left[\sum_{{\vec{\bmu}}\in \N^d} Q^{j}_{{\vec{\bmu}}}\left(\frac{1}{n}\sum_{i=1}^n\phi_{\vec{\bmu}}(\bt^{e_j}_i)-\int\phi_{\vec{\bmu}}(\bt)\right)\right]^2\nonumber\\
& \quad \leq \left[\sum_{{\vec{\bmu}}\in \N^d} (Q^{j}_{{\vec{\bmu}}})^2\left(1+\frac{\rho_{\vec{\bmu}}}{\|\phi_{\vec{\bmu}}\|_{L_2}^2}\right)^{a}\|\phi_{\vec{\bmu}}\|_{L_2}^2\right]\nonumber\\
& \quad\quad\times
\left[\sum_{{\vec{\bmu}}\in \N^d}\left(1+\frac{\rho_{\vec{\bmu}}}{\|\phi_{\vec{\bmu}}\|_{L_2}^2}\right)^{-a}\|\phi_{\vec{\bmu}}\|_{L_2}^{-2}\left(\frac{1}{n}\sum_{i=1}^n\phi_{\vec{\bmu}}(\bt^{e_j}_i) - \int\phi_{\vec{\bmu}}(\bt)\right)^2\right].\label{eqn:cgammnua1mf35}
\end{align}
For brevity, we write $f(\bt) = \partial f/\partial t_0$.
By Lemma \ref{lemma:inequanund1rhonua} in the supplementary material we have that if $a> 1/2m$, the sum of the first part in (\ref{eqn:cgammnua1mf35}) over $j=0,\ldots,p$ is bounded by
\begin{equation}
\label{eqn:boundgjl2a1m}
\begin{aligned}
&  \sum_{j=0}^p\sum_{{\vec{\bmu}}\in \N^d} \left(1+\frac{\rho_{\vec{\bmu}}}{\|\phi_{\vec{\bmu}}\|_{L_2}^2}\right)^a\|\phi_{\vec{\bmu}}\|_{L_2}^2\left\langle\frac{\partial (\widehat{f} - \bar{f})}{\partial t_j}\frac{\partial \phi_{\vec{\bnu}}}{\partial t_j},\phi_{\vec{\bmu}}\right\rangle_0^2\\
&  \lesssim  \|\widehat{f} - \bar{f}\|_{L_2(a+1/m)}^2\sum_{j=0}^p\sum_{{\vec{\bmu}}\in \N^d}\left(1+\frac{\rho_{\vec{\bmu}}}{\|\phi_{\vec{\bmu}}\|_{L_2}^2}\right)^a\|\phi_{\vec{\bmu}}\|_{L_2}^2\left\langle \frac{\partial \phi_{\vec{\bnu}}}{\partial t_j},\phi_{\vec{\bmu}}\right\rangle_0^2\\
& \lesssim \|\widehat{f} - \bar{f}\|_{L_2(a+1/m)}^2\left(1+\frac{\rho_{\vec{\bnu}}}{\|\phi_{\vec{\bnu}}\|_{L_2}^2}\right)^a\|\phi_{\vec{\bnu}}\|_{L_2}^2\left(1+\sum_{j=1}^p\nu_j^2\right)\\
& \asymp \|\widehat{f} - \bar{f}\|_{L_2(a+1/m)}^2 \left(1+\frac{\rho_{\vec{\bnu}}}{\|\phi_{\vec{\bnu}}\|_{L_2}^2}\right)^a.
\end{aligned}
\end{equation}
The second part of (\ref{eqn:cgammnua1mf35}) can be bounded by
\begin{equation}
\label{eqn:enuganmmacnf}
\begin{aligned}
& \E\left[\sum_{{\vec{\bmu}}\in \N^d}\left(1+\frac{\rho_{\vec{\bmu}}}{\|\phi_{\vec{\bmu}}\|_{L_2}^2}\right)^{-a}\|\phi_{\vec{\bmu}}\|_{L_2}^{-2}\left(\frac{1}{n}\sum_{i=1}^n\phi_{\vec{\bmu}}(\bt_i^{e_j}) - \int\phi_{\vec{\bmu}}(\bt)\right)^2\right]\\
&\quad  \leq \sum_{{\vec{\bmu}}\in \N^d}\left(1+\frac{\rho_{\vec{\bmu}}}{\|\phi_{\vec{\bmu}}\|_{L_2}^2}\right)^{-a}\|\phi_{\vec{\bmu}}\|_{L_2}^{-2}\left(\frac{1}{n}\int\phi_{\vec{\bmu}}^2(\bt)\right)\\
& \quad \asymp n^{-1}\sum_{{\vec{\bmu}}\in \N^d}\left(1+\frac{\rho_{\vec{\bmu}}}{\|\phi_{\vec{\bmu}}\|_{L_2}^2}\right)^{-a} \lesssim n^{-1}\sum_{{\vec{\bmu}}\in \N^d}\mu_1^{-2ma}\cdots\mu_d^{-2ma}\\
& \quad \leq n^{-1}\left(\sum_{\mu_1=1}^\infty\mu_1^{-2ma}\right)^d\asymp n^{-1},
\end{aligned}
\end{equation}
where the third step uses $\rho_{\vec{\bmu}}/\|\phi_{\vec{\bmu}}\|_{L_2}^2\asymp \mu_1^{2m}\cdots\mu_d^{2m}$, and the fourth step holds for $a> 1/2m$. Combing  (\ref{eqn:cgammnua1mf35}), (\ref{eqn:boundgjl2a1m}) and (\ref{eqn:enuganmmacnf}), we have for 
$a>1/2m$,
\begin{align}
& \E\left\{\left[\frac{1}{n\sigma_0^2}\sum_{i=1}^n(\widehat{f}-\bar{f})(\bt_i^{e_0}) \phi_{\vec{\bnu}}(\bt_i^{e_0}) - \frac{1}{\sigma_0^2}\int(\widehat{f}-\bar{f})(\bt)\phi_{\vec{\bnu}}(\bt)\right]^2 \right.\nonumber\\
& \left.+\sum_{j=1}^p\left[\frac{1}{n\sigma_j^2}\sum_{i=1}^n\frac{\partial (\widehat{f}-\bar{f})}{\partial t_j}(\bt_i^{e_j})\frac{\partial \phi_{\vec{\bnu}}}{\partial t_j}(\bt_i^{e_j})- \sum_{j=1}^p\frac{1}{\sigma_j^2}\int \frac{\partial (\widehat{f} - \bar{f})(\bt)}{\partial t_j}\frac{\partial \phi_{\vec{\bnu}}(\bt)}{\partial t_j}\right]^2\right\}\nonumber\\
&\quad \lesssim \frac{1}{n} \|\widehat{f} - \bar{f}\|_{L_2(a+1/m)}^2\left(1+\frac{\rho_{\vec{\bnu}}}{\|\phi_{\vec{\bnu}}\|_{L_2}^2}\right)^a.
\label{eqn:e1p1n-1hatfbarfl2a1m}
\end{align}
Therefore, if $1/2m< a<(2m-3)/4m$ and $0\leq c\leq a+1/m$,  (\ref{eqn:1p1ndejf})  and (\ref{eqn:e1p1n-1hatfbarfl2a1m}) imply that 
\begin{align}
 \E\|\widehat{f} -\tilde{f}^*\|_{L_2(c)}^2  \lesssim n^{-1}\|\widehat{f} - \bar{f}\|_{L_2(a+1/m)}^2N_{a+c}(\lambda).\nonumber
 \end{align}
 By Lemma \ref{lemma:varthetamultivariatetensorrank} in the supplementary material, we complete the proof for  (\ref{eqn:anoformhatftildefc}) and this lemma.
\end{proof}

Last, we combine Lemma \ref{lem:barff0l2ar}, Lemma \ref{lem:tilfbarfl2a} and Lemma \ref{lem:boundhattildef} and get the following proposition.

\begin{proposition}
\label{thm:convergenceforhatfnlambda}
Under the conditions of Theorem \ref{theorem:lowerbdfNlambdaregrandom} and assuming the distributions $\Pi^{e_0}$ and $\Pi^{e_j}$s are known.
If $1/2m< a<(2m-3)/4m$, $m>2$, and $n^{-1}\lambda^{-(2a+3/2m)}[\log(1/\lambda)]^{r-1}\rightarrow 0$, then for any $c\in[0,a+1/m]$, the $\widehat{f}$ given by (\ref{scheme1}) satisfies, when  $0\leq p<d$,
\begin{align}
 \|\widehat{f}- f_0\|_{L_2(c)}^2  = O\{\lambda^{1-c}J(f_0)\}  + O_\P\left\{n^{-1}\lambda^{-(c+1/2m)}[\log(1/\lambda)]^{r\wedge (d-p)-1}\right\},\nonumber
\end{align}
and when $p=d$,
\begin{align}
& \|\widehat{f} - f_0\|_{L_2(c)}^2 \nonumber\\
& = 
\begin{cases}
O\left\{\lambda^{\frac{(1-c)mr}{mr-1}}J(f_0)\right\}  + O_\P\left\{n^{-1}\lambda^{\frac{mr}{1-mr}\left(c+\frac{r-2}{2mr}\right)}\right\}\quad   \mbox{ if } r\geq 3,\\
O\left\{\lambda^{\frac{2m}{2m-1}}J(f_0)\right\}  + O_\P\left\{n^{-1}\log(1/\lambda)\right\}  \mbox{ if } r=2, c=0,\\
O\left\{\lambda^{\frac{2(1-c)m}{2m-1}}J(f_0)\right\}  + O_\P\left\{n^{-1}\lambda^{\frac{2mc}{1-2m}}\right\}  \mbox{ if } r=2, c>0,\\
O\left\{\lambda^{\frac{(1-c)m}{m-1}}J(f_0)\right\}  + O_\P\left\{n^{-1}\right\}  \mbox{ if } r=1, c<\frac{1}{2m},\\
O\left\{\lambda^{\frac{2m-1}{2(m-1)}}J(f_0)\right\}  + O_\P\left\{n^{-1}\log(1/\lambda)\right\}  \mbox{ if } r=1, c=\frac{1}{2m},\\
O\left\{\lambda^{\frac{(1-c)m}{m-1}}J(f_0)\right\}  + O_\P\left\{n^{-1}\lambda^{\frac{1-2mc}{2m-2}}\right\}  \mbox{ if } r=1, c>\frac{1}{2m}.
 \end{cases}\nonumber
\end{align}
\end{proposition}

Many results on the regularized estimator $\widehat{f}$ can be derived from Proposition \ref{thm:convergenceforhatfnlambda} including Theorem \ref{thm:mainupperrateestf0}. In fact, consider for $p=d$ and $r\geq 3$, by letting $\lambda \asymp n^{-\frac{2mr-2}{(2m+1)r-2}}$, $a=1/2m+\epsilon$ for some $\epsilon>0$ and $c=0$, then the condition
$ n^{-1}\lambda^{-(2a+3/2m)}[\log(1/\lambda)]^{r-1}\rightarrow 0$
is equivalent to 
\begin{equation}
\label{eqn:-1frac5mr1}
-1 + \frac{5(mr-1)}{2m^2r+mr-2m}<0, 
\end{equation}
and $m>2$ is sufficient for (\ref{eqn:-1frac5mr1}).
Thus the conditions for Proposition \ref{thm:convergenceforhatfnlambda} are satisfied. Similarly, we can verify that when $p=d$ and $r=2$,  $\lambda\asymp [n(\log n)]^{-(2m-1)/2m}$ satisfies the conditions for Proposition \ref{thm:convergenceforhatfnlambda}. When $p=d$ and $r=1$, $\lambda\lesssim n^{-(m-1)/m}$ satisfies the conditions for the above Proposition. When $0\leq p\leq d-r$, $\lambda\asymp [n(\log n)^{1-r}]^{-2m/(2m+1)}$  satisfies the conditions for the above Proposition,  as well as
when $d-r< p< d$ by letting $\lambda\asymp [n(\log n)^{1+p-d}]^{-2m/(2m+1)}$. This completes the proof for Theorem \ref{thm:mainupperrateestf0}.

\subsection{Proof of Corollary \ref{theorem:estboundderiva}}
This corollary can be directly derived from Proposition \ref{thm:convergenceforhatfnlambda}. 
Observe that
\begin{equation*}
\int_{\XX_1^d}\left[\frac{\partial^{d}\widehat{f}_{n\lambda}(\bt)}{\partial t_1\cdots\partial t_d}-\frac{\partial^d f_0(\bt)}{\partial t_1\cdots \partial t_d}\right]^2d\bt\asymp \|\widehat{f}_{n\lambda} - f_0\|_{L_2(1/m)}.
\end{equation*}
If $d-r<p<d$, we let $c=a=1/m$ and $\lambda\asymp [n(\log n)^{1+p-d}]^{-2m/(2m+1)}$  in Proposition \ref{thm:convergenceforhatfnlambda}, then the condition
$n^{-1}\lambda^{-(2a+3/2m)}[\log(1/\lambda)]^{r-1}\rightarrow 0$
is equivalent to 
\begin{equation}
\label{eqn:1frac12m1less0}
-1 + 7/(2m+1)<0,
\end{equation}
and $m>3$ is sufficient for (\ref{eqn:1frac12m1less0}). Thus the condition for Proposition \ref{thm:convergenceforhatfnlambda} are satisfied, and Proposition \ref{thm:convergenceforhatfnlambda} yields the rate of convergence for $ \|\widehat{f}_{n\lambda} - f_0\|_{L_2(1/m)}$ is
\begin{equation*}
O_\P\left([n(\log n)^{1+p-d}]^{-2(m-1)/(2m+1)}\right).
\end{equation*}
Similarly, if $0\leq p\leq d-r$, we let $\lambda\asymp [n(\log n)^{1-r}]^{-2m/(2m+1)}$; if $p=d$ and $r\geq 3$, let $\lambda\asymp n^{-2(mr-1)/(2mr+r-2)}$; if $p=d$ and $r=2$, let $\lambda\asymp n^{-(2m-1)/2m}$; if $p=d$ and $r=1$, let $\lambda\asymp n^{-(2m-2)/(2m-1)}$, then the conditions for Proposition \ref{thm:convergenceforhatfnlambda} will be satisfied. This completes the proof for Corollary \ref{theorem:estboundderiva}.


\section{Discussion}
This paper is the first to study the  minimax optimal rates for nonparametric estimation when data from first-order partial derivatives are available. We study the function estimation and partial derivative estimations with functional ANOVA models while there are few existing results in literature concerning the partial derivative estimations.

In Theorem \ref{theorem:lowerbdfNlambdareg}, Theorem \ref{thm;deterdesgnupperbdgeneralpdreg}, Theorem \ref{theorem:lowerbdfNlambdaregrandom} and Theorem \ref{thm:mainupperrateestf0}, we assume that all component functions are from a common RKHS $\HH_1$.  We also assume the eigenvalues decay at the polynomial rate, which is true for Sobolev kernels and other widely used kernels. More general settings are also interesting, for example, component RKHS are different, and the eigenvalues decay at different polynomial rates or  even exponentially, and the method of regularization in (\ref{scheme1}) uses other goodness of fit measures.
It would of course be of great interest to extend our results to a broad class of bounded linear functionals and to multivariate function spaces without tensor product structure. We leave these open for future studies.

\appendix

\section*{Acknowledgements}
X. Dai would like to thank Yuhua Zhu and Cuize Han for helpful discussions. We thank Grace Wahba for very helpful comments on an early version of the manuscript. 

\newpage


\appendix
\section{Proofs of technical results}

This appendix consists of five parts. In Section \ref{sec:frechetderiv}, we give a brief review on Fr\'echet derivative which is used in (\ref{eqn:frechetdlNfg}), (\ref{eqn:dlinftyfg2p1}), (\ref{eqn:d2lnfgh}) and (\ref{eqn:d2linftyfgh}) in the main text. 
 In Section \ref{subsec:reglattproof}, we give the proofs for results with deterministic designs in Section \ref{sec:minmaxriskregularlat}.
In Section \ref{sec:proofthmupperlimdderiv}, we prove the results of estimating partial derivatives in Section \ref{sec:minmaxriskpartder}.
We present some key lemmas used for the proofs in Section \ref{subsec:keylemmaforproofs}. All auxiliary technical lemmas are deferred to Section \ref{subsec:auxiliarytechlem}.


\subsection{Fr\'echet derivative of an operator}
\label{sec:frechetderiv}

Let $X$ and $Y$ be the normed linear spaces.  The Fr\'echet derivative of an operator $F: X\mapsto Y$ is a bounded linear operator $DF(a):X\mapsto Y$ with 
\begin{equation*}
\lim_{h\rightarrow 0,h\in X}\frac{\|F(a+h)-F(a)-DF(a)h\|_Y}{\|h\|_X} = 0.
\end{equation*}
For illustration, if $F(a+h) - F(a) = Lh+R(a,h)$ with a linear operator $L$ and $\|R(a,h)\|_Y/\|h\|_X\rightarrow0$ as $h\rightarrow 0$, then by the above definition, $L=DF(a)$ is the Fr\'echet derivative of $F(\cdot)$. The reader is referred to elementary functional analysis  textbooks such as Cartan \cite{cartan1971differential}  for a thorough investigation on Fr\'echet derivative.

\begin{lemma}
\label{lem:DlNfgDl2}
With the norm $\|\cdot\|_R$ in (\ref{def:normrp+1}), the first order Fr\'echet derivative of the functional $l_n(\cdot)$ for any $f,g\in\HH$ is
\begin{align}
Dl_n(f)g & =  \frac{2}{n(p+1)}\left[\frac{1}{\sigma_0^2}\sum_{i=1}^n\{f(\bt_i^{e_0}) - y_i^{e_0}\}g(\bt_i^{e_0}) \right.\nonumber\\
&\quad\quad\quad\quad\quad\quad\quad\quad \left.+ \sum_{j=1}^p\frac{1}{\sigma_j^2}\sum_{i=1}^n\left\{\frac{\partial f(\bt_i^{e_j})}{\partial t_j} - y_i^{e_j}\right\}\frac{\partial g(\bt_i^{e_j})}{\partial t_j}\right].\nonumber
\end{align}
The second order Fr\'echet derivative of $l_n(\cdot)$ for any $f,g,h\in\HH$ is
\begin{equation*}
\begin{aligned}
D^2l_n(f)gh  & =  \frac{2}{n(p+1)}\left[\frac{1}{\sigma_0^2}\sum_{i=1}^ng(\bt_i^{e_0})h(\bt_i^{e_0})\right.\\
& \quad\quad\quad\quad\quad\quad\quad\quad\left.+\sum_{j=1}^p\frac{1}{\sigma_j^2}\sum_{i=1}^n\frac{\partial g(\bt_i^{e_j})}{\partial t_j}\frac{\partial h(\bt^{e_j}_i)}{\partial t_j}\right].
\end{aligned}
\end{equation*}
\end{lemma}
\begin{proof}
By direct calculations, we have
\begin{align}
& l_{n}(f+g) - l_n(f) = \frac{2}{n(p+1)}\left[\frac{1}{\sigma_0^2}\sum_{i=1}^n\{f(\bt_i^{e_0}) - y_i^{e_0}\}g(\bt_i^{e_0}) \right.\nonumber\\
&\quad\quad\quad\quad\quad\quad\quad \left.+ \sum_{j=1}^p\frac{1}{\sigma_j^2}\sum_{i=1}^n\left\{\frac{\partial f(\bt_i^{e_j})}{\partial t_j} - y_i^{e_j}\right\}\frac{\partial g(\bt_i^{e_j})}{\partial t_j}\right] + R_n(f,g),\nonumber
\end{align}
where
\begin{equation*}
\begin{aligned}
R_n(f,g) & = \frac{1}{n(p+1)}\left[\frac{1}{\sigma_0^2}\sum_{i=1}^n g^2(\bt_i^{e_0}) + \sum_{j=1}^p\frac{1}{\sigma_j^2}\sum_{i=1}^n \left\{\frac{\partial g(\bt_i^{e_j})}{\partial t_j}\right\}^2\right] \\
& = \|g\|_0^2 + O(n^{-1/2}).
\end{aligned}
\end{equation*}
Note that $|R_n(f,g)|/\|g\|_R \rightarrow 0$ as $\|g\|_R\rightarrow 0$ and $n^{1/2}\|g\|_R\rightarrow \infty$. This proves the form of $Dl_n(f)g$ in the lemma.
For the second order Fr\'echet derivative, note that
\begin{align}
& Dl_n(f+h)g- Dl_n(f)g \nonumber\\
& \quad\quad =\frac{2}{n(p+1)}\left[\frac{1}{\sigma_0^2}\sum_{i=1}^ng(\bt_i^{e_0})h(\bt_i^{e_0})+\sum_{j=1}^p\frac{1}{\sigma_j^2}\sum_{i=1}^n\frac{\partial g(\bt_i^{e_j})}{\partial t_j}\frac{\partial h(\bt^{e_j}_i)}{\partial t_j}\right],\nonumber
\end{align}
which is linear in $h$. By definition of  Fr\'echet derivatives, we conclude the form of $D^2l_n(f)gh$ in the lemma.
\end{proof} 
We remark that following a similar derivation in the above proof, we can obtain the first and the second order Fr\'echet derivatives of the functional $l_\infty(\cdot)$ in (\ref{eqn:dlinftyfg2p1}) and (\ref{eqn:d2linftyfgh}), respectively.

\subsection{Proofs for Section \ref{sec:minmaxriskregularlat}: regular lattices}
\label{subsec:reglattproof}
For brevity,  we shall assume the regular lattice (\ref{eqn:reglatticeofti}) is $l_1=\cdots =l_d = l$ and $n=l^d$. The more general case can be showed similarly.
Write 
\begin{equation}
\label{eqn:trigbasisexp}
\psi_1(t)=1, \quad \psi_{2\nu}(t) = \sqrt{2}\cos 2\pi\nu t, \quad \psi_{2\nu+1}(t) = \sqrt{2}\sin 2\pi\nu t,
\end{equation}
for $\nu\geq 1$. Since $f_0$ has periodic boundaries on $\XX_1^d$, $\{\psi_\nu(t)\}_{\nu\geq 1}$ forms an orthonormal system in $L_2(\XX_1)$ and an eigenfunction system for $K$.
For a $d$-dimensional vector ${\vec{\bnu}}=(\nu_1,\ldots,\nu_d)\in\N^d$, write 
\begin{equation}
\label{eqn:gammabnuprodgamma}
\psi_{{\vec{\bnu}}}(\bt) = \psi_{\nu_1}(t_1)\cdots\psi_{\nu_d}(t_d) \quad \mbox{ and }\quad
\lambda_{\vec{\bnu}} = \lambda_{\nu_1}\lambda_{\nu_2}\cdots\lambda_{\nu_d},
\end{equation}
where $\lambda_{\nu_k}$s and $\psi_{\nu_k}(t_k)$s are defined in (\ref{eqn:mercerdecaykttgamma}) with $k=1,\ldots,d$.  
Then, any function $f(\cdot)$ in $\HH$ admits the Fourier expansion $f(\bt) = \sum_{{\vec{\bnu}}\in\N^d}\theta_{\vec{\bnu}} \psi_{\vec{\bnu}}(\bt)$, where $\theta_{\vec{\bnu}} = \langle f(\bt),\psi_{\vec{\bnu}}(\bt)\rangle_{L_2}$,
and $J(f) = \sum_{\vec{\bnu}\in\N^d}\lambda_{\vec{\bnu}}^{-1}\theta_{\vec{\bnu}}^2$. 
We also write $f_0(\bt) = \sum_{{\vec{\bnu}}\in\N^d}\theta_{\vec{\bnu}}^0\psi_{\vec{\bnu}}(\bt)$. 

By Page 23 of Wahba \cite{wahba1990}, it is known that 
\begin{equation*}
l^{-1}\sum_{i=1}^{l} \psi_{\mu}(i/l)\psi_\nu(i/l) = 
\begin{cases}
1,\quad & \mbox{if } \mu=\nu= 1,\ldots,l,\\
0,\quad & \mbox{if } \mu\neq \nu, \mu, \nu= 1,\ldots,l.
\end{cases}
\end{equation*}
Define
\begin{equation*}
\vec{\psi}_{{\vec{\bnu}}} = (\psi_{{\vec{\bnu}}}(\bt_1), \ldots, \psi_{{\vec{\bnu}}}(\bt_n))^\top,
\end{equation*} 
where $\{\bt_1,\ldots,\bt_n\}$ are design points in (\ref{eqn:reglatticeofti}). Thus, we have
\begin{equation*}
\langle \vec{\psi}_{{\vec{\bnu}}}, \vec{\psi}_{{\vec{\bmu}}} \rangle_n = 
\begin{cases}
1, \quad &\mbox{ if } \nu_k=\mu_k=1,\ldots, l; k=1,\ldots, d,\\
0, \quad &\mbox{ if there exists some } k \mbox{ such that } \nu_{k}\neq \mu_{k},
\end{cases}
\end{equation*}
where $\langle\cdot,\cdot\rangle_n$ is the empirical inner product in $\R^n$. 
This implies that $\{\vec{\psi}_{{\vec{\bnu}}}| \nu_k=1,\ldots, l; k=1,\ldots, d\}$ form an orthogonal basis in $\R^n$ with respect to the empirical norm $\|\cdot\|_n$. 
Denote the observed data vectors by  $\by^{e_0} = (y_1^{e_0},\ldots,y_n^{e_0})^\top$ and $\by^{e_j} = (y_1^{e_j},\ldots,y_n^{e_j})^\top$, and write
\begin{equation}
\label{eqn:distransdataz}
\begin{cases}
z_{{\vec{\bnu}}}^{e_0} & = \langle\by^{e_0}, \vec{\psi}_{{\vec{\bnu}}}\rangle_n,\\
z_{\nu_1,\ldots, 2\nu_k-1,\ldots,\nu_d}^{e_j} & =  (2\pi)^{-1}\langle \by^{e_j},  \vec{\psi}_{\nu_1,\ldots, 2\nu_k,\ldots, \nu_d}\rangle_n,\\
z_{\nu_1,\ldots, 2\nu_k,\ldots,\nu_d}^{e_j} & = - (2\pi)^{-1}\langle \by^{e_j}, \vec{\psi}_{\nu_1,\ldots, 2\nu_k-1,\ldots, \nu_d}\rangle_n,
\end{cases}
\end{equation}
for $\nu_k=1,\ldots, l$ and $k=1,\ldots,d$. Then $z_{{\vec{\bnu}}}^{e_0} = \tilde{\theta}_{\vec{\bnu}}^0 + \delta^{e_0}_{\vec{\bnu}}$ and $z_{{\vec{\bnu}}}^{e_j} = \nu_j\tilde{\theta}_{\vec{\bnu}}^0 + \delta^{e_j}_{\vec{\bnu}}$, where $\tilde{\theta}_{\vec{\bnu}}^0 = \theta_{\vec{\bnu}}^0 + \sum_{\mu_k\geq l, k=1,\ldots,d}\theta_{\vec{\bmu}}^0\langle \vec{\psi}_{\vec{\bnu}}, \vec{\psi}_{\vec{\bmu}}\rangle_n$, and 
$\delta^{e_0}_{\vec{\bnu}}$, $\delta^{e_j}_{\vec{\bnu}}$ are all independent with mean 0 and variance $\sigma_0^2/n$ and $\sigma_j^2/n$, respectively.

\subsubsection{Proof of minimax lower bound: Theorem \ref{theorem:lowerbdfNlambdareg}}
\label{sec:prooflowbdreg}

We now prove the lower bound for estimating functions under the regular lattice. By the data transformation (\ref{eqn:distransdataz}), it suffices to show the optimal rate in a special case
\begin{equation}
\label{eqn:ynuejwhitenoise}
\begin{cases}
z_{\vec{\bnu}}^{e_0} & = \theta^0_{\vec{\bnu}} + \delta^{e_0}_{\vec{\bnu}}, \\
z_{\vec{\bnu}}^{e_j} & = \nu_j\theta^0_{\vec{\bnu}} + \delta^{e_j}_{\vec{\bnu}},\quad \mbox{ for } 1\leq j\leq p,
\end{cases}
\end{equation}
where $\delta^{e_j}_{\vec{\bnu}}\sim \NN(0,\sigma_{j}^2/n)$ are independent. For any $\vec{\bnu}\in\N^d$, if we have the prior that $| \tilde{\theta}^0_{\vec{\bnu}} |\leq \pi_{\vec{\bnu}}$, then the minimax linear estimator is 
\begin{equation*}
\widehat{\theta}_{\vec{\bnu}}^L = \frac{\sigma_0^{-2}z^{e_0}_{\vec{\bnu}}+\sum_{j=1}^p\sigma_j^{-2}\nu_jz^{e_j}_{\vec{\bnu}}}{n^{-1}\pi_{\vec{\bnu}}^{-2} + \sigma_0^{-2}+\sum_{j=1}^p\sigma_j^{-2}\nu_j^2},
\end{equation*}
and the minimax linear risk is
\begin{equation*}
n^{-1}\left[n^{-1}\pi_{\vec{\bnu}}^{-2} + \sigma_0^{-2}+\sum_{j=1}^p\sigma_j^{-2}\nu_j^2\right]^{-1}.
\end{equation*}
By Lemma 6 and Theorem 7 in Donoho, Liu and MacGibbon \cite{donoho1990minimax}, if $\sigma_j^2$s are known, 
the minimax risk of estimating $\theta_{\vec{\bnu}}^0$ under the model (\ref{eqn:ynuejwhitenoise}) is larger than $80\%$ of the minimax linear risk of the hardest rectangle subproblem, and the latter linear risk is
\begin{equation}
\label{eqn:linearriskRLsumvecv}
R^{L} = n^{-1}\max_{\sum_{\vec{\bnu}\in V}(1+\lambda_{\vec{\bnu}})\pi_{\vec{\bnu}}^2=1} \sum_{\vec{\bnu}\in V} \left[n^{-1}\pi_{\vec{\bnu}}^{-2} + \sigma_0^{-2}+\sum_{j=1}^p\sigma_j^{-2}\nu_j^2\right]^{-1},
\end{equation}
where $\lambda_{\vec{\bnu}}$ is the product of eigenvalues in (\ref{eqn:gammabnuprodgamma}) and recall that the set  $V$ is defined in (\ref{def:Vbnu}).

We use the Lagrange multiplier method to find $\pi_{\vec{\bnu}}^2$ for solving (\ref{eqn:linearriskRLsumvecv}). Let $a$ be the scalar multiplier and  define
\begin{equation*}
L(\pi_{\vec{\bnu}}^2,a) = \sum_{\vec{\bnu}\in V} \left[n^{-1}\pi_{\vec{\bnu}}^{-2} + \sigma_0^{-2}+\sum_{j=1}^p\sigma_j^{-2}\nu_j^2\right]^{-1}  - a (1+\lambda_{\vec{\bnu}})\pi_{\vec{\bnu}}^2.
\end{equation*}
Taking partial derivative with respect to $\pi_{\vec{\bnu}}^2$ gives
\begin{equation*}
\frac{\partial L}{\partial \pi_{\vec{\bnu}}^2} = n^{-1}\left[n^{-1}+\left(\sigma_0^{-2}+\sum_{j=1}^p\sigma_j^{-2}\nu_j^2\right)\pi_{\vec{\bnu}}^2\right]^{-2} - a(1+\lambda_{\vec{\bnu}}) = 0.
\end{equation*}
This implies
\begin{equation*}
\widehat{\pi}_{\vec{\bnu}}^2 = \left(\sigma_0^{-2}+\sum_{j=1}^p\sigma_j^{-2}\nu_j^2\right)^{-1}\left[b(1+\lambda_{\vec{\bnu}})^{-1/2}- n^{-1}\right]_+,
\end{equation*}
where $b=(na)^{-1/2}$. On one hand, plugging the above formula into the constraint $\sum_{\vec{\bnu}\in V}(1+\lambda_{\vec{\bnu}})\pi_{\vec{\bnu}}^2=1$ gives
\begin{equation*}
\sum_{\vec{\bnu}\in V} \prod_{k=1}^d\nu_k^{2m}\left(\sigma_0^{-2}+\sum_{j=1}^p\sigma_j^{-2}\nu_j^2\right)^{-1}\left[b\prod_{k=1}^d\nu_k^{-m}- n^{-1}\right]_+ \asymp 1.
\end{equation*}
By restricting $\prod_{k=1}^d\nu_k\leq (nb)^{1/m}$, this becomes
\begin{equation}
\label{eqn:sumbnuinVprodk1d}
\begin{aligned}
& \sum_{\vec{\bnu}\in V, \prod_{k=1}^d\nu_k\leq (nb)^{1/m}}\left(\sigma_0^{-2}+\sum_{j=1}^p\sigma_j^{-2}\nu_j^2\right)^{-1}\\
& \quad\quad\quad\quad\quad\quad\quad\quad\quad\quad\times \left(b\prod_{k=1}^d\nu_k^{m}- n^{-1}\prod_{k=1}^d\nu_k^{2m}\right) \asymp 1.
\end{aligned}
\end{equation}
On the other hand, the linear risk in (\ref{eqn:linearriskRLsumvecv}) can be written as
\begin{equation}
\label{eqn:explicitexpofRL}
\begin{aligned}
R^L \asymp n^{-1}& \sum_{\vec{\bnu}\in V,\prod_{k=1}^d\nu_k\leq (nb)^{1/m}}\left(1-\frac{1}{nb}\prod_{k=1}^d\nu_k^m\right)\\
& \quad\quad\quad\quad\quad\quad\quad\quad\times \left(\sigma_0^{-2}+\sum_{j=1}^p\sigma_j^{-2}\nu_j^2\right)^{-1}.
\end{aligned}
\end{equation}
We discuss for $R^L$ in the above (\ref{eqn:explicitexpofRL}) under the condition (\ref{eqn:sumbnuinVprodk1d}) for three cases with $0\leq p\leq d-r$, $d-r<p<d$ and $p=d$.

If $0\leq p\leq d-r$, since $\vec{\bnu}\in V$, there are at most $r$ of $\nu_1,\ldots,\nu_d$ not equal to 1, which implies that the number of combinations of non-1 indices being summed in (\ref{eqn:sumbnuinVprodk1d}) is no greater than $C_d^1+C_d^2+\cdots+C_d^r<\infty$. Due to the term $(\sigma_0^{-2}+\sum_{j=1}^p\sigma_j^{-2}\nu_j^2)^{-1}$, the largest terms of the summation (\ref{eqn:sumbnuinVprodk1d}) over ${\vec{\bnu}}\in V$ correspond to the combinations of indices where as fewer $\nu_1,\ldots,\nu_p$ being summed as possible, for example, $v_k\equiv 1$ for $k\leq p$ and $k> p+r$, and $(\nu_{p+1},\ldots,\nu_{p+r})\in\N^r$ are non-1. Thus, (\ref{eqn:sumbnuinVprodk1d}) is equivalent to 
\begin{equation*}
\sum_{\prod_{k=1}^r\nu_{p+k}\leq (nb)^{1/m}}\left(b\prod_{k=1}^r\nu_{p+k}^{m}- n^{-1}\prod_{k=1}^r\nu_{p+k}^{2m}\right) \asymp 1.
\end{equation*}
Using the integral approximation, we have
\begin{equation*}
\int_{\prod_{k=1}^rx_{p+k}\leq (nb)^{1/m},x_{p+k}\geq1} \left(b\prod_{k=1}^rx_{p+k}^m - \frac{1}{n}\prod_{k=1}^rx_{p+k}^{2m}\right)dx_{p+1}\cdots dx_{p+r}\asymp 1.
\end{equation*}
By letting $z_{j} = \prod_{1\leq k\leq j}x_{p+k}$, $j=1,2,\ldots,r$, we have
\begin{equation*}
\int_1^{(nb)^{1/m}}\left[\int_1^{z_r}\cdots\int_{1}^{z_2} \left(bz_r^m - \frac{1}{n}z_r^{2m}\right)z_1^{-1}\cdots z_{r-1}^{-1}dz_1\cdots dz_{r-1}\right]dz_r\asymp 1,
\end{equation*}
where the LHS term is the order of $n^{(m+1)/m}b^{(2m+1)/m}[\log(nb)]^{r-1}$ and hence
\begin{equation}
\label{eqn:p0d-rbn-m12m1log}
b\asymp n^{-(m+1)/(2m+1)}(\log n)^{-m(r-1)/(2m+1)}.
\end{equation}
The linear risk in (\ref{eqn:explicitexpofRL}) becomes
\begin{align}
R^L & \asymp n^{-1}\int_{\prod_{k=1}^rx_{p+k}\leq (nb)^{1/m},x_{p+k}\geq 1}\left(1-\frac{1}{nb}\prod_{k=1}^rx_{p+k}^m\right)\nonumber\\
& \asymp [\log(nb)]^{r-1}n^{-1+1/m}b^{1/m} \asymp  [n(\log n)^{1-r}]^{-2m/(2m+1)},\nonumber
\end{align}
where the last step is by (\ref{eqn:p0d-rbn-m12m1log}).

If $d-r< p< d$, as discussed in the previous case, 
the number of combinations of non-1 indices being summed is finite, and 
the largest terms of the summation (\ref{eqn:sumbnuinVprodk1d}) over ${\vec{\bnu}}\in V$ correspond to the combinations of indices where as fewer than $\nu_1,\ldots,\nu_p$ being summed as possible, for example, $v_k\equiv 1$ for $k\leq d-r$, and $(\nu_{d-r+1},\ldots, \nu_d)\in\N^r$ are non-1. Thus, (\ref{eqn:sumbnuinVprodk1d}) is equivalent to 
\begin{equation*}
\begin{aligned}
& \sum_{\prod_{k=1}^r\nu_{d-r+k}\leq (nb)^{1/m}}\left(b\prod_{k=1}^r\nu_{d-r+k}^{m}- n^{-1}\prod_{k=1}^r\nu_{d-r+k}^{2m}\right) \\
&\quad\quad\quad\quad\quad\quad\quad\quad\quad\quad\quad\quad\quad\quad\quad\quad\times \left(1+\sum_{j=d-r+1}^p\nu_j^2\right)^{-1} \asymp 1.
\end{aligned}
\end{equation*}
Using the integral approximation, we have
\begin{align}
& \int_{\prod_{k=1}^rx_{d-r+k}\leq (nb)^{1/m},x_{d-r+k}\geq1}  \left(b\prod_{k=1}^rx_{d-r+k}^m - n^{-1}\prod_{k=1}^rx_{d-r+k}^{2m}\right)\nonumber\\
& \quad\quad\quad\quad \quad\quad\quad\quad\quad \times \left(1+\sum_{j=d-r+1}^px_j^2\right)^{-1}dx_{d-r+1}\cdots dx_{d}\asymp 1.\nonumber
\end{align}
By letting $z_{j} =x_{p+1}x_{p+2}\cdots x_j$, $j=p+1,\ldots, d$, we get
\begin{align*}
1& \asymp \int_{x_{d-r+1}\cdots x_pz_d\leq (nb)^{1/m}} \left[\int_1^{z_d}\cdots\int_{1}^{z_{p+2}} \right.\\
& \quad\quad \left(bx_{d-r+1}^m\cdots x_p^mz_d^m - \frac{1}{n}x_{d-r+1}^{2m}\cdots x_p^{2m}z_{d}^{2m}\right)z_{p+1}^{-1}\cdots z_{d-1}^{-1}\nonumber\\
& \quad\quad\left.\times\left(1+x_{d-r+1}^2+\cdots +x_p^2\right)^{-1}dz_{p+1}\cdots dz_{d-1}\vphantom{\int_1^{z_d}}\right]dx_{d-r+1}\cdots dx_pdz_d\nonumber\\
& =\int_{x_{d-r+1}\cdots x_pz_d\leq (nb)^{1/m}}bx_{d-r+1}^m\cdots x_p^mz_d^m\left(1 - \frac{1}{nb}x_{d-r+1}^{m}\cdots x_p^{m}z_{d}^{m}\right)\nonumber\\
& \quad\quad\quad\times (\log z_d)^{d-p-1} \left(1+x_{d-r+1}^2+\cdots +x_p^2\right)^{-1}dx_{d-r+1}\cdots dx_pdz_d\nonumber\\
& \asymp [\log (nb)]^{d-p-1} n^{1+1/m}b^{2+1/m},\nonumber
\end{align*}
where the last step is by Lemma \ref{lemma:intx1xrzxk1z1} in Section \ref{subsec:auxiliarytechlem}.
Hence,
\begin{equation}
\label{eqn:p0d-rbn-m12m1log2}
b\asymp n^{-(m+1)/(2m+1)}(\log n)^{-m(d-p-1)/(2m+1)}.
\end{equation}
The linear risk in (\ref{eqn:explicitexpofRL}) becomes
\begin{align}
R^L & \asymp n^{-1}\int_{\prod_{k=d-r+1}^dx_{k}\leq (nb)^{1/m},x_{k}\geq 1}\left(1-\frac{1}{nb}x_{d-r+1}^m\cdots x_d^m\right)\nonumber\\
& \quad\quad\quad\quad\quad\quad\quad\quad\quad\cdot (1+x^2_{d-r+1}+\cdots+x_p^2)^{-1}dx_{d-r+1}\cdots dx_d \nonumber\\
& \asymp n^{-1} \int_{x_{d-r+1}\cdots x_pz_d\leq (nb)^{1/m}}\left(1-\frac{1}{nb}x_{d-r+1}^m\cdots x_p^mz_d^m\right)(\log z_d)^{d-p-1}\nonumber\\
&\quad\quad\quad\quad\quad\quad\quad \quad\cdot(1+x_{d-r+1}^2+\cdots+x_p^2)^{-1}dx_{d-r+1}\cdots dx_pdz_d\nonumber\\
& \asymp  [\log (nb)]^{d-p-1}n^{-1+1/m}b^{1/m},\nonumber
\end{align}
where the second step uses the same change of variables by letting $z_{j} =x_{p+1}x_{p+2}\cdots x_j$, $j=p+1,\ldots, d$, and the last step is by Lemma \ref{lemma:intx1xrzxk1z1} in Section \ref{subsec:auxiliarytechlem}.
By (\ref{eqn:p0d-rbn-m12m1log2}), we have
\begin{align}
R^L &  \asymp  [n(\log n)^{1+p-d}]^{-2m/(2m+1)}.\nonumber
\end{align}

If $p=d$, as discussed in the previous two cases, the number of combinations of non-1 indices being summed is finite, and the largest terms of the summation (\ref{eqn:sumbnuinVprodk1d}) over ${\vec{\bnu}}\in V$ correspond to any combinations of $r$ non-1 indices, for example, $\nu_k\equiv 1$ for $k\geq r+1$, and $(\nu_1,\ldots,\nu_r)\in\N^r$.  Thus, (\ref{eqn:sumbnuinVprodk1d}) is equivalent to 
\begin{equation*}
\sum_{\prod_{k=1}^r\nu_{k}\leq (nb)^{1/m}}\left(b\prod_{k=1}^r\nu_{k}^{m}- n^{-1}\prod_{k=1}^r\nu_{k}^{2m}\right)\left(1+\sum_{k=1}^r\nu_k^2\right)^{-1} \asymp 1.
\end{equation*}
Using the integral approximation, we have
\begin{align}
1& \asymp \int_{\prod_{k=1}^rx_{k}\leq (nb)^{1/m},x_{k}\geq1}  \left(b\prod_{k=1}^rx_{k}^m - n^{-1}\prod_{k=1}^rx_{k}^{2m}\right) \left(1+\sum_{k=1}^rx_k^2\right)^{-1}dx_{1}\cdots dx_{r}\nonumber\\
& \asymp \int_{\prod_{k=1}^rx_{k}\leq (nb)^{1/m},x_{k}\geq1}  b\prod_{k=1}^rx_{k}^m \left(1+\sum_{k=1}^rx_k^2\right)^{-1}dx_{1}\cdots dx_{r}\nonumber
\end{align}
By letting $\beta=m>1$ and $\alpha=2$ in Lemma \ref{lemma:intx1xrzxk1z2} in Section \ref{subsec:auxiliarytechlem}, we have for any $r\geq 1$, 
\begin{equation}
\label{eqn:p0d-rbn-m12m1log3}
b\asymp  n^{-(mr+r-2)/(2mr+r-2)}.
\end{equation}
The linear risk in (\ref{eqn:explicitexpofRL}) becomes
\begin{align}
R^L & \asymp n^{-1}\int_{\prod_{k=1}^rx_{k}\leq (nb)^{1/m},x_{k}\geq 1}\left(1-\frac{1}{nb}x_{1}^m\cdots x_r^m\right)\nonumber\\
& \quad\quad\quad\quad\quad\quad\quad\quad\quad\quad\quad\quad\quad\quad\quad\cdot (1+x^2_{1}+\cdots+x_r^2)^{-1}dx_{1}\cdots dx_r\nonumber\\
& \asymp n^{-1}\int_{\prod_{k=1}^rx_{k}\leq (nb)^{1/m},x_{k}\geq 1}(1+x^2_{1}+\cdots+x_r^2)^{-1}dx_{1}\cdots dx_r\nonumber\\
& \asymp \left[n^{-1}(nb)^{(r-2)/(mr)}\right] \mathbbm{1}_{r\geq 3} + \left[n^{-1}\log(nb)\right]\mathbbm{1}_{r=2}+\left(n^{-1}\right)\mathbbm{1}_{r=1},\nonumber
\end{align}
where the last step uses Lemma \ref{lemma:intx1xrzxk1z2} in Section \ref{subsec:auxiliarytechlem} by letting $\beta=0$ and $\alpha=2$. 
By (\ref{eqn:p0d-rbn-m12m1log3}), we have
\begin{equation*}
R^L \asymp\left[n^{-(2mr)/[(2m+1)r-2]}\right] \mathbbm{1}_{r\geq 3} + \left[n^{-1}\log(n)\right]\mathbbm{1}_{r=2}+n^{-1}\mathbbm{1}_{r=1},
\end{equation*}
where the constant factor only depends on $\sigma_0^2$, $\sigma_j^2$, $m$, $r$, $p$  and $d$. This completes the proof.

\subsubsection{Proof of minimax upper bound: Theorem \ref{thm;deterdesgnupperbdgeneralpdreg}}
\label{sec:proofupperbdreg}

We now prove the theorem for only $r=d$ and $p=d-1$. Other cases can be proved similarly with slight changes.

Using the discrete transformed data (\ref{eqn:distransdataz}), the regularized estimator $\widehat{f}_{n\lambda}$ by (\ref{scheme1}) can be obtained through 
\begin{align}
\widehat{\theta}_{\vec{\bnu}}=\underset{\tilde{\theta}_{\vec{\bnu}}\in\R}{\arg\min} &\left\{ \frac{1}{n(p+1)}\left[\frac{1}{\sigma_0^2}\sum_{\vec{\bnu}\in V, \nu_k\leq l}\left(z_{\vec{\bnu}}^{e_0} - \theta_{\vec{\bnu}}\right)^2\right.\right.\nonumber\\
&\quad\quad\quad\left.\left.+\sum_{j=1}^p\frac{1}{\sigma_j^2}\sum_{\vec{\bnu}\in V, \nu_k\leq l }\left(z^{e_j}_{\vec{\bnu}}- \nu_j\theta_{\vec{\bnu}}\right)^2\right]+\lambda\sum_{\vec{\bnu}\in V, \nu_k\leq l}\lambda_{\vec{\bnu}}\theta_{\vec{\bnu}}^2
\right\}\nonumber
\end{align}
and $\widehat{f}_{n\lambda}(\bt) = \underset{\vec{\bnu}\in V, \nu_k\leq l}{\sum}\widehat{\theta}_{\vec{\bnu}}\psi_{\vec{\bnu}}(\bt)$, where $V$ is defined in (\ref{def:Vbnu}).
Direct calculations give
\begin{equation*}
\widehat{\theta}_{\vec{\bnu}} =  \frac{\sigma_0^{-2}z^{e_0}_{\vec{\bnu}}+\sum_{j=1}^p\sigma_j^{-2}\nu_jz^{e_j}_{\vec{\bnu}}}{ \sigma_0^{-2}+\sum_{j=1}^p\sigma_j^{-2}\nu_j^2 + \lambda\lambda^{-1}_{\vec{\bnu}}}. 
\end{equation*}

The deterministic error of $\widehat{f}_{n\lambda}$ can be analyzed by two parts. On the one hand, since $f_0\in\HH$ and $\lambda_\nu\asymp\nu^{-2m}$, we know $\sum_{\vec{\bnu}\in V, \nu_k\geq l+1}(\theta_{\vec{\bnu}}^0)^2\asymp n^{-2m}$. This is the truncation error due to $\widehat{\theta}_{\vec{\bnu}} =0$ for $\nu_k\geq l+1$, $1\leq k\leq d$.
On the other hand, note that $\langle\vec{\psi}_{\vec{\bnu}},\vec{\psi}_{\vec{\bmu}}\rangle_n^2\leq 1$ and then 
\begin{equation*}
\left(\sum_{\vec{\bmu}\in V, \mu_k\geq l+1}\theta_{\vec{\bmu}}^0\langle\vec{\psi}_{\vec{\bnu}},\vec{\psi}_{\vec{\bmu}}\rangle_n\right)^2\leq \sum_{\vec{\bmu}\in V, \mu_k\geq l+1}(\theta^0_{\vec{\bmu}})^2 \asymp n^{-2m}.
\end{equation*}
Thus,
\begin{align}
& \sum_{\vec{\bnu}\in V, \nu_k\leq l}\left(\E\widehat{\theta}_{\vec{\bnu}} - \theta_{\vec{\bnu}}^0\right)^2 \nonumber\\
& \quad \lesssim  \sum_{\vec{\bnu}\in V, \nu_k\leq l}\left( \frac{\lambda\lambda^{-1}_{\vec{\bnu}}}{\sigma_0^{-2}+\sum_{j=1}^p\sigma_j^{-2}\nu_j^2+\lambda\lambda^{-1}_{{\vec{\bnu}}}}\right)^2(\theta_{\vec{\bnu}}^0)^2 + n^{-2m+1}\nonumber\\
&\quad \leq \lambda^2\sup_{{\vec{\bnu}}\in V}\frac{\lambda^{-1}_{\vec{\bnu}}}{\left(\sigma_0^{-2}+\sum_{j=1}^p\sigma_j^{-2}\nu_j^2+\lambda\lambda^{-1}_{{\vec{\bnu}}}\right)^2} \sum_{{\vec{\bnu}}\in V}\lambda_{\vec{\bnu}}^{-1}(\theta_{\vec{\bnu}}^0)^2+n^{-2m+1}\nonumber \\
& \quad \asymp \lambda^2 J(f_0)\sup_{{\vec{\bnu}}\in V} \frac{\nu_1^{2m}\cdots\nu_d^{2m}}{(1+\sum_{j=1}^p\nu_j^2+\lambda\nu_1^{2m}\cdots\nu_d^{2m})^2} + n^{-2m+1}.\nonumber
\end{align}
Define that
\begin{equation*}
B_\lambda({\vec{\bnu}}) =  \frac{\nu_1^{2m}\cdots\nu_d^{2m}}{(1+\sum_{j=1}^p\nu_j^2+\lambda\nu_1^{2m}\cdots\nu_d^{2m})^2}.
\end{equation*}
For the $\sup_{{\vec{\bnu}}\in V}B_\lambda({\vec{\bnu}})$ term above, suppose that $\prod_{j=1}^d\nu_j^{2m}>0$ is fixed and denoted by $x^{-1}$, then $B_\lambda({\vec{\bnu}})$ is maximized by letting $\sum_{j=1}^p\nu_j^2$ be as small as possible, where $p=d-1$. This suggests $\nu_1=\nu_2=\cdots=\nu_p=1$, and 
\begin{align}
\underset{\vec{\bnu}\in V}{\sup}B_{\lambda}({\vec{\bnu}}) 
& \asymp \sup_{x>0}\frac{x^{-1}}{(1+\lambda x^{-1})^2} \asymp \lambda^{-1}, \nonumber
\end{align}
where the last step is achieved when $x \asymp \lambda$. Combining all parts of bias gives
\begin{equation}
\label{eqn:detererrordeterdesign}
\sum_{{\vec{\bnu}}\in V}\left(\E\widehat{\theta}_{\vec{\bnu}} - \theta_{\vec{\bnu}}^0\right)^2=O\left\{\lambda J(f_0) + n^{-2m+1} + n^{-2m}\right\},
\end{equation}
where the constant factor only depends on $\sigma_0^2$, $\sigma_j^2$, $m$, $r, p$ and $d$.

The stochastic error is bounded as follows:
\begin{align}
\sum_{\vec{\bnu}\in V}\left(\widehat{\theta}_{\vec{\bnu}} - \E\widehat{\theta}_{\vec{\bnu}}\right)^2 & = \sum_{\vec{\bnu}\in V,\nu_k\leq l}\frac{n^{-1}(\sigma_0^{-2}+\sum_{j=1}^p\sigma_j^{-2}\nu_j^2)}{(\sigma_0^{-2}+\sum_{j=1}^p\sigma_j^{-2}\nu_j^2+\lambda\lambda_{\vec{\bnu}}^{-1})^2}\nonumber\\
& \lesssim\sum_{{\vec{\bnu}}\in V,\nu_k\leq l}\frac{1+\sum_{j=1}^p\nu_j^2}{n(1+\sum_{j=1}^p\nu_j^2+\lambda\nu_1^{2m}\cdots\nu_d^{2m})^2}.\nonumber
\end{align}
Using Lemma \ref{lemma:varthetamultivariatetensorrank} in Section \ref{subsubsecdefnalambda} with $a=0$ and $p=d-1$, we have 
\begin{equation}
\label{eqn:vardeterdesign}
\sum_{{\vec{\bnu}}\in V}\left(\widehat{\theta}_{\vec{\bnu}} - \E\widehat{\theta}_{\vec{\bnu}}\right)^2 = O\left\{n^{-1}\lambda^{-1/2m}\right\}.
\end{equation}
Combining (\ref{eqn:detererrordeterdesign}) and (\ref{eqn:vardeterdesign}) and letting $\lambda\asymp n^{-2m/(2m+1)}$ completes the proof.

\subsection{Proofs of results  in Section \ref{sec:minmaxriskpartder}: estimating partial derivatives}
\label{sec:proofthmupperlimdderiv}

We now turn to prove the results for estimating partial derivatives under the random design.

\subsubsection{Proof of minimax lower bound: Theorem \ref{theorem:lowerbndlimDderi}}
\label{subsubsec:lowerbdestpartder}

The minimax lower bound will be established by using Fano's lemma but the proof is different from  Section \ref{subsubsec:prooflowerrandom} in construction details. It suffices to consider a special case that noises $\epsilon^{e_0}$ and $\epsilon^{e_j}$s are Gaussian with $\sigma_0=1$ and $\sigma_j=1$, and $\Pi^{e_0}$ and $\Pi^{e_j}$s are uniform, and $\HH_1$ is generated by periodic kernels. For simplicity, we still use the notation introduced in Section \ref{subsubsec:prooflowerrandom}. In the rest of this section, without less of generality, we consider estimating $\partial f_0/\partial t_1(\cdot)$ with $p\geq 1$.

First, the number of multi-indices $\vec{\bnu}=(\nu_1,\ldots,\nu_r)\in\N^r$ satisfying 
\begin{equation*}
\nu_1^{(m-1)/m}\nu_2\cdots\nu_r\leq N
\end{equation*}
is $c'_0N^{m/(m-1)}$,
where $c'_0$ is some constant. Define a length-$\{c'_0N^{m/(m-1)}\}$ binary sequence as
\begin{equation*} 
b=\{b_{\vec{\bnu}}: \nu_1^{(m-1)/m}\nu_2\cdots\nu_r\leq N\}\in\{0,1\}^{c'_0N^{m/(m-1)}}.
\end{equation*} 
We write
\begin{align}
 & h_b(t_1,\ldots,t_r) = N^{-m/2(m-1)}\sum_{\nu_1^{(m-1)/m}\nu_2\cdots\nu_r\leq  N}b_{\vec{\bnu}}\left(1+\nu_1^2+\cdots+\nu_r^2\right)^{-1/2}\nonumber\\
 &\quad\quad\quad\quad\quad\quad\quad\times \left[\nu_1^{(m-1)/m}\nu_2\cdots\nu_r+N\right]^{-m}\psi_{\nu_1}(t_1)\psi_{\nu_2}(t_2)\cdots\psi_{\nu_r}(t_r).\nonumber
\end{align}
where $\psi_{\nu_k}(t_j)$s are the trigonometric basis in (\ref{eqn:trigbasisexp}).
Note that
\begin{align}
\|h_b\|_\HH^2  & \lesssim N^{-m/(m-1)}\sum_{\nu_1^{(m-1)/m}\nu_2\cdots\nu_r\leq  N}b_{\vec{\bnu}}^2\nu_1^2\left(1+\nu_1^2+\cdots+\nu_r^2\right)^{-1}\nonumber\\
& \leq N^{-m/(m-1)}\sum_{\nu_1^{(m-1)/m}\nu_2\cdots\nu_r\leq  N}\nu_1^2\left(1+\nu_1^2+\cdots+\nu_r^2\right)^{-1}\asymp1,\nonumber
\end{align}
where the last step is by Lemma \ref{applem:x1m1mx2x12} in Section \ref{subsec:auxiliarytechlem}. Hence, $h_b(\cdot)\in\HH$.

Then, using the Varshamov-Gilbert bound, there exists a collection of binary sequences $\{b^{(1)},\ldots,b^{(M)}\}\subset\{0,1\}^{c'_0N^{m/(m-1)}}$ such that 
\begin{equation*}
M\geq 2^{c'_0N^{m/(m-1)}/8}
\end{equation*} 
and 
\begin{equation*}
H(b^{(l)},b^{(q)})\geq c'_0N^{m/(m-1)}/8,\quad \forall 1\leq l<q\leq M.
\end{equation*}
For $b^{(l)},b^{(q)}\in\{0,1\}^{c'_0N^{m/(m-1)}}$, we have
\begin{align}
& \left\|\frac{\partial h_{b^{(l)}}}{\partial t_1}-\frac{\partial h_{b^{(q)}}}{\partial t_1}\right\|_{L_2}^2\nonumber\\
& \quad \geq c'N^{-m/(m-1)}(2N)^{-2m}\sum_{\nu_1^{(m-1)/m}\nu_2\cdots\nu_r\leq  N}\nu_1^2(1+\nu_1^2+\cdots+\nu_r^2)^{-1}\left[b^{(l)}_{\vec{\bnu}} - b^{(q)}_{\vec{\bnu}}\right]^2\nonumber\\
& \quad \geq c'N^{-m/(m-1)}(2N)^{-2m}\sum_{c'_17N/8\leq \nu_1^{(m-1)/m}\nu_2\cdots\nu_r\leq N}\nu_1^2(1+\nu_1^2+\cdots+\nu_r^2)^{-1}\nonumber\\
& \quad = c'_2 N^{-2m}\nonumber
\end{align}
for some constant $c'$, $c'_1$ and $c'_2$, where the last step is by Lemma \ref{applem:x1m1mx2x12} in Section \ref{subsec:auxiliarytechlem}.
On the other hand,  for any $b^{(l)}\in\{b^{(1)},\ldots,b^{(M)}\}$, 
\begin{align}
&  \|h_{b^{(l)}}\|_{L_2}^2+\sum_{j=1}^p\| \partial h_{b^{(l)}}/\partial t_j\|_{L_2}^2\nonumber\\
& \quad \leq N^{-m/(m-1)}N^{-2m}\sum_{\nu_1^{(m-1)/m}\nu_2\cdots\nu_r\leq  N}\left[b^{(l)}_{\vec{\bnu}} \right]^2\nonumber\\
& \quad \leq c'_3 N^{-2m}\nonumber
\end{align}
with some constant $c'_3$, where the last step is a corollary of Lemma \ref{applem:x1m1mx2x12}.

Last, by the same argument in (\ref{eqn:inftildeffoinhh1}), (\ref{eqn:fanolemhatthetap}), (\ref{eqn:inftildeffoinhh3}) and (\ref{eqn:inftildeffoinhh4}), we obtain 
\begin{align}
& \inf_{\tilde{f}}\sup_{f_0\in\HH}\P\left\{\left\|\tilde{f}(\bt) - \frac{\partial f_0(\bt)}{\partial t_1}\right\|_{L_2}^2\geq \frac{1}{4}c_2' N^{-2m}\right\}\nonumber\\
& \quad\quad\quad\quad\quad\quad \geq 1-\frac{2c_3'n(p+1)N^{-2m} + \log 2}{c'_0(\log 2)N^{m/(m-1)}/8}.\nonumber
\end{align}
Taking $N=c_4'n^{(m-1)/(2m^2-m)}$ with an appropriately chosen $c_4'$,  we have
\begin{equation*}
\underset{n\rightarrow \infty}{\lim\sup}\inf_{\tilde{f}} \sup_{f_0\in\HH}\P\left\{\left\|\tilde{f}(\bt) - \frac{\partial f_0(\bt)}{\partial t_1}\right\|_{L_2}^2\geq C_2n^{-2(m-1)/(2m-1)}\right\}>0,
\end{equation*}
where the constant factor $C_2$ only depends on  $\sigma_0^2$, $\sigma_j^2$, $m$, and bounded values $r$, $p$ and $d$. This completes the proof.
\subsubsection{Proof of minimax upper bound: Theorem \ref{theorem:upperbndlimDderi}}
\label{subsubsec:upperbdestpartder}

We continue to use the notation and definitions  such as the minimizer $\bar{f}$, the Fr\'echet derivatives $Dl_n(f)g$, $Dl_\infty(f)g$, $D^2l_n(f)gh$, $D^2l_\infty(f)gh$, the operator $G_\lambda^{-1}$ and most importantly $\tilde{f}^*$ in Section \ref{subsubsec:proofupperrrandom}. Unlike Section \ref{subsubsec:proofupperrrandom}, here we do not require $\pi^{e_j}$s are known nor $f_0$ has periodic boundaries on $\XX_1^d$ by some transformation. 

By the assumption that $\pi^{e_j}$s are bounded away from 0 and infinity, we have  for any $1\leq j\leq p$,
\begin{equation*}
\int_{\XX_1^d}\left[ \frac{\partial \widehat{f}_{n\lambda}(\bt)}{\partial t_j} -\frac{\partial f_0(\bt)}{\partial t_j}\right]^2d\bt \lesssim \|\widehat{f} - f_0\|_0^2.
\end{equation*}
Hence, the following lemma is sufficient for proving Theorem \ref{theorem:upperbndlimDderi}.
\begin{lemma} 
\label{lem:limd2ninftyn0}
Under the conditions of Theorem \ref{theorem:lowerbndlimDderi}, then $\widehat{f}_{n\lambda}$ given by (\ref{scheme1}) satisfies
\begin{equation*}
\lim_{D_2\rightarrow\infty}\underset{n\rightarrow\infty}{\lim\sup}\sup_{f_0\in \HH} \P\left\{\|\widehat{f} - f_0\|_{0}^2 > D_2n^{-2(m-1)/(2m-1)}\right\} = 0,
\end{equation*}
if the tuning parameter $\lambda$ is chosen by $\lambda \asymp n^{-2(m-1)/(2m-1)}$.
\end{lemma}
\paragraph{\textit{A lemma for the proof.}}
In $\HH$, the quadratic form $\langle f,f\rangle_{0}$ is completely continuous with respect to $\langle f,f\rangle_{R}$. By the theory in Section 3.3 of Weinberger \cite{weinberger1974variational}, there exists an eigen-decomposition for the generalized Rayleigh quotient $\langle f,f\rangle_0/\langle f,f\rangle_R$ in $\HH$, where we denote the eigenvalues are $\{(1+\gamma_\nu)^{-1}\}_{\nu\geq 1}$ and the corresponding eigenfunctions are $\{(1+\gamma_\nu)^{-1/2}\xi_\nu\}_{\nu\geq 1}$.
Thus, $\langle \xi_{\nu},\xi_{\mu}\rangle_R = (1+\gamma_{\nu})\delta_{\nu\mu}$ and $\langle \xi_{\nu},\xi_{\mu}\rangle_0 = \delta_{\nu\mu}$, where $\delta_{\nu\mu}$ is Kronecker's delta. 
The following proposition gives the decay rate of  $\gamma_\nu$ and its proof is given in Section \ref{subsubsec:proofproeigengammanu}.
\begin{lemma}
\label{lem:eigendecayrhonu}
By the well-ordering principle, the elements in the set 
\begin{equation*}
\left\{\left(1 +\sum_{j=1}^p\nu_j^2\right)\prod_{k=1}^d\nu_k^{-2m}: {\vec{\bnu}}\in V\right\}
\end{equation*} 
can be ordered from large to small, where $V$ is defined in (\ref{def:Vbnu}). Denote by $\{\gamma'_\nu\}_{\nu\geq 1}$ the ordered sequence. Then $\gamma_\nu\asymp (\gamma'_\nu)^{-1}$.
\end{lemma}
The proof of this lemma is delegated to Section \ref{subsubsec:proofproeigengammanu}.
The lemma bridges the gap between the proof needed for Lemma \ref{lem:limd2ninftyn0} and the proof for Theorem \ref{thm:mainupperrateestf0} shown in Section \ref{subsubsec:proofupperrrandom} since the eigenvalues $\rho_{\vec{\bnu}}$ in  Section \ref{subsubsec:proofupperrrandom} satisfies $\rho_{\vec{\bnu}}\asymp (1+\sum_{j=1}^p\nu_j^2)^{-1}\prod_{k=1}^d\nu_k^{2m}$. Hence in later analysis, we can exchange the use of $\{\gamma_\nu,\nu\in \N\}$ and $\{\rho_{\vec{\bnu}}: \vec{\bnu}\in V\}$ in some asymptotic calculation settings.

For any function $f\in \HH$, it can be decomposed as
\begin{equation*}
f(t_1,\ldots,t_d) = \sum_{\nu\in\N}f_\nu\xi_\nu(t_1,\ldots,t_d),\quad  \mbox{ where } f_\nu = \langle f(\bt),\xi_\nu(\bt)\rangle_0,
\end{equation*} 
and $J(f) = \langle f,f\rangle_R - \langle f,f\rangle_0 = \sum_{\nu\in\N}\gamma_\nu f_\nu^2$. 

First, we present an upper bound of the deterministic error $(\bar{f}-f_0)$. 
\begin{lemma} 
\label{lem:determerrorderiv}
The deterministic error satisfies
\begin{equation*}
\|\bar{f} - f_0\|_{0}^2 = O\left\{\lambda J(f_0)\right\}.
\end{equation*}
\end{lemma}

\begin{proof} 
For any $0\leq a\leq 1$,
\begin{align}
\|\bar{f} - f_0\|_{0}^2 & = \sum_{\nu=1}^\infty \left(\frac{\lambda\gamma_\nu}{1+\lambda\gamma_\nu}\right)^2(f_\nu^0)^2\nonumber\\
& \leq \lambda^2\sup_{\nu\in \N}\frac{\gamma_\nu}{(1+\lambda\gamma_\nu)^2}\sum_{\nu=1}^\infty\gamma_\nu(f_\nu^0)^2\nonumber\\
&\leq \lambda^2J(f_0)\sup_{x>0}\frac{x^{-1}}{(1+\lambda x^{-1})^2}\nonumber\\
& \asymp \lambda^2 J(f_0)\lambda^{-1} = \lambda J(f_0),\nonumber
\end{align}
where the fourth step is achieved when $x\asymp\lambda$.
\end{proof}

Second, we show an upper bound of $(\tilde{f}^*-\bar{f})$, which accounts for a part of the stochastic error.
\begin{lemma}
\label{lem:stocherrortilde-barderiv}
For $1\leq p\leq d$, then if $m>5/4$, we have
\begin{equation*}
\|\tilde{f}^* - \bar{f}\|_0^2 = O_\P\left\{n^{-1}\lambda^{-1/(2m-2)}\right\}.
\end{equation*}
\end{lemma}
\begin{proof}
As shown in (\ref{eqn:dlnlambdang0}), $\E[\frac{1}{2}Dl_{n,\lambda}(\bar{f})g]^2= O\{n^{-1}\|g\|_0^2\}$. By the definition of $G_\lambda^{-1}$ in (\ref{def:ginv}),
\begin{equation*}
\|G_\lambda^{-1}g\|_0^2 = \sum_{\nu=1}^\infty \left(1+\lambda\gamma_\nu\right)^{-2}\langle g,\xi_\nu\rangle_R^2,\quad\forall g\in\HH.
\end{equation*}
Thus,
\begin{align}
\E\|\tilde{f}^* - \bar{f}\|_{0}^2 & = \frac{1}{4}\E\left[\sum_{\nu=1}^\infty(1+\lambda \gamma_\nu)^{-2}\langle Dl_{n\lambda}(\bar{f}), \xi_\nu\rangle_R^2\right]\nonumber\\
& \leq \sum_{\nu=1}^\infty(1+\lambda\gamma_\nu)^{-2}\E\left[\frac{1}{2}Dl_{n\lambda}(\bar{f})\xi_\nu\right]^2\nonumber\\
& \lesssim n^{-1}\sum_{\nu=1}^\infty\left(1+\lambda\gamma_\nu\right)^{-2}\nonumber\\
& \asymp n^{-1}M_0(\lambda),\nonumber
\end{align}
where the last step is because of Lemma \ref{lem:eigendecayrhonu}, and $M_a(\lambda)$ for $0\leq a\leq 1$ is defined in Lemma \ref{lem:defMalambdaandbounds} of Section \ref{subsubsec:defofmalambda}. Hence, we complete the proof by using Lemma \ref{lem:defMalambdaandbounds}.
\end{proof}

Then, we give an upper bound of $(\widehat{f} - \tilde{f}^*)$, which accounts for another part of the stochastic error.
\begin{lemma} 
\label{lem:boundhattildefM}
If $n^{-1}\lambda^{-[a+ma/(m-1)+3/2m]}\left[\log(1/\lambda)\right]^{r-1}\rightarrow 0$ and $1/2m<a<(2m-3)/2m$, we have
\begin{equation*}
\|\widehat{f} - \tilde{f}^*\|_0^2 = o_\P\left\{n^{-1}\lambda^{-1/(2m-2)}\right\}.
\end{equation*}
\end{lemma}
\begin{proof}
Observe that
\begin{equation}
\begin{aligned}
& \E\|\widehat{f} - \tilde{f}\|_{0}^2 \\
& \asymp \E\sum_{\vec{\bnu}\in V}(1+\lambda\gamma_{\vec{\bnu}})^{-2}\left[\frac{1}{2}D^2l_\infty(\bar{f})(\widehat{f} - \bar{f})\phi_{\vec{\bnu}} - \frac{1}{2}D^2l_{n}(\bar{f})(\widehat{f} - \bar{f})\phi_{\vec{\bnu}}\right]^2\nonumber\\
& \leq \E\sum_{\vec{\bnu}\in V}(1+\lambda\gamma_{\vec{\bnu}})^{-2}\nonumber\\
& \times\frac{1}{p+1}\left\{\left[\frac{1}{n\sigma_0^2}\sum_{i=1}^n(\widehat{f}-\bar{f})(\bt_i^{e_0})\phi_{\vec{\bnu}}(\bt_i^{e_0})-\frac{1}{\sigma_0^2} \int (\widehat{f} - \bar{f})(\bt)\phi_{\vec{\bnu}}(\bt)\pi^{e_0}(\bt)\right]^2\right.\nonumber\\
&  \left. +\sum_{j=1}^p\left[\frac{1}{n\sigma_j^2}\sum_{i=1}^n\frac{\partial (\widehat{f}-\bar{f})}{\partial t_j}(\bt_i^{e_0})\frac{\partial\phi_{\vec{\bnu}}}{\partial t_j}(\bt_i^{e_0}) - \frac{1}{\sigma_j^2}\int \frac{\partial(\widehat{f} - \bar{f})(\bt)}{\partial t_j}\frac{\partial\phi_{\vec{\bnu}}(\bt)}{\partial t_j}\pi^{e_0}(\bt)\right]^2\right\}\nonumber\\
& \lesssim n^{-1}\|\widehat{f} - \bar{f}\|_{L_2(a+1/m)}^2 \sum_{{\vec{\bnu}}\in V}\left(1+\frac{\rho_{\vec{\bnu}}}{\|\phi_{\vec{\bnu}}\|_{L_2}^2}\right)^a(1+\lambda\rho_{\vec{\bnu}})^{-2}\nonumber\\
& = n^{-1}\|\widehat{f} - \bar{f}\|^2_{L_2(a+1/m)}M_a(\lambda)\nonumber\\
& \leq \left\{n^{-1}\lambda^{-[a+3/2m+ma/(m-1)]}[\log(1/\lambda)]^{r-1}\right\}n^{-1}\lambda^{-1/(2m-2)},\nonumber
\end{aligned}
\end{equation}
where the first step exchange the use of $\{\gamma_\nu,\nu\in \N\}$ and $\{\rho_{\vec{\bnu}}: \vec{\bnu}\in V\}$, the third step is by  (\ref{eqn:e1p1n-1hatfbarfl2a1m}), and the last step is Lemma \ref{lem:tilfbarfl2a}, Lemma \ref{lem:boundhattildef} and Lemma \ref{lem:defMalambdaandbounds} in Section \ref{subsubsec:defofmalambda}. The above inequality holds for any $1/2m<a<(2m-3)/2m$. This completes the proof.
\end{proof}

Last, we combine Lemma \ref{lem:determerrorderiv}, Lemma \ref{lem:stocherrortilde-barderiv} and Lemma \ref{lem:boundhattildefM}. By letting $\lambda\asymp n^{-2(m-1)/(2m-1)}$ and $a=1/2m + \epsilon$ for some $\epsilon>0$, then 
\begin{equation*}
n^{-1}\lambda^{-(a+3/2m+ma/(m-1))}[\log(1/\lambda)]^{r-1}\rightarrow 0
\end{equation*} 
holds as long as $m>2$. Therefore, we conclude that for any $1\leq p\leq d$ and $m>2$,
\begin{align}
\|\widehat{f} - f_0\|_0^2 & = O\left\{\lambda J(f_0)\right\} + O_\P\left\{ n^{-1}\lambda^{-1/(2m-2)}\right\} + o_\P\left\{ n^{-1}\lambda^{-1/(2m-2)}\right\} \nonumber\\
& = O_\P\left\{n^{-2(m-1)/(2m-1)}\right\}.\nonumber
\end{align}
This completes the proof for Lemma \ref{lem:limd2ninftyn0} and the proof for Theorem \ref{theorem:upperbndlimDderi} .

\subsection{Key lemmas}
\label{subsec:keylemmaforproofs}

Now we prove and show some keys lemmas used for the proofs in Section \ref{sec:proofsofall}, Section \ref{subsec:reglattproof} and Section \ref{sec:proofthmupperlimdderiv}.
We remind the reader that the proofs in this section rely on some lemmas to be stated later in Section \ref{subsec:auxiliarytechlem}.

\subsubsection{Proof of Lemma \ref{lemmanormrequiv}}
\label{subsec:proofoflemmanormequiv}
\begin{equation*}
\text{\textit{The norm $\|\cdot\|_R$ is equivalent to $\|\cdot\|_{\HH}$ in $\HH$.}}
\end{equation*}
\begin{proof}
Observe that for any $g\in \HH$, by the assumption that $\pi^{e_0}$ and $\pi^{e_j}$s are bounded away from  0 and infinity, we have
\begin{align}
& \frac{1}{p+1}\left[\frac{1}{\sigma_0^2}\int g^2(\bt)\pi^{e_0}(\bt) + \sum_{j=1}^p \frac{1}{\sigma_j^2}\int\left\{\frac{\partial g(\bt)}{\partial t_j}\right\}^2\pi^{e_j}(\bt) \right]  \nonumber\\
&\leq c_1\left[\int g^2(\bt) + \sum_{j=1}^p \int\left\{\frac{\partial g(\bt)}{\partial t_j}\right\}^2\right] \leq c_2\cdot c_{K}^{2d}\|g\|_{\HH}^2,\nonumber
\end{align}
for some constant $c_1$ and $c_2$, where the last step is by Lemma \ref{lem:bounddejft}. Hence 
\begin{equation}
\label{eqn:grnormlessc2ckghhnorm}
\|g\|_R^2 \leq (c_2c_{K}^{2d}+1)\|g\|_\HH^2.
\end{equation}

One the other hand, for any $g\in\HH$ we can do the orthogonal decomposition $g=g^0+g^1$ where $\langle g^0,g^1\rangle_\HH=0$, $g^0$ is in the null space of $J(\cdot)$ and $g^1$ is in the orthogonal space of the null space of $J(\cdot)$ in $\HH$.  Since the null space of $J(\cdot)$ has a finite basis which forms a positive definite kernel matrix, we assume the minimal eigenvalue of the kernel matrix is $\mu_{\min}'>0$. Then there exists a constant $c_3>0$ such that
\begin{equation}
\label{eqn:f0r2c2f0l2}
\|g^0\|_R^2\geq c_3\|g^0\|_{L_2}^2 \geq c_3 \mu_{\min}' \|g^0\|_\HH^2.
\end{equation}
For $g^1$, we have $\|g^1\|_R^2\geq J(g^1) = \|g^1\|_\HH^2$. Thus, for any $g\in\HH$,
\begin{align}
\|g\|_R^2 & \geq c_3\int \left(g^0+g^1\right)^2 + \|g^1\|_\HH^2\nonumber\\
& \geq c_3 \left\{\|g^0\|_{L_2}^2 + \frac{1+c_3}{c_3}\|g^1\|_{L_2}^2 - 2\|g^0\|_{L_2}\|g^1\|_{L_2}\right\}\nonumber\\
& \geq \frac{c_3}{1+c_3}\|g^0\|^2_{L_2},\nonumber
\end{align}
where the second inequality is by $\|g^1\|_\HH^2\geq \|g^1\|_{L_2}^2$. Then by (\ref{eqn:f0r2c2f0l2}), we obtain $\|g\|_R^2 \geq (1+c_3)^{-1}c_3\mu_{\min}'\|g^0\|_\HH^2$. Together with $\|g\|_R^2\geq J(g^1) = \|g^1\|_\HH^2$, we have
\begin{equation}
\label{eqn:grnormgtrc2ckghhnorm}
\|g\|_R^2 \geq \left(1+ \frac{1+c_3}{c_3\mu_{\min}'}\right)^{-1}\|g\|_\HH^2.
\end{equation}
Combining (\ref{eqn:grnormlessc2ckghhnorm}) and (\ref{eqn:grnormgtrc2ckghhnorm}) completes the proof.
\end{proof}

\subsubsection{Proof of Lemma \ref{lem:eigendecayrhonu}}
\label{subsubsec:proofproeigengammanu}

\begin{proof}
When $d=1$, this problem is solved in Cox \cite{cox1988approximation}. Their method is finding an orthonormal basis in $L_2(\XX_1)$ to simultaneously diagonalize $\langle f,f\rangle_0$ and $\langle f,f\rangle_R$, and then obtain the decay rate of $\gamma_\nu$. However, their method cannot be applied to our case when $2\leq p\leq d$.
Alternatively, we use the Courant-Fischer-Weyl min-max principle to prove the lemma.

Note that for any $f\in \HH$, the norm $\|f\|_0^2$ is equivalent to
\begin{equation*}
\int f^2 + \sum_{j=1}^p\int\left(\frac{\partial f(\bt)}{\partial t_j}\right)^2.
\end{equation*} 
From Lemma \ref{lemmanormrequiv}, the norm $\|\cdot\|^2_R$ is equivalent to $\|\cdot\|^2_\HH$. Now by applying the mapping principle [see, e.g., Theorem 3.8.1 in Weinberger \cite{weinberger1974variational}], we may replace $\langle f,f\rangle_0$ by $\int f^2 + \sum_{j=1}^p\int(\partial f/\partial t_j)^2$ and $\langle f,f\rangle_R$ by $\|f\|^2_\HH$, and the resulting eigenvalues $\{\gamma''_\nu\}_{\nu\geq 1}$ of 
$\{\int f^2 + \sum_{j=1}^p\int(\partial f/\partial t_j)^2\}/\|f\|^2_\HH$
satisfy
 \begin{equation}
 \label{eqn:gamma''vasymp}
 \gamma''_\nu \asymp(1+\gamma_\nu)^{-1}.
 \end{equation} 
Thus,  we only need to study  $\{\gamma_\nu''\}_{\nu\geq 1}$. Since $f\in\HH$ has the tensor product structure, we denote by $\lambda_{\vec{\bnu}}[\{\int f^2 + \sum_{j=1}^p\int(\partial f/\partial t_j)^2\}/\langle f,f\rangle_\HH]$ the $\vec{\bnu}$th eigenvalue of the generalized Rayleigh quotient, where $\vec{\bnu}\in V$ and $V$ is defined in (\ref{def:Vbnu}). 

Second, by the assumption that $\lambda_\nu\asymp \nu^{-2m}$ in (\ref{eqn:mercerdecaykttgamma}),  $\HH_1$ is equivalent to a Sobolev space $\WW_2^m(\XX_1)$ and
the trigonometric functions $\{\psi_{\nu}\}_{\nu\geq 1}$ in (\ref{eqn:trigbasisexp}) form an eigenfunction basis of $\HH_1$ up to a $m$-dimensional linear space of polynomials of order less than $m$. See, for example, Wahba \cite{wahba1990}. Denote the latter linear space of polynomials by $\GG$. 
Denote by $\FF_\mu$ and $\FF_\mu^\perp$ the linear spaces spanned by $\{\psi_\nu:1\leq \nu\leq \mu\}$ and $\{\psi_\nu:\nu\geq \mu+1\}$, respectively. For any ${\vec{\bnu}}=(\nu_1,\nu_2,\ldots,\nu_d)\in V$,  by the Courant-Fischer-Weyl min-max principle,
\begin{align}
& \lambda_{(\nu_1-m)\vee 0, (\nu_2-m)\vee 0,\ldots,(\nu_d-m)\vee 0}\left.\left[\left\{\int f^2 + \sum_{j=1}^p\int\left(\frac{\partial f}{\partial t_j}\right)^2\right\}\right/\langle f,f\rangle_\HH\right] \nonumber\\
& \quad \geq \underset{f\in\HH\cap\otimes_{k=1}^d\{\FF_{\nu_k}\cap \GG^\perp\}}{\min}\left.\left[\left\{\int f^2 + \sum_{j=1}^p\int\left(\frac{\partial f}{\partial t_j}\right)^2\right\}\right/\langle f,f\rangle_\HH\right] \nonumber\\
& \quad \geq  c_1 \left(1+\sum_{j=1}^p\nu_j^2\right)\prod_{k=1}^d\nu_k^{-2m} \nonumber
\end{align}
for some constant $c_1>0$, where the last inequality is by the fact that $d\psi_{2\nu-1}(t)/dt = 2\pi\nu\psi_{2\nu}(t)$ and $d\psi_{2\nu}(t)/dt = -2\pi\nu\psi_{2\nu-1}(t)$.
On the other hand,
\begin{align}
& \lambda_{\nu_1+m, \nu_2+m,\ldots,\nu_d+m}\left.\left[\left\{\int f^2 + \sum_{j=1}^p\int\left(\frac{\partial f}{\partial t_j}\right)^2\right\}\right/\langle f,f\rangle_\HH\right]\nonumber\\
& \quad \leq \underset{f\in\HH\cap\otimes^d\{\FF_{k-1}^\perp\cap \GG^\perp\}}{\max}\left.\left[\left\{\int f^2 + \sum_{j=1}^p\int\left(\frac{\partial f}{\partial t_j}\right)^2\right\}\right/\langle f,f\rangle_\HH\right]\nonumber\\
& \quad \leq  c_2 \left(1+\sum_{j=1}^p\nu_j^2\right)\prod_{k=1}^d\nu_k^{-2m}\nonumber
\end{align}
for some constant $c_2>0$. Thus, for any ${\vec{\bnu}}\in V$,
\begin{equation*}
 \lambda_{\vec{\bnu}}\left.\left[\left\{\int f^2 + \sum_{j=1}^p\int\left(\frac{\partial f}{\partial t_j}\right)^2\right\}\right/\langle f,f\rangle_\HH\right] \asymp \left(1+\sum_{j=1}^p\nu_j^2\right)\prod_{k=1}^d\nu_k^{-2m}.
\end{equation*}
This implies $\gamma_\nu' = \gamma_\nu''$, where $\gamma_\nu'$ is defined in Lemma \ref{lem:eigendecayrhonu}. Together with (\ref{eqn:gamma''vasymp}), we complete the proof. 
\end{proof}

\subsubsection{Definition of $N_a(\lambda)$ and its upper bound}
\label{subsubsecdefnalambda}

\begin{lemma}
\label{lemma:varthetamultivariatetensorrank}
Recall that $V$ as a family of multi-index $\vec{\bnu}$ is defined in (\ref{def:Vbnu}). We let
\begin{equation}
\label{eqn:defNalambdanuV}
N_a(\lambda) = \sum_{{\vec{\bnu}}\in V}\frac{\left(\prod_{k=1}^d\nu_k^{2m}\right)^a\left(1+\sum_{j=1}^p\nu_j^2\right)}{\left(1+\sum_{j=1}^p\nu_j^2+\lambda\prod_{k=1}^d\nu_k^{2m}\right)^2}.
\end{equation}
Then, when $0\leq p<d$, we have for any $0\leq a<1-1/2m$,
\begin{equation*}
N_a(\lambda) = O\left\{\lambda^{-a-1/2m}\left[\log (1/\lambda)\right]^{(d-p)\wedge r-1}\right\},
\end{equation*}
and when $p=d$, we have for any $0\leq a\leq 1$,
\begin{align*}
N_a(\lambda)=
\begin{cases}
O\left\{\lambda^{\frac{mr}{1-mr}\left(a+\frac{r-2}{2mr}\right)}\right\},   \mbox{ if } r\geq 3;\\
O\left\{\log(1/\lambda)\right\},  \mbox{ if } r=2, a=0;\quad O\left\{1\right\}, \mbox{ if } r=2, 0<a\leq 1;\\
O\left\{1\right\},  \mbox{ if } r=1,a<\frac{1}{2m};  \quad O\left\{\log(1/\lambda)\right\},  \mbox{ if } r=1,a=\frac{1}{2m}; \\
 O\left\{\lambda^{\frac{1-2ma}{2m-2}}\right\},  \mbox{ if } r=1, a>\frac{1}{2m}.
\end{cases}
\end{align*}
\end{lemma}

\begin{proof}
We will discuss three separate cases for $0\leq p\leq d-r$, $d-r<p<d$ and $p=d$.

First, consider  $0\leq p\leq d-r$. Since ${\vec{\bnu}}\in V$, there are at most $r$ of $\nu_1,\ldots,\nu_d$ not equal to 1, which implies that the number of combinations of non-1 indices being summed in (\ref{eqn:defNalambdanuV}) is no greater than $C_d^1+C_d^2+\cdots+C_d^r<\infty$. Due to the appearance of $(1+\sum_{j=1}^p\nu_j^2)$ in the denominator of (\ref{eqn:defNalambdanuV}), the largest terms of the summation (\ref{eqn:defNalambdanuV}) over ${\vec{\bnu}}\in V$ correspond to the combinations of $r$ indices where as few $\nu_1,\ldots,\nu_p$ being summed as possible, which is the indices ${\vec{\bnu}}=(\nu_{k_1},\nu_{k_2},\ldots,\nu_{k_r})^\top\in \N^r$ with $k_1,k_2,\ldots,k_r>p$. Thus, by the integral approximation,
\begin{align}
& N_a(\lambda) \nonumber\\
& \asymp \sum_{\nu_{p+1}=1}^{\infty}\cdots\sum_{\nu_{p+r-1}=1}^{\infty} \sum_{\nu_{p+r}=1}^{\infty} \frac{\prod_{k=p+1}^{p+r}\nu_k^{2ma}}{\left(1+ \lambda \prod_{k=p+1}^{p+r}\nu_k^{2m}\right)^{2}}\nonumber\\
& \asymp \int_1^\infty\int_1^\infty\cdots\int_1^\infty \left(1+\lambda x_{p+1}^{b}\cdots x_{p+r-1}^b x_{p+r}^{b}\right)^{-2}dx_{p+1}\cdots dx_{p+r-1}dx_{p+r}, \nonumber
\end{align}
where $b= 2m/(2ma+1)$.
Let $z_k =x_{p+1}x_{p+2}\cdots x_{k}$ for $k=p+1,\ldots,p+r$. By using the change of variables to replace $(x_{p+1},\ldots,x_{p+r})$ by $(z_{p+1},\ldots,z_{p+r})$ and $z_{p+r}$ by $x=\lambda^{1/b}z_{p+r}$,
\begin{align}
& N_a(\lambda) \nonumber\\
&\asymp \int_1^{\infty}\int_1^{z_{p+r}}\cdots\int_1^{z_{p+2}}\left(1+\lambda z_{p+r}^{b}\right)^{-2}z_{p+1}^{-1}\cdots z_{p+r-1}^{-1}dz_{p+1}\cdots dz_{p+r-1}dz_{p+r}\nonumber\\
& \asymp \int_1^{\infty}(1+\lambda z_{p+r}^{b})^{-2}(\log z_{p+r})^{r-1}dz_{p+r}\nonumber\\
& \asymp \lambda^{-1/b} \int^\infty_{\lambda^{1/b}} (1+ x^{b})^{-2}\left(\log x - b^{-1}\log \lambda\right)^{r-1}dx\nonumber\\
& \asymp  \lambda^{-a-1/2m}\left[\log (1/\lambda)\right]^{r-1},\nonumber
\end{align}
where the last step follows from the fact that $2b>1$ for any $0\leq a <(2m-1)/(2m)$.

Second, we consider $d-r< p<d$. As discussed in the previous case, the number of combinations of non-1 indices being summed is finite, and the largest terms of the summation (\ref{eqn:defNalambdanuV}) over ${\vec{\bnu}}\in V$ correspond to the indices ${\vec{\bnu}} = (\nu_{k_1},\ldots,\nu_{k_{r+p-d}},\nu_{p+1},\ldots,\nu_d)^\top\in \N^r$, where the indices $k_1,\ldots,k_{r+p-d}\leq p$. Thus, by the integral approximation,
\begin{align}
& N_a(\lambda)\nonumber\\
 & \asymp \sum_{v_{d-r+1}=1}^{\infty}\cdots\sum_{v_d=1}^{\infty}\frac{\prod_{k=d-r+1}^d\nu_k^{2ma}\left(1+\sum_{k=d-r+1}^p\nu_k^2\right)}{\left(1+\sum_{k=d-r+1}^p\nu_k^2+  \lambda \prod_{k=d-r+1}^d\nu_k^{2m}\right)^2}\nonumber\\
&  \asymp\int_1^{\infty}\cdots\int_1^{\infty} \frac{1+x_{d-r+1}^{b/m}+\cdots+x_p^{b/m}}{\left(1+x^{b/m}_{d-r+1}+\cdots+x^{b/m}_p+\lambda x_{d-r+1}^{b}\cdots x_d^{b}\right)^{2}}dx_{d-r+1}\cdots dx_d,\nonumber
\end{align}
where $b=2m/(2ma+1)$. Set $z_k=x_{p+1}x_{p+2}\cdots x_k$ for $k=p+1,\ldots,d$. By using the change the variables to replace $(x_{p+1},\ldots,x_d)$ by $(z_{p+1},\ldots,z_d)$, and $z_d$ by $x=\lambda^{1/b}z_d$, and   $x$ by $u= x_{d-r+1}\cdots x_p \cdot x$.
We have
\begin{align}
& N_a(\lambda) \asymp \int_1^{\infty}\cdots\int_1^{\infty}\left[\int_1^{\infty}\int_1^{z_d}\cdots\int_1^{z_{p+2}}\right.\nonumber\\
&\quad\quad\quad\quad\quad\quad\quad\quad\quad x_{d-r+1}^{b/m}\left(1+x_{d-r+1}^{b/m}+\cdots x_p^{b/m}+\lambda x_{d-r+1}^{b}\cdots x_p^{b}z_d^{b}\right)^{-2}\nonumber\\
&\quad\quad\quad\quad\quad\quad\quad\quad\quad\quad\quad \left.\cdot z_{p+1}^{-1}\cdots z_{d-1}^{-1}dz_{p+1}\cdots dz_{d-1}dz_d\vphantom{\int_1^{\infty}}\right]dx_{d-r+1}\cdots dx_p\nonumber\\
& \asymp \lambda^{-1/b} \int_1^{\infty}\cdots\int_1^{\infty}\left[\int^\infty_{\lambda^{1/b}} \right.\nonumber\\
&\quad\quad\quad\quad\quad\quad\quad\quad\quad x_{d-r+1}^{b/m}(1+x^{b/m}_{d-r+1}+\cdots x_p^{b/m}+x_{d-r+1}^{b}\cdots x_p^{b} x^{b})^{-2}\nonumber\\
&\quad\quad\quad\quad\quad\quad\quad\quad\quad\quad\quad\quad 
\left.\cdot\left(\log x -b^{-1}\log \lambda\right)^{d-p-1}dx\vphantom{\int_1^{\infty}}\right]dx_{d-r+1}\cdots dx_p  \nonumber\\
& \lesssim\lambda^{-1/b} \int^\infty_{\lambda^{1/b}}\left[ \int_1^{\infty}\cdots\int_1^{\infty}\right. \nonumber\\
&\quad\quad\quad\quad\quad\quad\quad
x_{d-r+1}^{b/m}\left(1+x_{d-r+1}^{b/m}+\cdots+x_p^{b/m}+ u^{b}\right)^{-2} x_{d-r+1}^{-1}\cdots x_p^{-1}\nonumber\\
& \quad \cdot\left.\left(\log u - \log x_{d-r+1}-\cdots -\log x_p  -b^{-1}\log\lambda\right)^{d-p-1}dx_{d-r+1}\cdots dx_p\vphantom{\int_1^{\infty}}\right]du.\nonumber
\end{align}
By  Lemma \ref{lem:aversioofyoungsineq}, then for any $0<\tau<1$,
\begin{equation*}
\begin{aligned}
& \left(1+x_{d-r+1}^{b/m}+x_{d-r+2}^{b/m}+\cdots+x_p^{b/m}+ u^{b}\right)^{-2} \\
& \quad\quad \lesssim\left(1+x_{d-r+2}^{b/m}+\cdots+x_p^{b/m}+ u^{b}\right)^{-1+\tau} \cdot \left(x_{d-r+1}^{b/m}\right)^{-(1+\tau)}.
\end{aligned}
\end{equation*}
Together with the fact $\int_1^\infty t^{-1-\tau}(\log t)^k dt<\infty$ for any $k<\infty$, we have
\begin{align}
& N_a(\lambda) \lesssim \lambda^{-1/b} \int^\infty_{\lambda^{1/b}} \left[ \int_1^\infty\cdots\int_1^\infty\right.\nonumber\\
&\quad
 \left(1+x_{d-r+2}^{b/m}+\cdots+x_p^{b/m}+ u^{b}\right)^{-1+\tau}x_{d-r+2}^{-1}\cdots x_p^{-1} \nonumber\\
&\quad\cdot\left.\left(\log u - \log x_{d-r+2}-\cdots -\log x_p  -b^{-1}\log\lambda\right)^{d-p-1}dx_{d-r+2}\cdots dx_p\vphantom{\int_1^{\infty}}\vphantom{\int_1^\infty}\right]du.\nonumber
\end{align}
Continuing this procedure gives
\begin{align*}
N_a(\lambda)\lesssim\lambda^{-1/b} \int^\infty_{\lambda^{1/b}} \left(1+ u^{b}\right)^{-(1-\tau)^{p-d+r}} \left(\log u - b^{-1}\log\lambda\right)^{d-p-1}du.
\end{align*}
Since for any $\epsilon>0$ and $d-r<p<d$, we know if $\tau<\epsilon/d$,
\begin{equation*}
(1-\tau)^{p-d+r}\geq 1- \tau(p-d+r)\geq 1-\tau(d-1)>1-\epsilon.
\end{equation*} 
 Hence, for any $0\leq a< (2m-1)/(2m)$, there exists $\tau$ such that $(1-\tau)^{p-d+r}>a+1/(2m)=1/b$.
Therefore, 
\begin{align*}
N_a(\lambda)\lesssim\lambda^{-1/b}\left[\log (1/\lambda)\right]^{d-p-1} = \lambda^{-a-1/2m}\left[\log (1/{\lambda})\right]^{d-p-1}.
\end{align*}

Finally, we consider $p=d$. As argued in the previous two cases, the number of combinations of non-1 indices being summed is finite. Now since $p=d$, by the symmetry of indices, the largest terms of the summation (\ref{eqn:defNalambdanuV}) over ${\vec{\bnu}}\in V$ correspond to any combinations of $r$ non-1 indices, for example, the first $r$ indices.
Thus, by the integral approximation,
\begin{align}
& N_a(\lambda) \nonumber\\
& \asymp   \sum_{\nu_{1}=1}^{\infty}\cdots\sum_{\nu_{r-1}=1}^{\infty}\sum_{\nu_r=1}^{\infty}
\frac{\prod_{k=1}^r\nu_k^{2ma}\left(1+\sum_{k=1}^r\nu_k^2\right)}{\left(1+\sum_{k=1}^r\nu_k^2+\lambda\prod_{k=1}^r\nu_k^{2m}\right)^2}\nonumber\\
& \asymp \int_1^{\infty}\int_1^\infty\cdots\int_1^{\infty} \frac{1+x_1^{b/m}+\cdots+x_{r-1}^{b/m}+x_r^{b/m}}{\left(1 +x_1^{b/m}+\cdots + x_r^{b/m}+\lambda x_1^{b}\cdots x_{r-1}^bx_r^{b}\right)^{2}}\nonumber\\
& \quad\quad\quad\quad \quad\quad\quad\quad \quad\quad\quad\quad \quad\quad\quad\quad \quad\quad\quad\quad \quad\quad dx_1\cdots dx_{r-1}dx_r\nonumber
\end{align}
where $b=2m/(2ma+1)$.
Observe that if $x_1\cdots x_{r-1}x_r\lesssim \lambda^{mr/[b(1-mr)]}$, then
\begin{equation*}
\lambda x_1^{b}\cdots x_{r-1}^{b}x_r^b \lesssim x_1^{b/m}+\cdots + x_{r-1}^{b/m}+x_r^{b/m}.
\end{equation*}
By  Lemma \ref{lemma:intx1xrzxk1z2} with $\beta=0$ and $\alpha=b/m\leq 2$, we have
\begin{equation}
\label{eqn:intx1xrlesslambdamrb-1}
\begin{aligned}
& N_a(\lambda)\asymp\int_{x_1\cdots x_{r-1}x_r\lesssim\lambda^{mr/[b(1-mr)]}} \\
& \quad\quad\quad\quad\quad\quad  \left(1+x_1^{b/m}+\cdots+x_{r-1}^{b/m}+x_r^{b/m}\right)^{-1} dx_1\cdots dx_{r-1}dx_r\\
& \asymp
\begin{cases}
\lambda^{\frac{mr}{1-mr}\left(a+\frac{r-2}{2mr}\right)},   \mbox{ if } r\geq 3;\\
\log(1/\lambda),  \mbox{ if } r=2, a=0;\quad \lambda^{\frac{2ma}{1-2m}}, \mbox{ if } r=2, 0<a\leq 1;\\
1,  \mbox{ if } r=1,a<\frac{1}{2m};  \quad \log(1/\lambda),  \mbox{ if } r=1,a=\frac{1}{2m};\\\lambda^{\frac{1-2ma}{2m-2}},  \mbox{ if } r=1, a>\frac{1}{2m}.
\end{cases}
\end{aligned}
\end{equation}
On the other hand, if $\lambda^{mr/[b(1-mr)]}(x_1\cdots x_{r-1}x_r)^{-1}=o(1)$, without less of generality,
 we assume $x_r=\min\{x_1,\cdots, x_r\}$. Let $z=\lambda^{1/b}x_1\cdots x_{r-1}x_r$.  By changing $x_r$ to $z$, we have
 \begin{equation}
 \label{eqn:intx1xrgtrmrlmabdalog}
\begin{aligned}
& N_a(\lambda) \asymp\int_{\lambda^{mr/[b(1-mr)]}(x_1\cdots x_{r-1}x_r)^{-1}=o(1)}\\
& \quad\quad\left(1 +x_1^{b/m}+\cdots + x_r^{b/m}+\lambda x_1^{b}\cdots x_{r-1}^bx_r^{b}\right)^{-1}dx_1\cdots dx_{r-1}dx_r\\
&\lesssim \lambda^{-1/b}\int_{\lambda^{1/[b(1-mr)]}z^{-1}=o(1), \lambda^{-(r-1)/(br)}z^{(r-1)/r}\leq x_1\cdots x_{r-1}\leq \lambda^{-1/b}z}\\
& \quad\quad\left(1 +x_1^{b/m}+\cdots + x_{r-1}^{b/m}+z^{b}\right)^{-1} x_1^{-1}\cdots x_{r-1}^{-1}dx_{1}\cdots dx_{r-1}dz \\
&\lesssim \lambda^{-1/b}\int_{\lambda^{1/[b(1-mr)]}z^{-1}=o(1)}\left[\int_{ \lambda^{-(r-1)/(br)}z^{(r-1)/r}\leq x_1\cdots x_{r-1}\leq \lambda^{-1/b}z} \right.\\
& \quad\quad\left.\left(x_1^{b/m}+\cdots + x_{r-1}^{b/m}\right)^{-\tau}x_1^{-1}\cdots x_{r-1}^{-1}dx_{1}\cdots dx_{r-1} \vphantom{\int_1^\infty}
\right] z^{b(-1+\tau)}dz\\
&\lesssim \lambda^{-1/b}\int_{\lambda^{1/[b(1-mr)]}z^{-1}=o(1)} \lambda^{\tau/(mr)}z^{-\tau b/(mr)}\cdot z^{b(-1+\tau)}dz\\
& = o\left[\lambda^{\frac{mr}{1-mr}\left(a+\frac{r-2}{2mr}\right)}\right],
\end{aligned}
\end{equation}
where the third step follows from the Lemma \ref{lem:x1xrgtrxi} in Section \ref{subsec:auxiliarytechlem} for $\beta=-1$ and $\alpha=\tau b/m$.
Combining (\ref{eqn:intx1xrlesslambdamrb-1}) and (\ref{eqn:intx1xrgtrmrlmabdalog}), we complete the proof for $p=d$ and this lemma.
\end{proof}

\subsubsection{Definition of $M_a(\lambda)$ and its upper bound}
\label{subsubsec:defofmalambda}
\begin{lemma}
\label{lem:defMalambdaandbounds}
Recall that $V$ as a family of multi-index $\vec{\bnu}$ is defined in (\ref{def:Vbnu}). We let
\begin{equation*}
M_a(\lambda) = \sum_{{\vec{\bnu}}\in V} \frac{\left(\prod_{k=1}^d\nu_k^{2m}\right)^a}{\left[1+ \lambda\prod_{k=1}^d\nu_k^{2m}(1+\sum_{j=1}^p\nu_j^2)^{-1}\right]^{2}}.
\end{equation*}
When $m>5/(4-2a)$, we have for any $1\leq p\leq d$ and $0\leq a\leq 1$,
\begin{equation*}
M_a(\lambda) = O\left\{\lambda^{-(2ma+1)/(2m-2)}\right\}.
\end{equation*}
\end{lemma}

\begin{proof}
We first show for any $1\leq s\leq r$, 
\begin{equation}
\label{eqn:sumnumalambdlem}
\begin{aligned}
& \sum_{\nu_1=1}^\infty\cdots\sum_{\nu_r=1}^\infty \frac{\prod_{k=1}^r\nu_k^{2ma}}{\left[1+ \lambda\prod_{k=1}^r\nu_k^{2m}(1+\sum_{j=1}^s\nu_j^2)^{-1}\right]^{2}}\\
& \quad\quad\quad\quad\quad\quad\quad \asymp \sum_{\nu_1=1}^\infty\cdots\sum_{\nu_r=1}^\infty \frac{\prod_{k=1}^r\nu_k^{2ma}}{\left[1+ \lambda\prod_{k=1}^r\nu_k^{2m}(1+\nu_s^2)^{-1}\right]^{2}}.
\end{aligned}
\end{equation}
Note that in (\ref{eqn:sumnumalambdlem}), the LHS is greater than the RHS up to some constant. On the contrary, observe that
\begin{align}
&  \sum_{\nu_1=1}^\infty\cdots\sum_{\nu_r=1}^\infty \frac{\prod_{k=1}^r\nu_k^{2ma}}{\left[1+ \lambda\prod_{k=1}^r\nu_k^{2m}(1+\sum_{j=1}^s\nu_j^2)^{-1}\right]^{2}}\nonumber\\
 & \quad\quad\quad\quad\asymp \sum_{\nu_1=1}^\infty\cdots\sum_{\nu_r=1}^\infty \sum_{i=1}^s \frac{(1+\nu_i^2)^2\prod_{k=1}^r\nu_k^{2ma}}{\left(1+\sum_{j=1}^s\nu_j^2+ \lambda\prod_{k=1}^r\nu_k^{2m}\right)^{2}}\nonumber\\
 & \quad\quad\quad\quad\asymp  \sum_{\nu_1=1}^\infty\cdots\sum_{\nu_r=1}^\infty \frac{(1+\nu_s^2)^2\prod_{k=1}^r\nu_k^{2ma}}{\left(1+\sum_{j=1}^s\nu_j^2+ \lambda\prod_{k=1}^r\nu_k^{2m}\right)^{2}}\nonumber\\
 & \quad\quad\quad\quad\leq \sum_{\nu_1=1}^\infty\cdots\sum_{\nu_r=1}^\infty \frac{\prod_{k=1}^r\nu_k^{2ma}}{\left[1+ \lambda\prod_{k=1}^r\nu_k^{2m}(1+\nu_s^2)^{-1}\right]^{2}}.\nonumber
\end{align}
This proves (\ref{eqn:sumnumalambdlem}).
Moreover,  note that
\begin{equation}
\label{eqn:sum1nuj2m1j1equivprod=}
\begin{aligned}
& \sum_{\nu_1=1}^\infty\cdots\sum_{\nu_r=1}^\infty \frac{\prod_{k=1}^r\nu_k^{2ma}}{\left[1+ \lambda\prod_{k=1}^r\nu_k^{2m}(1+\nu_s^2)^{-1}\right]^{2}} \\
& \quad\quad\quad\quad\quad\quad\quad\quad \geq \sum_{\nu_1=1}^\infty\cdots\sum_{\nu_r=1}^\infty \frac{\prod_{k=1}^r\nu_k^{2ma}}{\left(1+ \lambda\prod_{k=1}^r\nu_k^{2m}\right)^{2}}.
\end{aligned}
\end{equation}
Now return to the proof of the lemma.
Since ${\vec{\bnu}}\in V$ and $1\leq p\leq d$, by  (\ref{eqn:sumnumalambdlem}),  (\ref{eqn:sum1nuj2m1j1equivprod=}) and the integral approximation, we have
\begin{align}
& M_a(\lambda)  \asymp \sum_{\nu_1=1}^\infty\cdots\sum_{\nu_r=1}^\infty  \frac{\prod_{k=1}^r\nu_k^{2ma}}{\left[1+ \lambda\prod_{k=1}^r\nu_k^{2m}(1+\nu_r^2)^{-1}\right]^{2}}\nonumber\\
& \asymp \int_1^\infty\int_1^\infty\cdots\int_1^\infty\left[1+\lambda x_1^{b}\cdots x_{r-1}^{b}x_r^{b(m-1)/m}\right]^{-2}dx_1\cdots dx_{r-1}dx_r, \nonumber
\end{align}
where  $b=2m/(2ma+1)$. Let $z= \lambda^{m/[b(m-1)]}x_1^{m/(m-1)}\cdots x_{r-1}^{m/(m-1)}x_r$ and change $x_r$ to $z$. Then,
\begin{align*}
& M_a(\lambda)\\
&\asymp \lambda^{-m/[b(m-1)]}\int_{\lambda^{-m/[b(m-1)]}}^\infty\int_1^\infty\cdots\int_1^\infty \\
& \quad\quad\quad\quad\quad\left[1+z^{b(m-1)/m}\right]^{-2}x_1^{-m/(m-1)}\cdots x_{d-1}^{-m/(m-1)}dx_1\cdots dx_{d-1}dz\nonumber\\
& \asymp \lambda^{-m/[b(m-1)]}\int_{\lambda^{-m/[b(m-1)]}}^\infty \left[1+z^{b(m-1)/m}\right]^{-2}dz,\\
& \leq \lambda^{-m/[b(m-1)]}\int_0^\infty\left[1+z^{b(m-1)/m}\right]^{-2}dz \nonumber\\
& = O\left\{\lambda^{-(2ma+1)/(2m-2)}\right\},\nonumber
\end{align*}
where the second step is because $m/(m-1)>1$ and the last step holds for any $m>5/(4-2a)$.
\end{proof}

\subsubsection{Boundedness of functions in the RKHS $\HH$}
\begin{lemma} 
\label{lem:bounddejft}
For any $g\in\HH$, there exists a constant $c_K$ which is independent of $g$ such that
\begin{equation*}
\sup_{\bt\in\XX_1^d}|g(\bt)|\leq c^d_K\|g\|_{\HH},
\end{equation*}
and
\begin{equation*}
\sup_{\bt\in\XX_1^d}|\partial g/\partial t_j(\bt)|\leq c^d_K\|g\|_{\HH}, \quad \forall 1\leq j\leq d.
\end{equation*} 
\end{lemma}
\begin{proof}
Since we assume that $K$ is continuous in the compact domain $\XX_1$ and satisfies (\ref{eqn:partial2tt'kcx1x1}), there exists some constant $c_K$ such that 
\begin{equation*}
\sup_{t\in\XX_1} |K(t,t)|\leq c_K \quad \mbox{ and }\quad\sup_{t\in\XX_1} \left\vert\frac{\partial^2K(t,t)}{\partial t\partial t'}\right\vert\leq c_K.
\end{equation*}
This implies for any $\bt\in\XX_1^d$,
\begin{equation*}
\left\|\frac{\partial K_d(\bt,\cdot)}{\partial t_j}\right\|_\HH^2 = \left\vert\frac{\partial^2 K(t_j,t_j)}{\partial t_j\partial t_j'}\right\vert \prod_{l\neq j}|K(t_l,t_l)| \leq c_K^d.
\end{equation*}
Thus, for any $g\in\HH$, by the Cauchy-Schwarz inequality,
\begin{equation*}
\sup_{\bt\in\XX_1^d} \left\vert\frac{\partial g(\bt)}{\partial t_j}\right\vert  \leq \sup_{\bt\in\XX_1^d}\left\|\frac{\partial K_d(\bt,\cdot)}{\partial t_j}\right\|_\HH\|g\|_\HH \leq c_K^d\|g\|_\HH,\quad\forall 1\leq j\leq d.
\end{equation*}
Similarly, we can show that $\sup_{\bt}|g(\bt)|\leq c_K^d\|g\|_\HH$.

\end{proof}
\subsection{Auxiliary technical lemmas}
\label{subsec:auxiliarytechlem}

\begin{lemma}[A variant of Young's inequality] 
\label{lem:aversioofyoungsineq}
For any $a,b\geq 0$ and $0<\tau<1$, we have
\begin{equation}
\label{eqn:ineqab-2leqtau}
(a+b)^{-2}\leq \frac{(1-\tau)^{1-\tau}(1+\tau)^{1+\tau}}{4}a^{-(1+\tau)} b^{-(1-\tau)}.
\end{equation}
When $\tau$ is small, the coefficient $(1-\tau)^{1-\tau}(1+\tau)^{1+\tau}/4$ is close to $1/4$.
\end{lemma}
\begin{proof}
To prove (\ref{eqn:ineqab-2leqtau}), it is sufficient to show
\begin{equation*}
a+b\geq 2(1-\tau)^{-(1-\tau)/2}(1+\tau)^{-(1+\tau)/2}a^{(1+\tau)/2}b^{(1-\tau)/2}.
\end{equation*}
Letting $p=2/(1+\tau)$, $a'=a^{1/p}$, $b'=[b/(p-1)]^{(p-1)/p}$, the above formula is equivalent to 
\begin{equation*}
\frac{a'}{p}+\frac{\left(b'\right)^{p/(p-1)}}{p/(p-1)}\geq a'b',
\end{equation*}
which holds by Young's inequality. This completes the proof.
\end{proof}

\begin{lemma}[Bounding the norm of product of functions]
\label{lemma:inequanund1rhonua}
For any $f,g\in\otimes^d\HH_1$, $ a>1/2m$, and $1\leq p\leq d$, we have that
\begin{align}
&\sum_{{\vec{\bnu}}\in\N^d}\left(1+\frac{\rho_{\vec{\bnu}}}{\|\phi_{\vec{\bnu}}\|_{L_2}^2}\right)^a\|\phi_{\vec{\bnu}}\|_{L_2}^2\left\langle \frac{\partial f(\bt)}{\partial t_j}\frac{\partial g(\bt)}{\partial t_j}, \phi_{\vec{\bnu}}(\bt)\right\rangle_0^2\nonumber\\
& \quad\quad\lesssim \|f\|^2_{L_2(a+1/m)} \left[\sum_{{\vec{\bnu}}\in\N^d}\left(1+\frac{\rho_{\vec{\bnu}}}{\|\phi_{\vec{\bnu}}\|_{L_2}^2}\right)^a\|\phi_{\vec{\bnu}}\|_{L_2}^2\left\langle\frac{\partial g(\bt)}{\partial t_j}, \phi_{\vec{\bnu}}(\bt)\right\rangle_0^2\right].\nonumber
\end{align}
\end{lemma}

\begin{proof}
Recall that $\{\psi_\nu(t)\}_{\nu\geq 1}$ is the  trigonometrical basis on $L_2(\XX_1)$ and $\phi_{\vec{\bnu}}(\cdot)$ is defined in 
(\ref{eqn:phivecbnudef}). Write $\psi_{\vec{\bnu}}(\bt)=\psi_{\nu_1}(t_1)\psi_{\nu_2}(t_2)\cdots\psi_{\nu_d}(t_d)$. Note that
\begin{equation*}
\sum_{{\vec{\bnu}}\in\N^d}\left(1+\frac{\rho_{\vec{\bnu}}}{\|\phi_{\vec{\bnu}}\|_{L_2}^2}\right)^a\|\phi_{\vec{\bnu}}\|_{L_2}^2\langle f, \phi_{\vec{\bnu}}\rangle_0^2 = \sum_{{\vec{\bnu}}\in\N^d}\left(1+\frac{\rho_{\vec{\bnu}}}{\|\phi_{\vec{\bnu}}\|_{L_2}^2}\right)^a\left(\int_{\XX_1^d}f\psi_{\vec{\bnu}}\right)^2.
\end{equation*}
By Theorem A.2.2 and Corollary A.2.1 in Lin \cite{lin1998tensor}, if $a>1/2m$, then for any $f,g\in\otimes^d\HH_1$,  
\begin{align}
&  \sum_{{\vec{\bnu}}\in\N^d}(1+\rho_{\vec{\bnu}})^a\left(\int_{\XX_1^d}fg \psi_{\vec{\bnu}}\right)^2 \nonumber\\
& \lesssim  \left[\sum_{{\vec{\bnu}}\in\N^d}\left(1+\frac{\rho_{\vec{\bnu}}}{\|\phi_{\vec{\bnu}}\|_{L_2}^2}\right)^a\left(\int_{\XX_1^d}f \psi_{\vec{\bnu}}\right)^2\right]\left[\sum_{{\vec{\bnu}}\in\N^d}\left(1+\frac{\rho_{\vec{\bnu}}}{\|\phi_{\vec{\bnu}}\|_{L_2}^2}\right)^a\left(\int_{\XX_1^d}g \psi_{\vec{\bnu}}\right)^2\right].\nonumber
\end{align}
Thus,
\begin{align}
&\sum_{{\vec{\bnu}}\in\N^d}\left(1+\frac{\rho_{\vec{\bnu}}}{\|\phi_{\vec{\bnu}}\|_{L_2}^2}\right)^a\|\phi_{\vec{\bnu}}\|_{L_2}^2\left\langle \frac{\partial f(\bt)}{\partial t_j}\frac{\partial g(\bt)}{\partial t_j}, \phi_{\vec{\bnu}}(\bt)\right\rangle_0^2\nonumber\\
&  = \sum_{{\vec{\bnu}}\in\N^d}\left(1+\frac{\rho_{\vec{\bnu}}}{\|\phi_{\vec{\bnu}}\|_{L_2}^2}\right)^a\left(\int_{\XX_1^d}\frac{\partial f(\bt)}{\partial t_j}\frac{\partial g(\bt)}{\partial t_j} \psi_{\vec{\bnu}}(\bt)\right)^2 \nonumber\\
& \lesssim  \left[\sum_{{\vec{\bnu}}\in\N^d}\nu_j^2\left(1+\prod_{k=1}^d\nu_k^{2m}\right)^{a}\left(\int_{\XX_1^d}f(\bt)\psi_{\vec{\bnu}}(\bt)\right)^{2}\right]\nonumber\\
&\quad\quad\quad\quad\quad\quad\quad\quad\quad\quad \times\left[\sum_{{\vec{\bnu}}\in\N^d}\left(1+\frac{\rho_{\vec{\bnu}}}{\|\phi_{\vec{\bnu}}\|_{L_2}^2}\right)^a\left(\int_{\XX_1^d}\frac{\partial g(\bt)}{\partial t_j} \psi_{\vec{\bnu}}(\bt)\right)^2\right]\nonumber\\
& \leq \left\{\sum_{{\vec{\bnu}}\in\N^d}\left[1+\prod_{k=1}^d\nu_k^{2m}\right]^{a+\frac{1}{m}}\left(\int_{\XX_1^d}f(\bt)\psi_{\vec{\bnu}}(\bt)\right)^{2}\right\}\nonumber\\
&\quad\quad\quad\quad\quad\quad\quad\quad\quad\quad \times \left[\sum_{{\vec{\bnu}}\in\N^d}\left(1+\frac{\rho_{\vec{\bnu}}}{\|\phi_{\vec{\bnu}}\|_{L_2}^2}\right)^a\left(\int_{\XX_1^d}\frac{\partial g(\bt)}{\partial t_j} \psi_{\vec{\bnu}}(\bt)\right)^2\right]\nonumber\\
& \asymp \|f\|^2_{L_2(a+1/m)} \left[\sum_{{\vec{\bnu}}\in\N^d}\left(1+\frac{\rho_{\vec{\bnu}}}{\|\phi_{\vec{\bnu}}\|_{L_2}^2}\right)^a\left(\int_{\XX_1^d}\frac{\partial g(\bt)}{\partial t_j} \psi_{\vec{\bnu}}(\bt)\right)^2\right].\nonumber
\end{align}
This completes the proof.
\end{proof}

\begin{lemma}[Inverse transformation]
\label{lem:lineartransformuniform}
Assume that design points $\bt^{e_j}$s have known density $\pi^{e_j}$s which are supported on $\XX_1^d$. Then, there exists a linear transformation to data $(\bt^{e_j},Y^{e_j})$ such that transformed design points $\bx^{e_j}$s are independently uniformly distributed on $\XX_1^d$ and the transformed responses $Z^{e_j}$s are the $j$th first-order partial derivative data of some function.
\end{lemma}
\begin{proof}
As remarked after (\ref{eqn:minaxhatfttprob}), the design under our consideration has the following structure:  different types design points can be grouped to some sets, where within the sets different types design points are drawn identically and across the sets the design points are drawn independently. We give the proof for two cases as follows for the illustration. 

First, we consider that function observations and partial derivatives data share a common design, i.e., $\bt_i^{e_{j}}=\bt_i^{e_{k}}$, $\forall 1\leq i\leq n, 0\leq j<k\leq p$. Write $\bt^{e_j} = (t^{e_j}_1,\ldots,t^{e_j}_d)\in\XX_1^d$. We allow covariates of $\bt^{e_j}$ can be correlated, that is the density of $\bt^{e_j}$ is decomposed as:
\begin{equation*}
\pi^{e_j}(t_1,\ldots,t_d)  = \pi^{e_j}_d(t_d)\pi^{e_j}_{d-1}(t_{d-1}|t_d)\cdots \pi^{e_j}_1(t_1|t_d,t_{d-1}, \ldots,t_{2}).
\end{equation*}
Denote by $\Pi_q^{e_j}$  the CDF corresponding to $\pi_q^{e_j}$, $1\leq q\leq d$. Let
\begin{equation*}
x^{e_j}_d=\Pi^{e_j}_d(t^{e_j}_d),  x^{e_j}_{d-1}=\Pi^{e_j}_{d-1}(t^{e_j}_{d-1}|t^{e_j}_d), \ldots, x^{e_j}_1=\Pi^{e_j}_1(t^{e_j}_1|t^{e_j}_d,t^{e_j}_{d-1}\ldots,t^{e_j}_{2}).
\end{equation*} 
Then, $\bx^{e_j} = (x_1^{e_j},x_2^{e_j},\ldots,x_d^{e_j})$ is uniformly distributed on $\XX_1^d$. 
Define that
\begin{equation*}
\begin{aligned}
& h(x_1, x_2,\ldots,x_d) \\
& \quad = f\left(\{\Pi_1^{e_j}\}^{-1}(x_1|x_d,\ldots,x_2),\{\Pi_2^{e_j}\}^{-1}(x_2|x_d,\ldots,x_3),\ldots, \{\Pi_d^{e_j}\}^{-1}(x_d)\right).
\end{aligned}
\end{equation*} 
Thus,
\begin{align}
\frac{\partial h(\bx)}{\partial x_j}  = \sum_{k=1}^j\frac{\partial f(\bt)}{\partial t_k}\cdot\frac{\partial t_k}{\partial x_j}  = \sum_{k=1}^{j-1}\frac{\partial f}{\partial t_k}\cdot\frac{\partial t_k}{\partial x_j} + \frac{\partial f}{\partial t_j}\cdot\frac{1}{\pi^{e_j}_j(t_j|t_d,\ldots,t_{j+1})}.\nonumber
\end{align}
With the design $\bx^{e_j}$ defined, we transform the responses $Y^{e_j}$s to $Z^{e_j}$s  by letting $Z^{e_0} = Y^{e_0}$ and  for any $j=1,\ldots,p$,
\begin{equation*}
Z^{e_j} = \sum_{k=1}^{j-1} Y^{e_k}\frac{\partial t_k^{e_j}(x_d^{e_j}, x_{d-1}^{e_j}\ldots,x_k^{e_j})}{\partial x_j}+ \frac{Y^{e_j}}{\pi^{e_j}_j(t_j^{e_j}|t_d^{e_j},\ldots,t_{j+1}^{e_j})}.
\end{equation*}
Write
\begin{equation*}
\tilde{\sigma}_{j}^2 = \sum_{k=1}^{j-1}\sigma_k^2\left[\frac{\partial t_k^{e_j}}{\partial x_j}(x_{d}^{e_j},x_{d-1}^{e_j},\ldots,x_{k}^{e_j})\right]^2+\frac{\sigma_j^2}{\left[\pi_j^{e_j}(t^{e_j}_{j}|(t^{e_j}_{d}, \ldots,t^{e_j}_{j+1})\right]^2}.
\end{equation*}
Then, it is clear that  $Z^{e_j} = \partial h/\partial x_j(\bx^{e_j}) + \widetilde{\epsilon^{e_j}}$, where the errors $\widetilde{\epsilon^{e_j}}$s are independent centered noises with variance $\tilde{\sigma}_j^2$s.

Second, we consider that not all types of function observations and partial derivatives data  share a common design, i.e., $\exists 0\leq j\neq k\leq p$ and $1\leq i\leq n$ such that $\bt^{e_{j}}_i\neq \bt^{e_{k}}_i$. We require the covariates of each $\bt^{e_j}$ are independent, that is the density of $\bt^{e_j}$ can be decomposed as:
\begin{equation*}
\pi^{e_j}(t_1,\ldots,t_d) = \pi_1^{e_j}(t_1)\pi_2^{e_j}(t_2)\cdots \pi_d^{e_j}(t_d)
\end{equation*} 
Now let 
\begin{equation*}
x_1^{e_j}=\Pi_1^{e_j}(t_1^{e_j}), \quad x_2^{e_j} = \Pi_2^{e_j}(t_2^{e_j}),\quad\ldots,\quad x_d^{e_j} = \Pi_d^{e_j}(t_d^{e_j}).
\end{equation*} 
Then $\bx^{e_j} = (x_1^{e_j}, x_2^{e_j},\ldots, x_d^{e_j})$ is uniformly distributed on $\XX_1^d$. Define the function
\begin{equation*}
h(x_1,\ldots,x_d) = f\left(\{\Pi_1^{e_j}\}^{-1}(x_1),\{\Pi_2^{e_j}\}^{-1}(x_2),\ldots,\{\Pi_d^{e_j}\}^{-1}(x_d)\right).
\end{equation*} 
Thus, we have 
\begin{equation*}
\frac{\partial h(\bx)}{\partial x_j} = \frac{\partial f(\bt)}{\partial t_j}\cdot\frac{\partial t_j(x_j)}{\partial x_j} =  \frac{\partial f(\bt)}{\partial t_j}\cdot\frac{1}{\pi_j^{e_j}(t_j)}.
\end{equation*} 
Correspondingly, the responses $Y^{e_j}$ is transformed to $Z^{e_j}$, $0\leq j\leq p$, by letting $Z^{e_0} =Y^{e_0}$ and $Z^{e_j} = Y^{e_j}/\pi_j^{e_j}(t_j^{e_j})$ for $1\leq j\leq d$, and write the transformed variance $\tilde{\sigma}_{j}^2  = \sigma_j^2/[\pi_j^{e_j}(t_{j}^{e_j})]^2$.
\end{proof}

\begin{lemma} 
\label{lemma:intx1xrzxk1z1}
Suppose that $s\geq 1$, $\beta\geq 0$ and $\beta\neq 1$, and $r\geq 1$. Then
\begin{align}
& \int_{x_1\cdots x_r\cdot z\leq \Xi, x_k\geq 1, z\geq 1} x_1^\beta\cdots x_r^\beta z^\beta(\log z)^s(x_1^2+\cdots +x_r^2)^{-1}dx_1\cdots dx_rdz\nonumber\\
& \quad\quad\quad\quad\quad\quad\quad\quad\quad\quad\quad\quad\quad\asymp \Xi^{\beta+1}(\log \Xi)^s, \quad \mbox{ as } \Xi\rightarrow\infty.\nonumber
\end{align}
\end{lemma}
\begin{proof}
For any $\tau\geq 1$, we have $\{1\leq z\leq \Xi\tau^{-r},1\leq x_k\leq \tau, k=1,\ldots,r\}\subset \{x_1\cdots x_r\cdot z\leq \Xi, z\geq 1, x_k\geq 1, k=1,\ldots,r \}$. Thus, if $\Xi\rightarrow \infty$,
\begin{align}
& \int_{x_1\cdots x_r\cdot z\leq \Xi, x_k\geq 1, z\geq 1} x_1^\beta\cdots x_r^\beta z^\beta(\log z)^s(x_1^2+\cdots +x_r^2)^{-1}dx_1\cdots dx_rdz\nonumber \\
& \geq \int_1^{\Xi\tau^{-r}}\int_1^\tau\cdots\int_1^\tau z^\beta(\log z)^sx_1^{\beta-2}\cdots x_r^{\beta-2}dx_1\cdots dx_rdz\nonumber\\
& \asymp \Xi^{\beta+1}\tau^{-r(\beta+1)}(\log \Xi-r\log \tau)^s\tau^{r(\beta-1)}.\nonumber
\end{align}
Let $\tau\rightarrow 1$, we have $\int_{x_1\cdots x_r\cdot z\leq \Xi, x_k\geq 1, z\geq 1} (\log z)^s(x_1^2+\cdots +x_r^2)^{-1}dx_1\cdots dx_rdz\gtrsim \Xi^{\beta+1}(\log \Xi)^s$.

On the other hand, define $u=x_1\cdots x_r\cdot z$ and change the variable $z$ to $u$. We have that as $\Xi\rightarrow \infty$,
\begin{align}
& \int_{x_1\cdots x_r\cdot z\leq \Xi, x_k\geq 1, z\geq 1} x_1^\beta\cdots x_r^\beta z^\beta(\log z)^s(x_1^2+\cdots +x_r^2)^{-1}dx_1\cdots dx_rdz\nonumber \\
&= \int_{1}^\Xi\int_1^{u}\int_1^{u/x_r}\cdots \int_1^{u/(x_rx_{r-1}\cdots x_{2})}u^\beta(\log u-\log x_r-\cdots -\log x_1)^s\nonumber\\
& \quad\quad\quad\quad \cdot\left(x_1^2+\cdots + x_{r-1}^2 +x_r^2\right)^{-1}x_1^{-1}\cdots x_{r-1}^{-1}x_{r}^{-1}dx_1\cdots dx_{r-1}dx_{r}du\nonumber\\
&\lesssim \int_{1}^\Xi\int_1^{u}\int_1^{u/x_r}\cdots \int_1^{u/(x_rx_{r-1}\cdots x_{2})}u^\beta(\log u-\log x_r-\cdots -\log x_1)^s\nonumber\\
& \quad\quad\quad\quad \cdot x_1^{-1-2/r}\cdots x_{r-1}^{-1-2/r}x_{r}^{-1-2/r}dx_1\cdots dx_{r-1}dx_{r}du\nonumber\\
& \lesssim \int_{1}^\Xi u^\beta(\log u)^s du \asymp \Xi^{\beta+1}(\log \Xi)^s,\nonumber
\end{align}
where the second step is by Lemma \ref{lem:aversioofyoungsineq}. This completes the proof.
\end{proof}

\begin{lemma} 
\label{lemma:intx1xrzxk1z2}
Suppose that $\beta\geq 0$ and $0<\alpha\leq 2$. Then, as $\Xi\rightarrow \infty$,
\begin{align*}
& \int_{x_1\cdots x_r\leq \Xi, x_k\geq 1}\prod_{k=1}^rx_k^\beta (x_1^\alpha+x_2^{\alpha}+\cdots +x_r^\alpha)^{-1}dx_1\cdots dx_r\\
& \asymp 
\begin{cases}
\Xi^{\beta+1-\alpha/r},   \mbox{ if } r\geq 3;\\
\log(\Xi),  \mbox{ if } r=2, \beta=\alpha/2-1;\quad \Xi^{\beta+1-\alpha/2} \mbox{ if } r=2, \beta>\alpha/2-1;\\
1,  \mbox{ if } r=1,\beta<\alpha-1;  \quad \log(\Xi)  \mbox{ if } r=1, \beta=\alpha-1; \\
\Xi^{\beta-\alpha+1}  \mbox{ if } r=1, \beta>\alpha-1.
\end{cases}
\end{align*}
\end{lemma}
\begin{proof}
By the symmetry of covariates, 
\begin{align}
& \int_{x_1\cdots x_r\leq \Xi, x_k\geq 1} \prod_{k=1}^rx_k^\beta(x_1^\alpha+x_2^\alpha+\cdots +x_r^\alpha)^{-1}dx_1\cdots dx_r\nonumber\\
& \asymp \int_{x_1\cdots x_r\leq \Xi, x_1\geq x_2\geq\cdots\geq x_r\geq 1}\prod_{k=1}^rx_k^\beta (x_1^\alpha+x_2^\alpha+\cdots +x_r^\alpha)^{-1}dx_r\cdots dx_1\nonumber\\
& :=\EE.\nonumber
\end{align}

First we prove when $r\geq 3$,  as $\Xi\rightarrow \infty$, we have
\begin{align}
\label{eqn:r3intx1xraa12r}
\EE\lesssim \Xi^{\beta+1-\alpha/r}.
\end{align}
For this, define the set
$\KK=\left\{0\leq k\leq r-2: \left(\frac{\Xi}{x_1\cdots x_{r-k-1}}\right)^{1/(k+1)}\leq x_{r-k-1}\right\}.$
If $\KK$ is not empty, we denote the smallest element in $\KK$ by $k^*$. Then $0\leq k^*\leq r-2$. For any $(x_1,\ldots, x_r)\in\{(x_1,\ldots,x_r):x_1\cdots x_r\leq \Xi, x_1\geq x_2\geq\cdots\geq x_r\geq 1,x_r\leq x_{r-1}\leq \frac{\Xi}{x_1\cdots x_{r-1}}\}$, we have
\begin{equation}
\label{eqn:condxrk1xkstar}
\begin{cases}
1\leq x_{r-k}\leq x_{r-k-1}& \quad\mbox{ for } 0\leq k\leq k^*-1,\\
1\leq x_{r-k^*}\leq \left(\frac{\Xi}{x_1\cdots x_{r-k^*-1}}\right)^{1/(k^*+1)}& \quad \mbox{ for }  k=k^*,\\
x_{r-k}\geq \left(\frac{\Xi}{x_1\cdots x_{r-k-1}}\right)^{1/(k+1)} & \quad\mbox{ for } k^*+1\leq k\leq r-2,\\
x_{1}\geq \Xi^{1/r}& \quad\mbox{ for }  k= r-1.
\end{cases}
\end{equation}
Thus, as $\Xi\rightarrow \infty$, 
\begin{equation}
\label{eqn:a1-2rkknotempty}
\begin{aligned}
\EE& \lesssim \int_{x_1\cdots x_r\leq \Xi, x_1\geq x_2\geq\cdots\geq x_r\geq 1}\\
& \quad \quad\quad\quad \left\{(x_1)^{\beta-\alpha/(r-1)}\cdots (x_{r-k^*-1})^{\beta-\alpha/(r-1)}\right\} x_{r-k^*}^\beta\\
& \quad\quad\quad\quad\quad\quad\cdot \left\{(x_{r-k^*+1})^{\beta-\alpha/(r-1)}\cdots (x_{r})^{\beta-\alpha/(r-1)}\right\}d\bx\\
& \asymp \int_{x_1\cdots x_r\leq \Xi, x_1\geq x_2\geq\cdots\geq x_r\geq 1}\\
& \quad \quad\quad\quad  \left\{(x_1)^{\beta-\alpha/(r-1)}\cdots (x_{r-k^*-1})^{\beta-\alpha/(r-1)}\right\} \\
& \quad\quad\quad \quad\quad\quad
\cdot (x_{r-k^*})^{[\beta+1-\alpha/(r-1)]k^*+\beta} dx_{r-k^*}dx_{r-k^*-1}\cdots dx_1\\
& \asymp \int_{x_1\cdots x_r\leq \Xi, x_1\geq x_2\geq\cdots\geq x_r\geq 1}\\
& \quad \quad\quad\quad\left\{(x_1)^{-1-\alpha/[(r-1)(k^*+1)]}\cdots (x_{r-k^*-1})^{-1-\alpha/[(r-1)(k^*+1)]}\right\}\\
& \quad\quad\quad \quad\quad\quad
\cdot \Xi^{\beta+1-\alpha k^*/[(r-1)(k^*+1)]}  dx_{r-k^*-1}\cdots dx_1\\
& = \Xi^{\beta+1-\alpha/r},
\end{aligned}
\end{equation}
where the first step uses $x_{r-k^*}\geq 1$ and Lemma \ref{lem:aversioofyoungsineq},  the second step uses $x_{r-k}\leq x_{r-k-1}$ for all $k\leq k^*-1$ in (\ref{eqn:condxrk1xkstar}), the third step uses the upper bound on $x_{r-k^*}$ in (\ref{eqn:condxrk1xkstar}), the fourth step uses the lowers bounds on $x_{r-k}$ for all $k^*+1\leq k\leq r-2$ in  (\ref{eqn:condxrk1xkstar}).
If $\KK$ is empty, then for any $(x_1,\ldots, x_r)\in\{(x_1,\ldots,x_r):x_1\cdots x_r\leq \Xi, x_1\geq x_2\geq\cdots\geq x_r\geq 1,x_r\leq x_{r-1}\leq \Xi/(x_1\cdots x_{r-1})\}$, it satisfies 
\begin{align*}
1\leq x_k\leq x_{k-1} \mbox{ for any } 2\leq k\leq r, \quad \mbox{ and }\quad 1\leq x_1\leq \Xi^{1/r}.
\end{align*}
Thus,  as $\Xi\rightarrow \infty$,
\begin{equation}
\label{eqn:a1-2rkkempty}
\begin{aligned}
\EE& = \int_1^{\Xi^{1/r}}\cdots\int_1^{x_{r-2}}\int_1^{x_{r-1}}\\
& \quad\quad\quad\quad\quad\prod_{k=1}^rx_k^\beta(x_1^\alpha+x_2^\alpha+\cdots+x_{r-1}^\alpha+x_r^\alpha)^{-1}dx_rdx_{r-1}\cdots dx_1\\
& \lesssim  \int_1^{\Xi^{1/r}}\cdots\int_1^{x_{r-2}}\int_1^{x_{r-1}}\\
& \quad\quad\quad\quad\quad
x_1^{\beta-\alpha/r}\cdots x_{r-1}^{\beta-\alpha/r}x_r^{\beta-\alpha/r}dx_rdx_{r-1}\cdots  dx_1\asymp \Xi^{\beta+1-\alpha/r}.
\end{aligned}
\end{equation}
Combining (\ref{eqn:a1-2rkknotempty}) and (\ref{eqn:a1-2rkkempty}) completes the proof for (\ref{eqn:r3intx1xraa12r}).

On the other hand, when $r\geq 3$ and as $\Xi\rightarrow \infty$,
\begin{equation}
\label{eqn:r3intx1xraa12rlower}
\begin{aligned}
\EE& \geq \int_1^{\Xi^{1/r}}\cdots\int_1^{x_{r-2}}\int_1^{x_{r-1}}\\
& \quad\quad\quad\quad\quad\prod_{k=1}^rx_k^\beta(x_1^\alpha+\cdots+x_{r-1}^\alpha+x_r^\alpha)^{-1}dx_rdx_{r-1}\cdots dx_1\\
& \geq \int_{1}^{\Xi^{1/r}}\cdots\int_1^{x_{r-2}}\int_{1}^{x_{r-1}}\\
& \quad\quad\quad\quad\quad\prod_{k=1}^{r}x_k^\beta\cdot r^{-1}x_{1}^{-\alpha}dx_rdx_{r-1}\cdots dx_1 \asymp \Xi^{\beta+1-\alpha/r}.
\end{aligned}
\end{equation}
Therefore, combining (\ref{eqn:r3intx1xraa12r}) and (\ref{eqn:r3intx1xraa12rlower}) completes the proof of the lemma for $r\geq 3$. 

Then we consider for $r=2$. For  $0<\alpha\leq 2$,
\begin{align}
\EE& \leq  2\int_{1}^{\sqrt{\Xi}}\int_1^{x_1}x_1^{\beta-\alpha}x_2^\beta dx_2dx_1 + 2\int_{\sqrt{\Xi}}^{\Xi}\int_1^{\Xi/x_1}x_1^{\beta-\alpha}x_2^\beta dx_2dx_1\nonumber\\
& \asymp 
\begin{cases}
\log(\Xi)  \quad\mbox{ when }2\beta+2-\alpha=0\\
\Xi^{\beta+1-\alpha/2} \quad \mbox{ when } 2\beta+2-\alpha>0
\end{cases}\quad \mbox{ as } \Xi\rightarrow \infty.
\label{eqn:r=2equivx1xrleqxix1gt}
\end{align}
On the other hand, we have
\begin{equation}
\label{eqn:r=2r3intx1xraa12}
\begin{aligned}
\EE& \geq \int_1^{\sqrt{\Xi}}\int_1^{x_1}x_1^\beta x_2^\beta(x_1^\alpha+x_2^\alpha)^{-1}dx_2dx_1\\
 & \geq 2^{-1}\int_1^{\sqrt{\Xi}}\int_1^{x_1}x_1^{\beta-2}x_2^\beta dx_2dx_1 \\
 & \asymp 
\begin{cases}
\log(\Xi)  \quad\mbox{ when }2\beta+2-\alpha=0\\
\Xi^m \quad \mbox{ when } 2\beta+2-\alpha>0
\end{cases} \mbox{ as } \Xi\rightarrow \infty.
\end{aligned}
\end{equation}
Combining (\ref{eqn:r=2equivx1xrleqxix1gt}) and (\ref{eqn:r=2r3intx1xraa12}) completes the proof of the lemma for $r=2$. 

Finally, we consider for $r=1$. Note that $\int_{1}^\Xi x_1^\beta x_1^{-\alpha}dx_1 \asymp 1$ when $0\leq \beta<\alpha-1$, and $\int_{1}^\Xi x_1^\beta x_1^{-\alpha}dx_1 \asymp \log(\Xi)$ when $\beta=\alpha-1$, and $\int_{1}^\Xi x_1^\beta x_1^{-\alpha}dx_1 \asymp \Xi^{\beta-\alpha+1}$ when $\beta>\alpha-1$. This complete the proof.
\end{proof}

\begin{lemma}
\label{lem:x1xrgtrxi}
Suppose that $\beta\leq -1$ and $\alpha>0$. Then, as $\Xi\rightarrow \infty$,
\begin{align*}
& \int_{x_1\cdots x_r\geq \Xi, x_k\geq 1}\prod_{k=1}^rx_k^\beta (x_1^\alpha+x_2^\alpha+\cdots +x_r^\alpha)^{-1}dx_1\cdots dx_r\asymp \Xi^{\beta+1-\alpha/r}. 
\end{align*}
\end{lemma}
\begin{proof}
The proof is similar to the proof for Lemma \ref{lemma:intx1xrzxk1z2}. We omit the details here.
\end{proof}

\begin{lemma}
\label{applem:x1m1mx2x12}
Suppose that $m>1$. Then, as $\Xi\rightarrow\infty$,
\begin{equation*}
\int_{x_1^{(m-1)/m}x_2\cdots x_r\leq \Xi,x_k\geq 1}(x_1^2+x_2^2+\cdots+x_r^2)^{-1}x^2_1dx_1\cdots dx_r\asymp \Xi^{m/(m-1)}.
\end{equation*}
\end{lemma}

\begin{proof}
When $r= 1$, the lemma can be verified by direct calculations. In what follows, assume $r\geq 2$. First, we show that LHS of the formula above is larger than the RHS up to some constant. It suffices to consider a subset of $(x_1,x_2,\ldots, x_r)$ which satisfy $x_1^{(m-1)/m}\geq x_2\geq \cdots\geq x_r\geq 1$. 
Let $u_1 = x_1^{(m-1)/m}$, and $u_j = u_1x_2\cdots x_j$ for $2\leq j\leq r$. By changing variables $(x_1,x_2,\ldots, x_r)$ to $(u_1,u_2,\ldots,u_r)$, the LHS in the lemma satisfies
\begin{align}
& \int_{x_1^{(m-1)/m}x_2\cdots x_r\leq \Xi,x_k\geq 1}(x_1^2+x_2^2+\cdots+x_r^2)^{-1}x^2_1dx_1\cdots dx_r\nonumber\\
& \geq \int_{x_1^{(m-1)/m}x_2\cdots x_r\leq \Xi,x_k\geq 1} (rx_1^2)^{-1}x_1^2dx_1\cdots dx_r\nonumber\\
& = r^{-1} \int_1^{\Xi} \int_{u_r^{(r-1)/r}}^{u_r}\cdots \int_{u_2^{1/2}}^{u_2}u_1^{1/(m-1)}u_1^{-1}\cdots u_{r-1}^{-1}du_1\cdots du_{r-1}du_r\nonumber\\
& \asymp \Xi^{m/(m-1)}.\nonumber
\end{align}
Second, we show that RHS of the formula above is larger than the LHS up to some constant. Note  that  $(x_1^2+x_2^2+\cdots +x_r^2)^{-1}x_1^2\leq 1$, so the LHS satisfies
\begin{align}
& \int_{x_1^{(m-1)/m}x_2\cdots x_r\leq \Xi,x_k\geq 1}(x_1^2+x_2^2+\cdots+x_r^2)^{-1}x^2_1dx_1\cdots dx_r\nonumber\\
& \leq \int_{x_1^{(m-1)/m}x_2\cdots x_r\leq \Xi,x_k\geq 1} 1 dx_1\cdots dx_r\nonumber\\
& = r^{-1} \int_1^{\Xi} \int_{u_r^{(r-1)/r}}^{u_r}\cdots \int_{u_2^{1/2}}^{u_2}u_1^{1/(m-1)}u_1^{-1}\cdots u_{r-1}^{-1}du_1\cdots du_{r-1}du_r\nonumber\\
& \asymp \Xi^{m/(m-1)}.\nonumber
\end{align}
This completes the proof.
\end{proof}


\end{document}